%% file: arxiv_main.tex
\documentclass[11pt]{article}
\RequirePackage{amsthm,amsmath,amsfonts,amssymb}
\RequirePackage[utf8]{inputenc}
\RequirePackage{mathtools}
\RequirePackage[dvipsnames]{xcolor}
\colorlet{OliveGreenOrig}{OliveGreen}
\colorlet{OliveGreen}{OliveGreenOrig!80!LimeGreen}
\RequirePackage[authoryear,square]{natbib}
\RequirePackage[margin=1in]{geometry}
\RequirePackage{enumitem,graphicx,rotating,tabularx,bm,array,mdframed}
\RequirePackage{subfig}
\graphicspath{{./}{../}}
\newcolumntype{P}[1]{>{\raggedright\arraybackslash}p{#1}}
\RequirePackage{tikz,tikz-3dplot}
\usetikzlibrary{arrows.meta,calc,positioning,shapes.geometric}
\RequirePackage[colorlinks,hypertexnames=false,citecolor=OliveGreen,urlcolor=blue,linkcolor=BrickRed]{hyperref}

\numberwithin{equation}{section}
\definecolor{bladefill}{HTML}{97C459}
\definecolor{bladeline}{HTML}{3B6D11}
\definecolor{bladeedge}{HTML}{639922}
\definecolor{axiscol}{HTML}{D85A30}
\definecolor{cylline}{HTML}{888780}
\definecolor{cylfill}{HTML}{EDECE6}
\definecolor{ink}{HTML}{0755AE}

\usepackage{float}
\usepackage{graphicx}
\usepackage{url,booktabs,nicefrac,microtype,comment}
\floatstyle{ruled}
\newfloat{algorithm}{htbp}{loa}
\floatname{algorithm}{Algorithm}
\newcounter{pseudoline}
\newlength{\pseudoindent}
\newenvironment{algorithmic}[1][1]{%
  \setcounter{pseudoline}{0}\setlength{\pseudoindent}{0pt}%
  \begin{list}{}{\setlength{\leftmargin}{2em}%
    \setlength{\labelwidth}{1.6em}\setlength{\labelsep}{0.4em}%
    \setlength{\itemsep}{2pt}\setlength{\parsep}{0pt}%
    \setlength{\topsep}{4pt}}%
}{\end{list}}
\newcommand{\STATE}{\refstepcounter{pseudoline}%
  \item[\footnotesize\thepseudoline:]\hspace*{\pseudoindent}\ignorespaces}
\newcommand{\REPEAT}{\STATE\textbf{repeat}\addtolength{\pseudoindent}{1em}}
\newcommand{\UNTIL}[1]{\addtolength{\pseudoindent}{-1em}\STATE\textbf{until}\ #1}
\newcommand{\IF}[1]{\STATE\textbf{if} #1 \textbf{then}\addtolength{\pseudoindent}{1em}}
\newcommand{\ENDIF}{\addtolength{\pseudoindent}{-1em}\STATE\textbf{end if}}
\newcommand{\RETURN}{\STATE\textbf{return}\ \ignorespaces}

\DeclareMathOperator{\Var}{Var}

\newcommand{\Cov}{{\bm{\mathrm{Cov}}}}

\newcommand{\dd}{\,\mathrm d}
\newcommand{\la}{\left\{}
\newcommand{\ra}{\right\}}
\newcommand{\lb}{\left(}
\newcommand{\rb}{\right)}
\newcommand{\mb}{\mathbb}
\newcommand{\ve}{\varepsilon}
\newcommand{\mc}{\mathcal}
\newcommand{\Sp}{\mathbb S}
\newcommand{\br}[1]{\left(#1\right)}
\newcommand{\innerprod}[1]{\left\langle#1\right\rangle}
\newcommand{\bc}[1]{\left\{#1\right\}}
\newcommand{\balpha}{{\bm{\alpha}}}
\newmdenv{boxedlemma}
\def\cf{{\it cf}}

\newcommand{\bbR}{\mathbb R}
\newcommand{\bbE}{\mathbb E}
\newcommand{\sfZ}{{\bm{\mathsf Z}}}
\newcommand{\sfH}{{\bm{\mathsf H}}}
\newcommand{\bA}{{\bm{A}}}
\newcommand{\bB}{{\bm{B}}}
\newcommand{\bC}{{\bm{C}}}

\newcommand{\bG}{{\bm{G}}}
\newcommand{\bH}{{\bm{H}}}
\newcommand{\bI}{{\bm{I}}}
\newcommand{\bK}{{\bm{K}}}
\newcommand{\bL}{{\bm{L}}}
\newcommand{\bM}{{\bm{M}}}

\newcommand{\bQ}{{\bm{Q}}}

\newcommand{\bPi}{{\bm{\Pi}}}

\newcommand{\bcN}{{\mathcal{N}}}

\newcommand{\bJ}{{\bm{J}}}
\newcommand{\bSigma}{{\bm{\Sigma}}}
\newcommand{\bpi}{{\bm{\pi}}}
\newcommand{\bs}[1]{\left[#1\right]}
\newcommand{\norm}[1]{\left\lVert#1\right\rVert}
\newcommand{\op}{\mathrm{op}}
\newcommand{\ddt}{\frac{d}{dt}}
\newcommand{\Angle}{\operatorname{Angle}}
\newcommand{\Dmax}{D_{\max}}

\newcommand{\pgmm}{p_{\mathrm{GMM}}}
\newcommand{\pdata}{p_{\rm{data}}}
\newcommand{\Law}{\text{Law}}
\newcommand{\pr}{{ \mathbb{P}}}
\newcommand{\Rectify}[1]{\texttt{Rectify}\lb#1\rb}
\newcommand{\Rectflow}[1]{\texttt{Rectflow}\lb#1\rb}

\DeclareMathOperator*{\argmin}{arg\,min}
\DeclareMathOperator{\supp}{supp}

\theoremstyle{plain}
\newtheorem{thm}{Theorem}[section]
\newtheorem{lemma}{Lemma}[section]
\newtheorem{prop}{Proposition}[section]

\theoremstyle{definition}

\newtheorem{remark}{Remark}[section]

\newtheorem{defi}{Definition}

\begin{document}
\title{Setting Things Straight: On the Straightness and Non-Straightness of Two-Step Rectified Flow}
\author{%
Tuan Pham\textsuperscript{1,*} \quad 
Saptarshi Roy\textsuperscript{2,*}\\[0.6em]
Alessandro Rinaldo\textsuperscript{1} \quad 
Purnamrita Sarkar\textsuperscript{1}\\[0.75ex]
\small\textsuperscript{1}Department of Statistics and Data Science, University of Texas at Austin\\
\small\textsuperscript{2}Division of Computing and Mathematical Sciences, MBZUAI\\[0.5ex]
\small\textsuperscript{*}~Equal contribution}
\date{}
\maketitle

\begingroup
\renewcommand{\thefootnote}{}
\footnotetext{Corresponding authors: Tuan Pham (\href{mailto:tuan.pham@utexas.edu}{tuan.pham@utexas.edu}) and Saptarshi Roy (\href{mailto:Saptarshi.Roy@mbzuai.ac.ae}{Saptarshi.Roy@mbzuai.ac.ae}).}
\addtocounter{footnote}{-1}
\endgroup

\begin{abstract}
Rectified flow (RF) is an ODE-based generative modeling framework that has achieved remarkable success in practice over the past few years.  By design, RF seeks to produce flows that have a high degree of straightness, in the sense that the resulting ODE trajectories can be obtained with relatively few (ideally, one) sampling steps. 
However, a rigorous analysis of straightness or near-straightness of RF flows remains missing in the literature. In this paper, we study settings in which the coupling generated by two-step RF is provably straight at the population level. We provide a complete characterization in two popular classes of data distributions: (i) smooth, strongly log-concave densities and (ii) Gaussian mixture models. Our main results yield  sharp conditions for the coupling generated by two-step RF to be straight. Non-straight couplings are constructed when these conditions fail, which demonstrates the sharpness of these conditions. 

Using the theory developed on the straightness of two-step RF, we propose a notion of approximate straightness that measures the straightness of the learned couplings and can be estimated from the data. We introduce a procedure for constructing asymptotic confidence intervals for our proposed notion of approximate straightness and propose a data-driven method for selecting the number of reflow rounds in multiple iterations of RF.

\end{abstract}

\medskip
\noindent\textbf{Keywords:} rectified flows; generative modeling; sampling; flow matching.

\tableofcontents
\clearpage

\section{Introduction} \label{sec:intro}



Generative models have gained substantial attention over the last few years due to their impressive performance in drawing approximate samples from complex probability distributions. Early successes in  deep generative modeling include variational autoencoders (VAE) \citep{kingma2014auto} and generative adversarial networks (GAN) \citep{goodfellow2014generative}, which learn data distributions through latent-variable modeling and adversarial training, respectively. Diffusion and score-based models subsequently emerged as powerful approaches that generate samples by learning to reverse a process that gradually transforms data into noise \citep{sohl2015deep,ho2020denoising}. Their continuous-time formulations connect generative modeling to stochastic and ordinary differential equations (ODEs) that transport a simple reference distribution to the target data distribution \citep{song2021denoising}.

In the past few years, Rectified Flow or Flow matching \citep{liu2022rectified,liu2023flow, lipman2022flow} has emerged as a prominent sampling methodology  due to its state-of-the-art generative performance across various modalities including image generation  \citep{liu2023instaflow, yan2024perflow}, audio generation \citep{ audio_RF_wang2024frieren, audio_RF_liu2025flashaudio}, and video generation \citep{video_RF_chen2025goku, video_RF_liu2025improving}. Apart from applications, RF has also attracted several strands of theoretical research including low-dimensional adaptivity \citep{low_boffi2025shallow, roy2026low,  low_kumar2026flow}, sampling performance \citep{bansal2025wasserstein, wang2026computational, guan2026mirror, kumar2026conditional}, alignment \citep{alignment_fan2025online, alignment_potaptchik2026meta}, optimal distillation \citep{roy2026optimal} and many more.
The success of RF stems from two key factors: (1) RF is a deterministic ordinary differential equation (ODE)-based sampler that is easier to simulate compared to other stochastic samplers, and (2) RF learns almost straight trajectories between source and target data distributions via multiple reflow steps, leading to faster sampling. In particular, the straightness or near-straightness of its ODE trajectories is critical: flows with nearly straight trajectories require fewer sampling steps to generate new samples and are thus more efficient. Though this feature of RF is generally recognized as a main reason for its empirical successes and wide applicability, the theoretical understanding of the conditions under which the RF trajectories are guaranteed to be (near) straight or not has remained limited. { In this paper we make a step forward in filling this gap by providing a rigorous framework for understanding the straightness of RFs. The key message is that our framework yields sharp conditions for straightness and identifies precisely when non-straightness can happen. We demonstrate the use of our framework on two common classes of data distributions: smooth, strongly log-concave distributions and Gaussian mixture models (GMMs). }

\paragraph*{Rectified flows}
Let us briefly recall the RF framework. 
In this generative modeling framework, we observe
$\bm X_1,\dots,\bm X_n \sim p_{\rm data}$, where $p_{\rm data}$ is an unknown distribution on $\mb R^d$, and aim to generate new samples approximately from this distribution. The RF framework learns a velocity field ${\bm v}(\bm x,t)$ by minimizing the following objective:
\begin{equation}
\label{eq: drift_obj}
    \bm v = \argmin_{\bm f} \int_0^1 \bbE{\norm{\bm X_1 - \bm X_0 - \bm f\lb \bm X_t, t\rb}_2^2} dt , 
\end{equation}
where 
\[
\bm X_t = t\bm X_1+(1-t)\bm X_0,
\qquad
(\bm X_0, t) \sim \bcN\lb\bm 0,\bm I_d\rb \times \operatorname{Unif}(0,1),
\]
with $(\bm X_0,t)$ independent of the data. 
In the objective \eqref{eq: drift_obj} above,
$\bm X_t$ is formed by interpolating the same pair
$(\bm X_0,\bm X_1)$ that appears in the displacement
$\bm X_1-\bm X_0$. For each $t$, the above objective in (\ref{eq: drift_obj}) is minimized at
\begin{align} \label{eq: drift-true}
 \lb \bm x,t\rb \mapsto   \bm v_t\lb \bm x\rb := \bm v\lb \bm x,t\rb &= \bbE[ \bm X_1- \bm X_0 \mid \bm X_t = \bm x]. 
\end{align}
Under appropriate regularity conditions (investigated in \citet{bansal2025wasserstein, mena2025statistical}) the velocity is differentiable and the ODE
\begin{align} 
    d \bm Z_t = \bm v_t\lb \bm Z_t \rb\ dt, \quad \bm Z_0 \sim p_0 = \bcN\lb \bm 0, \bI_d\rb \label{eq: ode-true}
\end{align}
 has a unique solution and satisfies the marginal preserving property,
 i.e., $\Law\lb \bm Z_t\rb = \Law\lb \bm X_t\rb = p_t$. Under a mild non-explosivity assumption on the ODE and the second moment assumption on $\pdata$, \cite{bansal2025wasserstein} showed that 
 ODE \eqref{eq: ode-true} admits a unique solution. Throughout this paper, we assume that the above ODE has a unique solution, which is also necessary for a coupling to be rectifiable \citep{liu2022rectified,liu2023flow}. 
 
The initial fitting procedure yields a learned velocity field, which we denote
by $\bm v^{(1)}$. More generally, after $k$ rectification rounds we write
$\bm v^{(k)}$ for the learned velocity field at round $k$. To simplify
notation, when discussing the first round we will write $\bm v:=\bm v^{(1)}$.
We refer to the endpoint coupling induced by $\bm v^{(1)}$ as the $1$-RF coupling:
\[
\mbox{1-RF}
:=
\lb
\bm X_0^{(1)},
\sfZ_1^{(1)}\lb\bm X_0^{(1)}\rb
\rb,
\qquad
\bm X_0^{(1)}\sim\bcN\lb\bm 0,\bm I_d\rb,
\]
where the flow map $\sfZ_t^{(1)}$ solves
\[
\frac{d}{dt} \sfZ_t^{(1)}(\bm x)
=
\bm v^{(1)}
\lb \sfZ_t^{(1)}(\bm x),t\rb,
\qquad
\sfZ_0^{(1)}(\bm x)=\bm x.
\]
This coupling is deterministic in the sense that, conditional on the learned velocity field, the terminal state is a deterministic function of the initial state $\bm X_0^{(1)}$. We can iterate this procedure to obtain $K$-RF by performing $K-1$ additional reflow rounds, as illustrated below:
\begin{figure}[t]
\centering
\resizebox{\linewidth}{!}{%
\begin{tikzpicture}[>=latex,baseline]
  \tikzstyle{box}=[draw=black!50,rounded corners,inner sep=2pt,align=center,minimum height=12mm]
  \tikzstyle{data}=[box,fill=blue!6]
  \tikzstyle{vel}=[box,fill=orange!10]
  \tikzstyle{rf}=[box,fill=green!8]
  \tikzstyle{arr}=[->,thick,draw=black!60]

  \node[data] (data) {$(\bm X_0,\bm X_1)$};
  \node[vel,right=8mm of data] (v1) {$\bm v^{(1)}$};
  \node[rf,right=8mm of v1] (rf1) {$1$-RF\\$(\bm X_0^{(1)},\sfZ_1^{(1)}(\bm X_0^{(1)}))$};
  \node[vel,right=8mm of rf1] (v2) {$\bm v^{(2)}$};
  \node[rf,right=8mm of v2] (rf2) {$2$-RF\\$(\bm X_0^{(2)},\sfZ_1^{(2)}(\bm X_0^{(2)}))$};
  \node[right=6mm of rf2] (dots) {$\cdots$};
  \node[vel,right=6mm of dots] (vK) {$\bm v^{(K)}$};
  \node[rf,right=8mm of vK] (rfK) {$K$-RF\\$(\bm X_0^{(K)},\sfZ_1^{(K)}(\bm X_0^{(K)}))$};

  \draw[arr] (data) -- node[midway,above=15pt,font=\scriptsize]{learn} (v1);
  \draw[arr] (v1) -- node[midway,above=15pt,font=\scriptsize]{generate} (rf1);
  \draw[arr] (rf1) -- node[midway,above=15pt,font=\scriptsize]{learn} (v2);
  \draw[arr] (v2) -- node[midway,above=15pt,font=\scriptsize]{generate} (rf2);
  \draw[arr] (rf2) -- (dots);
  \draw[arr] (dots) -- (vK);
  \draw[arr] (vK) -- node[midway,above=15pt,font=\scriptsize]{generate} (rfK);
\end{tikzpicture}%
}
\caption{Schematic of the iterative reflow procedure for constructing $K$-RF.}
\label{fig:reflow-procedure}
\end{figure}
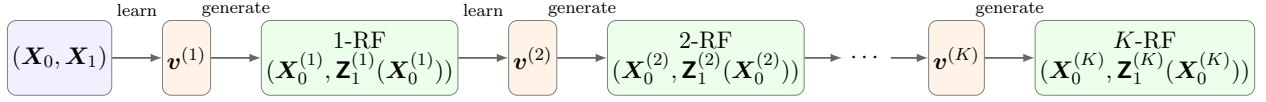
At each subsequent round, a new velocity field is fitted using the
paired initial and terminal states generated by the preceding learned
flow. We refer readers to \cite{liu2023flow} for a comprehensive
description of the procedure and its implementation. Once the procedure stops after the $K$-th rectification, new samples are generated by
\[
\text{drawing
$\bm X_0^{(K)}\sim\bcN\lb\bm 0,\bm I_d\rb$
and computing
$\sfZ_1^{(K)}\lb\bm X_0^{(K)}\rb$}.
\]
This can be done through numerical integration of the ODE associated with the final
learned velocity field $\hat{\bm v}^{(K)}$.
In practice, $2$-RF or $3$-RF can already produce high-quality samples \citep{liu2023flow}.
Interestingly, repeated rectifications seem to increase the straightness; see Theorem D.7 in  \cite{liu2023flow}.


Despite the impressive empirical success of RF, theoretical understanding
and guarantees regarding its straightness remained limited until recently,
when \citet{bansal2025wasserstein} developed a framework for studying RF
straightness. In particular, they showed that $2$-RF produces straight flows under
certain technical assumptions on the target distribution and the velocity
field. A substantial limitation of these results is that they provide
only a partial characterization of the conditions under which RF produces
straight flows. Furthermore, they offer limited practical guidance for
assessing the straightness of trajectories generated by a learned velocity
field, which is an important consideration in determining how many
rectification steps are needed and whether RF is an appropriate generative
modeling framework for a given dataset.

The goal of the present article is to provide a general, simple framework
for understanding the straightness property of RF under some common
classes of data distributions via the concept of ``maximum rotation''
which borrows ideas from analysis. In particular, we will be concerned
with the following questions:

(i) Given a data distribution $p_{\rm data}$, when does RF produce
straight couplings after two rectification steps, as one observes in
practice? If not, can it be non-straight even for very regular data
distributions?

(ii) Based on the data $\bm X_1,\dots,\bm X_n\in\mb R^d$, how many
rectification steps $K\geq1$ are needed so that the learned velocity
fields after $K$ steps give approximately straight couplings that are
good enough to use for sampling?

Our results make both theoretical and methodological contributions,
which can be summarized as follows.
\begin{enumerate}
    \item We develop a general theory for proving the straightness
    of two-step RF by bounding the maximum angular rotation of the
    endpoint map. We apply our theory to two concrete, popular
    classes of distributions: (a) smooth, strongly log-concave
    distributions and (b) Gaussian mixture models with isotropic
    covariance matrices.

Using the maximum rotation framework, geometric conditions
    under which $2$-RF produces exactly straight couplings
    (at the population level) are obtained for these two classes
    of distributions. We find that the condition number of the log-concave distribution and the maximum cluster mean separation in the case of GMM are the crucial quantities for characterizing the straightness of RF. Additionally, we show that our conditions are sharp by
    constructing non-straight flows when these geometric conditions
    fail. This answers question~(i) in full generality. To the best of our knowledge, this is the first construction of non-straight couplings obtained from $2$-RF.

    \item We consider the problem of assessing the straightness of a learned velocity field, which is crucial for determining the number of rectification steps upon observing the data. We propose a notion of “approximate straightness” and construct a consistent estimator to quantify it. Using this estimator, we construct valid confidence intervals and develop a data-driven stopping rule based on a prescribed straightness tolerance.

    {\item On the technical level, the framework developed here is useful for understanding both straightness and non-straightness of the couplings produced by RF. In particular, our results provide the first sharp characterization of RF straightness. Proposition \ref{ns:prop:criterion} below provides a route to proving that a coupling is non-straight, and we expect it to be relevant to other settings beyond those considered in this paper. } 

\end{enumerate}

\textbf{Notation.} Throughout the paper, with some abuse of notation, we will identify  a probability distribution with its Lebesgue density, which we will assume exists unless otherwise noted. For a distribution (Lebesgue density) $p$, write $\supp(p)$ to denote the support of $p$. For a random variable $X$ we use $\Law(X)$ to denote its distribution. For a  function $\bm f=\lb f_1,\dots,f_d \rb: \mb R^d \to \mb R^d$ differentiable at $\bm x$, we denote by $D \bm f(\bm x)$ the Jacobian matrix of $\bm f$ at $\bm x$
\[
D \bm f( \bm x):= \lb \frac{\partial  f_{i}}{\partial x_j} \lb \bm x \rb \rb_{i,j=1}^d.
\]
When there are many variables, we will also write $D_{\bm x}$ to indicate that the Jacobian matrix is taken with respect to $\bm x$. In the case where the function takes real values, we use $\nabla$ to indicate its gradient and $\nabla^2$ to indicate its Hessian matrix.

For a set $\{x_j\}_{j \in[K]}\subset \bbR$ and probability vector $\bpi := (\pi_1, \ldots, \pi_K)$, we denote the variance of the set by 
$$\Var_\bpi(\{x_j\}_{j \in [K]}) = \sum_{j \in [K]} \pi_j x_j^2 - \lb \sum_{j \in [K]}\pi_j x_j \rb^2.$$ 
For two vectors $\bm u, \bm v$, we denote by $\Angle(\bm u, \bm v) \in [0, \pi]$ the angle between $\bm u$ and $\bm v$. Given a matrix $\bM \in \bbR^{d \times d}$, we also define the maximum angle
\begin{align} \label{eq:maximum-angle}
\phi_\bM:= \sup_{\bm u \in \mb R^d: \bM \bm u \ne \bm 0} \Angle(\bm u, \bM \bm u),
\end{align}
which represents the maximum rotation achieved by the matrix $\bM$. For two square matrices $\bm M$ and $\bm N$, we write their commutator as $[\bm M, \bm N] = \bm M \bm N - \bm N \bm M$. Finally, for a symmetric matrix $\bm M$, we identify its maximum and minimum eigenvalues as $\lambda_{\max}(\bm M)$ and $\lambda_{\min}(\bm M)$, respectively.


We use $\mbox{GL}_k(\mb R)$ to denote the group of real  invertible matrices of size $k$, i.e.
\[
\mbox{GL}_k(\mb R):= \la  \bm A \in \mb R^{k \times k}: \mbox{det}(\bm A) \neq 0\ra.
\]
Similarly, the space of positive-definite matrices of size $k$ will be denoted by $\mb S_{++}^k$:
\[
\mb S_{++}^k:= \Big\{  \bm A: \exists \ve>0 \ \text{s.t.} \ \bm A \succeq \ve \cdot  \bm I_k  \Big\}.
\]
Both these spaces are regarded as open subsets of $\mb R^{k^2}$ and $\mb R^{k(k+1)/2}$, respectively, when we talk about differentiability. For example, a map $\bm T$ is said to be in $ C^1 \Big( [0,1], \mbox{GL}_k(\mb R) \Big)$ if $\bm T$ takes values in $\mbox{GL}_k(\mb R)$ and is $C^1$ as a map from $[0,1]$ to $\mb R^{k^2}$. 
\section{Background and Preliminaries}
\label{sec: background}

\subsection{Straight coupling and straight flow}
\label{sec: straight coupling adn straight flow}
Rectified flow interpolates between the two distributions in straight line paths as $\bm X_t = t \bm X_1 + \lb 1-t\rb\bm X_0$ for $t \in [0,1]$, giving $\bm v_t\lb \bm x\rb$ as in \eqref{eq: drift-true}. Therefore, the resulting sampling ODE \eqref{eq: ode-true} also approximates a geodesic path between $p_0$ and $p_1$ that allows faster sampling.
In the context of RFs, the optimization step \eqref{eq: drift_obj} is known as a rectification step, and the solution path 
\begin{align} \label{eq:bm-Z}
   \bm {\mc{Z}}: = \la\bm Z_t\ra_{t \in [0,1]} 
\end{align}
of the ODE \eqref{eq: ode-true} is known as the rectified flow, denoted by $\bm {\mc{Z}} = \Rectflow{\bm X_0, \bm X_1}$. 

It is known that RF is marginal preserving \citep{liu2023flow}, i.e., $\Law\lb  \bm X_t\rb = \Law\lb \bm Z_t\rb$. Hence, RF yields a new \textit{deterministic} coupling $\lb \bm Z_0, \bm Z_1\rb:= \Rectify{\bm X_0, \bm X_1}$ between $p_0$ and $p_1$. In fact, $\lb \bm Z_0,\bm  Z_t\rb$ is also a deterministic coupling between $p_0$ and 
\begin{align} \label{eq:p_t}
   p_t:= \text{density of $\mbox{Law} \lb \bm X_t \rb$},  
\end{align}
 for any $t \in [0,1]$. To underscore this deterministic relation, we denote by $\sfZ_t : \bbR^d \to \bbR^d$ the transport/ flow map defined through the ODE
\begin{equation}
\label{eq: ode non-random}
    \ddt \sfZ_t\lb \bm z_0\rb = \bm v_t\lb \sfZ_t\lb \bm z_0\rb\rb, \quad \sfZ_0\lb \bm z_0\rb = \bm z_0,
\end{equation}
for an initial point $\bm z_0\in \bbR^d$.
It is easy to see that $\bm Z_t = \sfZ_t\lb \bm Z_0\rb$ by comparing \eqref{eq: ode-true} and \eqref{eq: ode non-random}. In particular, we have $\lb \bm Z_0, \bm Z_1\rb = \lb \bm Z_0, \sfZ_1\lb \bm Z_0\rb\rb$ for $\bm Z_0 \sim \bcN\lb \bm 0, \bI_d\rb$. For notational brevity, we also define the chord map
\begin{align} \label{eq:chord-map}
\sfH_t\lb \bm x\rb := \lb 1-t\rb\bm x + t\sfZ_1\lb \bm x\rb, \qquad \bm x\in\bbR^d.
\end{align}
Now, we recall the definitions of straight coupling and straight flow from \cite{liu2023flow}. 
\begin{defi}[Straight coupling and Straight flow]
\label{def: Qiang straightness}
    A coupling $\lb \bm Z_0, \sfZ_1\lb \bm Z_0\rb\rb$ such that $\bbE \| \sfZ_1\lb \bm Z_0\rb - \bm Z_0\|^2< \infty$ is called a \textit{straight coupling} if
    $$\bbE[\sfZ_1\lb \bm Z_0\rb - \bm Z_0 \mid t, \sfH_t\lb \bm Z_0\rb ] = \sfZ_1\lb \bm Z_0\rb - \bm Z_0, $$ 
   almost surely with respect to independent draws of $ \bm Z_0 \sim p_0$ and $ t \sim  \text{Unif}\lb 0, 1\rb$.
  We call the $\Rectflow{\bm Z_0, \sfZ_1\lb \bm Z_0\rb}$ generated by rectification of a straight coupling $\lb \bm Z_0, \sfZ_1\lb \bm Z_0\rb\rb$ a \emph{straight flow}.
\end{defi} 

It is worth pointing out why rectifying a straight coupling yields
constant velocity along each trajectory. Write
$\bm D=\sfZ_1(\bm Z_0)-\bm Z_0$ and
$\sfH_t(\bm Z_0)=\bm Z_0+t\bm D$.
By straightness, the 2-RF velocity satisfies
\[
    \bm v_t^{(2)}(\sfH_t(\bm Z_0))
    =\bbE[\bm D\mid t,\sfH_t(\bm Z_0)]
    =\bm D
    =\frac{d}{dt}\sfH_t(\bm Z_0)
\]
almost surely in $(\bm Z_0,t)$. Thus, the linear interpolation
solves the 2-RF ODE and, under uniqueness, coincides with
its trajectory. Consequently, 2-RF has constant velocity along
almost every trajectory, for almost every time. This does not
assert that the original 1-RF trajectories are straight, nor
that $\bm v_t^{(2)}(\bm x)$ is independent of $t$ at a fixed $\bm x$.
Following this, we state a stronger result that is also true:

\begin{lemma}[Theorem D.6, \cite{liu2023flow}]
    \label{lemma: straight coupling and flow}
    A coupling $(\bm X_0, \bm X_1)$ is a straight coupling iff $\Rectflow{\bm X_0, \bm X_1}$ is a straight flow.
\end{lemma}
The above result is a subtle but crucial result that we use in proving the straightness of 2-RF. It is essentially an adaptation of Theorem D.6 of \cite{liu2023flow}, and we add the proof in Appendix~\ref{sec: proof of straight flow and coupling} to make the paper self-contained.

Throughout the paper, by 1-RF and 2-RF we mean the application of $\Rectflow{\cdot, \cdot}$ once and twice, respectively, starting from an independent coupling $(\bm X_0,\bm X_1)$. Then, provided that 1-RF produces a straight coupling, 2-RF will necessarily achieve {\it straightness} due to Lemma \ref{lemma: straight coupling and flow}, i.e. it will yield a straight flow, so that samples from a target distribution of $\bm X_1$ can be generated with a single Euler step. 

Indeed, the main appeal of the RF machinery, in addition to its inherent conceptual simplicity among ODE-based sampling methods, is the fact that by design it seeks to produce trajectories with minimal curvature. Methodologically, this is a critical feature, as it allows for dramatically faster sampling. However, from a theoretical standpoint, very little is known in the literature about whether or when RF trajectories are provably straight or even approximately straight.  
To the best of our knowledge, sufficient conditions for straightness are known only for the case of Gaussian-to-Gaussian \citep{ mena2025statistical, bansal2025wasserstein, hertrich2026relation} and Gaussian-to-isotropic-Gaussian-mixtures, under appropriate (in fact, as we will see, sub-optimal) separation conditions \citep{bansal2025wasserstein}. In this paper, we  partially fill this gap by fully characterizing the straightness of RF for two classical tasks: sampling from  log-concave distributions and from (isotropic) Gaussian mixtures. 

In their seminal paper, \citet{liu2023flow} showed that successive rectifications eventually lead to a near straight flow that allows faster sampling in the generation stage. Subsequently, \citet{bansal2025wasserstein}  formulated general conditions under which  2-RF leads to straight flows. 
Their main finding is that global injectivity of \;$\sfH_t\lb \cdot\rb$ leads to straightness of 1-RF coupling.
They also showed that the positive-definiteness of the \textit{symmetric part} of the Jacobian $\bJ_t\lb \bm z_0\rb : = D_{\bm z_0} \sfZ_t\lb \bm z_0\rb$ for all $\bm z_0 \in \bbR^d$ is sufficient to ensure global invertibility of $\sfH_t\lb \cdot\rb$. In detail, the positive-definiteness condition\footnote[1]{$\bJ_1\lb \bm z_0\rb$ is not necessarily symmetric.} 

    \begin{equation}
    \label{eq: J positive definite}
    \bm u^\top \bJ_1\lb \bm z_0\rb \bm u > 0; \quad \forall \bm z_0 \in \bbR^d\; \text{and}\; \bm u \in \Sp^{d-1} \Longleftrightarrow \phi_{\bJ_1\lb \bm z_0\rb} < \pi/2; \quad \forall \bm z_0 \in \bbR^d
\end{equation}

\noindent ensures that $\sfH_t\lb \cdot\rb$ is \emph{proper}, i.e., $\lim_{\norm{\bm z_0}_2 \to \infty} \norm{\sfH_t\lb \bm z_0\rb}_2 = \infty$, and $\det\lb D \sfH_t\lb \bm z_0\rb\rb>0$, which yield global invertibility of $\sfH_t\lb \cdot\rb$ (\cf.~\cite{ plastock1974homeomorphisms,ruzhansky2015global}) . Furthermore, \cite{bansal2025wasserstein} noted that the Jacobian $\bJ_t(\bm z_0)$ satisfies the ODE
\begin{equation}
    \label{eq: ode J_t}
    \ddt \bJ_t\lb \bm z_0\rb =   D \bm v_t\lb \sfZ_t\lb \bm z_0\rb\rb \bJ_t\lb \bm z_0\rb , \quad \bJ_0\lb \bm z_0\rb = \bI_d,
\end{equation}
and used it to examine the properties of the Jacobian flow across different target models.
In contrast, in this paper we investigate the injectivity of $\sfH_t\lb \cdot\rb$ directly, without relying on the ODE \eqref{eq: ode J_t}.

\paragraph*{Related work} Our results complement several recent works on sufficient and necessary
conditions for straightness. Although \cite{bansal2025wasserstein}
established straightness of 1-RF under Condition~\eqref{eq: J positive definite}, they observed
experimentally that this condition can fail for low-variance Gaussian mixtures.
This failure was subsequently established rigorously by
\cite{nonstraight_pierret2026tessellations} using a tessellation argument.
Importantly, however, failure of the Jacobian condition alone does not imply that
the resulting RF coupling is non-straight. Related positive results were obtained
by \cite{nonstraight_chambolle2026limit}, who introduced \textit{weak
rectification} and established straightness in a limiting sense under a
generalization of Condition~\eqref{eq: J positive definite}, and by
\cite{nonstraight_tsimpos2026one} under similar structural assumptions.

On the negative side, \cite{hertrich2026relation} and
\cite{nonstraight_tsimpos2026one} showed that rectification can produce
non-straight couplings when the underlying distributions have compact and
disconnected supports. Such examples, however, are qualitatively different from
the standard generative-modeling setting, where the source distribution is
typically Gaussian and therefore has connected and unbounded support. Additionally, \cite{mena2025statistical} showed that compact support may not even warrant existence of RF.  Our results
show that compact or disconnected support is not necessary for non-straightness. In
Theorems~\ref{ns:thm:main} and \ref{ns:thm:gmm}, we construct examples in which
the source is Gaussian and the target is absolutely continuous (log-concave or a mixture of Gaussians), yet the coupling
obtained after one rectification is non-straight. As a key ingredient, we develop
a general criterion for certifying non-straightness
(Proposition~\ref{ns:prop:criterion}), which may be of independent interest.

\section{Maximum rotation and straightness}
\label{sec: a sufficient condition for straightness}
In this section, we present our proposed criterion for straightness, which is based on the angle between the difference of two initial points and the difference of their images under the endpoint map; see Theorem \ref{theorem: straightness theorem} for the precise result. We apply this criterion in Subsections \ref{sec:straightness-log-concave} and \ref{sec:straightness-gmm} to derive sufficient conditions for straightness for some popular types of data distributions. Sharpness of this condition will be studied in Section \ref{sec:nonstraight-construction}.

\subsection{Maximum rotation} \label{sec: tracking maximum angle}
We begin by ensuring that distinct initial points remain distinct
under the ODE flow. Throughout this section, we assume that the
flow in \eqref{eq: ode non-random} is defined over $[0,1]$ for any
initial point and that its Borel measurable velocity field satisfies
\begin{equation}
    \label{eq: local regularity velocity}
    \int_0^1 \left(
        \|\bm v_t(\bm 0)\|_2
        + \sup_{ \Vert \bm z\Vert_2\le R}\|D\bm v_t(\bm z)\|_{\mathrm{op}}
    \right)\,dt < \infty,
    \qquad \text{for every } R>0.
\end{equation}
By applying uniqueness backward in time \citep[Lemma 8.1.4]{ambrosio2005gradient},
the endpoint map $\sfZ_1$ is injective. For completeness, we add a relevant discussion in Appendix~\ref{sec: proof of injectivity of terminal flow map}.
Thus, $\sfH_0=\operatorname{Id}$ and $\sfH_1=\sfZ_1$ are injective,
and it remains to consider $t\in(0,1)$.

Suppose $\sfH_t\lb \cdot\rb$ is not injective for some $t \in \lb 0,1\rb$. Then there exist distinct points $\bm x, \bm y$ such that $\sfH_t\lb \bm x\rb = \sfH_t\lb \bm y\rb$, i.e., 
\begin{equation}
    \label{eq: H_t not 1-1}
    \sfZ_1\lb \bm y\rb - \sfZ_1\lb \bm x\rb = -\frac{1-t}{t} \lb \bm y- \bm x\rb = -\frac{1-t}{t} \lb \sfZ_0\lb \bm y\rb-\sfZ_0\lb \bm x\rb\rb.
\end{equation}


Therefore, by writing 
\[
\bm\delta_t\lb \bm x, \bm y\rb:= \sfZ_t\lb \bm y\rb - \sfZ_t\lb \bm x\rb,
\]
the identity \eqref{eq: H_t not 1-1} implies that the angle between $\bm\delta_1\lb \bm x, \bm y\rb$ and $\bm\delta_0\lb \bm x, \bm y\rb$ is $\pi$. Therefore, one { equivalent} condition for global injectivity of $\sfH_t\lb \cdot\rb$ is
\begin{equation}
    \label{eq: angle n_t}
      \Angle\lb \bm\delta_1\lb \bm x, \bm y\rb, \bm\delta_0\lb \bm x, \bm y\rb\rb  < \pi \qquad \text{for all} \qquad \bm x \neq \bm y.
\end{equation}
    \begin{lemma} \label{lemma: injectivity lemma}
  If \eqref{eq: local regularity velocity} and \eqref{eq: angle n_t} hold, then $\sfH_t$ is globally injective for all $t \in [0,1]$. Consequently, the
$1$-RF endpoint coupling $\lb \bm Z_0, \sfZ_1\lb \bm Z_0\rb\rb$ is straight in the sense of Definition~\ref{def: Qiang straightness}.
\end{lemma}




The proof can be found in Appendix~\ref{sec:proof-injectivity+straightness}. This is a slight modification of what was used in the proof of Theorem 6 in \cite{bansal2025wasserstein}, which requires $\sfH_t$ to be a diffeomorphism to obtain straightness.
The rest of this subsection focuses on providing sufficient conditions for the inequality \eqref{eq: angle n_t} to hold, which in turn yields a criterion for straightness. To set the stage, we begin by noticing the following direct consequence of the ODE \eqref{eq: ode non-random}:
\begin{align*}
    \ddt\bm\delta_t\lb \bm x, \bm y\rb  &= \bm v_t\lb \sfZ_t\lb \bm y\rb\rb - \bm v_t\lb \sfZ_t\lb \bm x\rb\rb \\
    & = \left[ \int_0^1 D \bm v_t \lb  \xi \cdot  \sfZ_t\lb \bm y\rb + \lb 1-\xi\rb\cdot \sfZ_t\lb \bm x\rb \rb \; d\xi \right] \bm\delta_t\lb \bm x, \bm y\rb.
\end{align*}
 Next, setting 
 \begin{align} \label{eq:A_t}
     \bA_t\lb \bm x, \bm y\rb:= \int_0^1 D \bm v_t \lb \xi \cdot \sfZ_t\lb \bm y\rb + \lb 1-\xi\rb \cdot \sfZ_t\lb \bm x\rb \rb\; d\xi,
 \end{align}
 the ODE in the above display can be simply expressed as 
\begin{equation}
    \label{eq: ode delta_t}
    \ddt \bm\delta_t\lb \bm x, \bm y\rb = \bA_t\lb \bm x, \bm y\rb \cdot \bm\delta_t\lb \bm x, \bm y\rb.
\end{equation}
The ODE \eqref{eq: ode delta_t} provides control on the maximum angular rotation of $\bm\delta_{1}\lb \bm x,\bm y\rb$  under some assumptions on $\bA_t\lb \bm x, \bm y\rb$. It turns out that the spectral properties of $\bm A_t$ can be used to deduce sufficient conditions for \eqref{eq: angle n_t}.
Next, we present a heuristic argument for deriving \eqref{eq: angle n_t} from the spectral properties of $\bm A_t$. All of the derivations below can be made rigorous; { see Appendix~\ref{sec:proof-injectivity+straightness}} for more details. 
\paragraph*{Bounding the maximum rotation via the discretized solution} The discretized solution of \eqref{eq: ode delta_t} takes the following form:
\[
\bm{\delta}^{\mathrm{disc}}_{1}\lb \bm x, \bm y\rb = \bs{\bI_d + \frac{1}{N} \bA_{t_{N-1}}\lb \bm x, \bm y\rb}\ldots \bs{\bI_d + \frac{1}{N} \bA_{t_1}\lb \bm x, \bm y\rb}\bs{\bI_d + \frac{1}{N} \bA_{t_0}\lb \bm x, \bm y\rb}\bm\delta_0\lb \bm x, \bm y\rb.
\]
For notational convenience, define
\[
\tilde \bA_j = \bI_d + \frac{1}{N} \bA_{t_j}\lb \bm x, \bm y\rb; 
\qquad  \text{and}
\qquad \bA^{[j]} = \tilde \bA_j\cdots\tilde \bA_0,\qquad j=0,\ldots,N-1.
\]
 Choose $N$ sufficiently large such that the matrices 
$$\la\tilde \bA_j\ra_{j=0}^{N-1}$$ 
are symmetric and positive definite. Consider two cases: 
\vspace{1em}

\noindent\underline{ \it Case 1: $\tilde \bA_j$'s are commutative.}
In this case, $\Angle \lb \bm{\delta}^{\mathrm{disc}}_1\lb \bm x, \bm y\rb, \bm\delta_0\lb \bm x, \bm y\rb\rb< \pi/2$ as $\bA^{[N-1]}$ is symmetric and positive definite.

\vspace{1em}

\noindent\underline{\it Case 2: $\tilde \bA_j$'s are not commutative.} In this case, we can only bound the $\Angle\lb \bm u, \bA^{[N-1]} \bm u\rb$. First note that by elementary geodesic theory on $\Sp^{d-1}$, the following holds:
\begin{align*}
&\Angle\lb \bm\delta_0\lb \bm x, \bm y\rb, \bA^{[N-1]} \bm\delta_0\lb \bm x, \bm y\rb\rb\\
& \le \Angle\lb \bm\delta_0\lb \bm x, \bm y\rb, \bA^{[0]} \bm\delta_0\lb \bm x, \bm y\rb\rb \\
& \quad +  \sum_{j = 0}^{N-2} \Angle \lb \bA^{[j]}\bm\delta_0\lb \bm x, \bm y\rb,\tilde \bA_{j+1} \, \bA^{[j]} \bm\delta_0\lb  \bm x, \bm y\rb \rb\\
&\le \sum_{j = 0}^{N-1} \phi_{\tilde \bA_j}.
\end{align*}
Therefore, to ensure $\Angle \lb \bm{\delta}^{\mathrm{disc}}_1\lb \bm x, \bm y\rb, \bm\delta_0\lb \bm x, \bm y\rb \rb< \pi$ , we need $ \sum_{j = 0}^{N-1} \phi_{\tilde \bA_j} < \pi$. It turns out that this term can be bounded using
\[
\sum_{j = 0}^{N-1} \phi_{\tilde \bA_j} \leq \frac{1}{2} \int_0^1 \la \lambda_{\max}[\bA_t\lb \bm x, \bm y\rb] - \lambda_{\min}[\bA_t\lb \bm x, \bm y\rb]\ra \;dt +o(1).
\]

By letting $ N \to \infty$, we deduce that
\begin{align*}
\Angle\lb \bm\delta_1\lb \bm x, \bm y\rb, \bm\delta_0\lb \bm x, \bm y\rb\rb 
& \approx \lim_{N \to \infty}\Angle\lb \bm{\delta}^{\mathrm{disc}}_1\lb \bm x, \bm y\rb, \bm\delta_0\lb \bm x, \bm y\rb\rb \\
& \le \frac{1}{2} \int_0^1 \la \lambda_{\max}[\bA_t\lb \bm x, \bm y\rb] - \lambda_{\min}[\bA_{t}\lb \bm x, \bm y\rb]\ra \;dt.
\end{align*}
The above heuristic can be proved rigorously and suggests our proposed criterion, as stated in the next theorem. 
\begin{thm}[Straightness criterion]
\label{theorem: straightness theorem}
Let $(\bm X_0,\bm X_1)\sim \bcN(\bm 0,\bI_d)\otimes p_1$ with
$\bbE\norm{\bm X_1}_2^2<\infty$, and let $\bm v_t$ be the associated
RF velocity field. Assume that the flow $\sfZ_t$ defined in
\eqref{eq: ode non-random} exists on $[0,1]$~for every initial
point and satisfies the regularity condition
\eqref{eq: local regularity velocity}.
With $\bA_t$ defined in \eqref{eq:A_t}, suppose that at least one
of the following conditions holds:
\begin{enumerate}[label=(\alph*)]
    \item \textit{Commutativity.} For every $\bm x\ne\bm y$,
    \[
        \bs{\bA_t(\bm x,\bm y),\bA_s(\bm x,\bm y)}
        =\bm 0,
        \qquad s,t\in(0,1).
    \]
    Here $[\bm A, \bm B]$ denotes the commutator of two matrices, see the notation in Section \ref{sec:intro} for the precise definition. 

    \item \textit{Bounded spectral spread.} For every $\bm x\ne\bm y$,
    \[
        \frac12\int_0^1
        \left(
            \lambda_{\max}[\bA_t(\bm x,\bm y)]
            -\lambda_{\min}[\bA_t(\bm x,\bm y)]
        \right)\,dt<\pi.
    \]
\end{enumerate}
Then the angle condition \eqref{eq: angle n_t} holds, and
$\sfH_t$ is globally injective for every $t\in[0,1]$.
Consequently, the $1$-RF endpoint coupling
\[
    (\bm Z_0,\bm Z_1)=\Rectify{\bm X_0,\bm X_1}
\]
is straight in the sense of Definition~\ref{def: Qiang straightness}.
\end{thm}

{The proof is deferred to Appendix~\ref{sec:proof-injectivity+straightness}.}
Condition (a) (commutativity) holds trivially in the one-dimensional case. In \cite{bansal2025wasserstein}, the authors showed the stronger commutativity condition $[D \bm v_t\lb \bm x\rb, D \bm v_s\lb \bm y\rb] = \bm 0$ for all pairs of  $(\bm x, \bm y)$ and $(s,t)$ 
when the target distribution is $\bcN\lb \bm \mu, \bSigma\rb$ or a two-component GMM with isotropic covariance. Therefore, 1-RF yields a straight coupling in the above cases. In the next subsections, we will derive model-specific conditions for straightness under other popular and more general types of data distributions. 

\begin{remark}
\label{remark: eigen gap}
Letting
 $$\Delta\lb t\rb := \sup_{\bm z \in \bbR^d} \Big( \lambda_{\max}[D \bm v_t\lb \bm z\rb] - \lambda_{\min}[D \bm v_t\lb \bm  z\rb]\Big) , \qquad t \in [0,1]$$
 and noting that  $\lambda_{\max}[\bA_t\lb \bm x, \bm y\rb] - \lambda_{\min}[\bA_t\lb \bm x, \bm y\rb] \le \Delta\lb t\rb$ for all $t$,  Condition (b) in Theorem~\ref{theorem: straightness theorem} is implied by the stronger condition
   \[
   \frac12\int_0^1 \Delta\lb t\rb\; dt <  \pi.
   \]
\end{remark}

\subsection{Straightness when generating log-concave distributions} \label{sec:straightness-log-concave}
In this section we are going to investigate conditions under which straightness holds when the target distribution is log-concave. Towards that goal, suppose that the target distribution admits a Lebesgue density of the form 
\[
p_1\lb \bm x\rb = \exp\lb -V\lb \bm x\rb \rb,
\]
for some $C^3$ function $V: \bbR^d \to \bbR$. We further assume that
\begin{align} \label{strong convex+smoothness}
    m \bI_d \preceq \nabla^2 V\lb \bm x\rb \preceq L\bI_d
\end{align}
for some constants  $0 < m<L$.
Distributions satisfying the above condition \eqref{strong convex+smoothness} are known as smooth and strongly log-concave distributions and have been studied extensively in the optimization and machine learning literature; see, for example, \cite{dalalyan2017theoretical,durmus2019analysis} and the references therein. In our next result, we verify condition (b) of Theorem \ref{theorem: straightness theorem} to formulate a condition that guarantees the straightness of couplings obtained from $2$-RF when the data distribution satisfies \eqref{strong convex+smoothness}. 

\begin{thm} \label{2RF straight-logconcave}
    If $L/m <
    \exp\lb 4\pi\rb$, then $2$-RF is straight when applied to the independent coupling $\bcN\lb \bm 0, \bI_d\rb \otimes p_1$.
\end{thm}

Thus, Theorem \ref{2RF straight-logconcave} shows that the
condition number can determine the straightness of RF starting from
independent couplings. This should not come as a big surprise: the condition number is a central geometric
parameter in the log-concave sampling literature; see, e.g.,
\cite{chewi2026logconcave} and the references therein. It captures
the interaction between concentration and the potential's smoothness.
This parameter also appears in recent high-accuracy sampling
guarantees. For example, under their oracle model,
\citet{chen2026highaccuracy} obtain a complexity bound that depends
linearly on $L C_{\mathrm{LSI}}$, where $C_{\mathrm{LSI}}$ is the
log-Sobolev constant of the target distribution; see Definition~B.2
in \citet{chen2026highaccuracy}. Since $C_{\mathrm{LSI}}\leq 1/m$
under strong log-concavity, this effective condition number is
bounded by $L/m$. Our straightness result thus gives a complementary
geometric interpretation of the same parameter.




The proof of Theorem \ref{2RF straight-logconcave} hinges on bounding the maximum rotation angle from equation \eqref{eq: angle n_t}. In particular, we show that
\begin{align*}
 \sup_{\bm x \neq \bm y} \Big\{ \Angle \lb \bm\delta_1\lb \bm x, \bm y\rb, \bm\delta_0\lb \bm x, \bm y\rb \rb \Big\} &\leq \frac{1}{2} \Bigg[ \int_0^1 \frac{s-m\lb 1-s\rb}{s^2 +m\lb 1-s\rb^2} ds \\
 &\qquad - \int_0^1 \frac{s-L\lb 1-s\rb}{s^2 +L\lb 1-s\rb^2} ds \Bigg]\\
 &= \frac{1}{4} \log\lb \frac{L}{m}\rb,
\end{align*}
a bound implying that condition (b) in Theorem \ref{theorem: straightness theorem} holds if
$$L/m < \exp\lb 4\pi\rb \approx 286751.$$
In other words, if the condition number $L/m$ of the distribution is smaller than $\exp\lb 4\pi\rb$, then $2$-RF always yields straight couplings when applied to the independent coupling $\bcN\lb \bm 0, \bI_d\rb \otimes p_1$. To the best of our knowledge, this is the first straightness result under log-concave distributions. Interestingly, this upper bound on the distribution's condition number is {\it dimension-free}.

At first glance, one might wonder how sharp this criterion based on condition numbers is. Since the bound is independent of the dimension, it does not capture any high-dimensional effects, and one might expect it to become loose as the dimension increases. This turns out not to be the case. Surprisingly, we show that in any dimension $d \geq 3$, this condition is necessary to guarantee straightness for all data distributions satisfying \eqref{strong convex+smoothness}; see Section \ref{sec:non-straight-log-concave} for more details.

\subsection{Straightness when generating isotropic GMMs} \label{sec:straightness-gmm}
Next, we consider another class of popular data distributions, the Gaussian mixture models, arguably one of the most popular and widely studied classes of mixture models, with huge expressive power. The readers are referred to \cite{mclachlan2000finite,bishop2006pattern,nguyen2020approximation,lee2018beyond,chen2025learning} for basic properties of GMMs and their applications in machine learning. Below, we assume that 
\begin{equation}
\label{eq: base target model}
\bm X_1 \sim \pgmm := \sum_{j = 1}^K \pi_j \bcN\lb  \bm \mu_j, \sigma^2 \bm I_d\rb.
\end{equation}
To describe our straightness result, put 
\begin{align*}
    \quad \Dmax := \max_{i \ne j} \norm{\bm \mu_i - \bm \mu_j}_2.
\end{align*}
The quantity $D_{\rm max}$ is the maximum interpoint distance (see \cite{heiny2025maximum} for more details) formed by the set $\la \bm \mu_1,\dots,\bm \mu_K \ra$.  For data with distributions of GMM-types in \eqref{eq: base target model}, our criterion for straightness reads  


    \begin{thm} \label{theorem: straightnes GMM}
    If $\Dmax^2 < 16 \pi \sigma^2$, then 2-RF is straight when applied to the independent coupling $\bcN\lb \bm 0, \bI_d\rb \otimes \pgmm$.
\end{thm}

Theorem \ref{theorem: straightnes GMM} states that the dimension-free bound on the maximum interpoint distance
\[
D_{\rm max}^2 < 16\pi \cdot \sigma^2  
\]
is enough to guarantee the straightness of $2$-RF.  We note that a weaker version of Theorem \ref{theorem: straightnes GMM} was established in \cite{bansal2025wasserstein}, where the authors obtained the sub-optimal condition $D^2_{\rm max} \leq 4\sigma^2$.  Similar to the smooth, strongly log-concave case, our dimension-free condition is sharp; see Section \ref{sec:non-straight-GMM} for more details. Roughly speaking, for any constant greater than $16\pi$, there exists a distribution of the form \eqref{eq: base target model} such that $2$-RF provably yields a non-straight coupling when applied to the independent coupling.

It is worth pointing out that  a variant of the normalized maximal mean separation $D_{\rm max}/\sigma$ appears in recent diffusion sampling guarantees. Under some score regularity and estimation assumptions,
\citet{gatmiry2026highaccuracy} obtain high-accuracy sampling
guarantees whose complexity depends quadratically on $R/\sigma$,
up to logarithmic factors, where $R$ bounds the radius of the
mixing distribution. For isotropic GMMs, the smallest enclosing
radius $R_\star$ of the component means satisfies
\[
    \frac{D_{\max}}{2\sigma}
    \leq \frac{R_\star}{\sigma}
    \leq \frac{D_{\max}}{\sigma}.
\]
Thus, after centering, their normalized radius is equivalent,
up to a universal factor, to our normalized maximal mean separation. Therefore, our result gives a complementary geometric
interpretation of this ratio by identifying a threshold that
guarantees exact straightness.

\section{Sharpness of the maximum rotation criterion} \label{sec:nonstraight-construction}

In Section~\ref{sec: a sufficient condition for straightness}, we established sufficient
conditions for the straightness of 2-RF through the maximum-rotation framework.
In particular, we showed that 2-RF generates a straight coupling for every smooth,
strongly log-concave target distribution whose condition number is smaller than
$e^{4\pi}$. The same argument guarantees straightness for isotropic Gaussian
mixtures with identity component covariances whenever the squared maximum pairwise
distance between the component means is smaller than $16\pi$.
These results naturally raise the question of whether the thresholds imposed by
the maximum-rotation criterion are merely artifacts of our analysis, or whether
they reflect a genuine obstruction to straightness. The purpose of this section
is to show that the latter is the case. More precisely, we construct log-concave
and Gaussian-mixture examples for which one reflow produces a non-straight
coupling once the corresponding maximum-rotation criterion is violated. Thus,
the thresholds obtained in Section~\ref{sec: a sufficient condition for straightness}
are, in this sense, sharp.

Throughout this section, $\sfZ_t$ denotes the RF transport map defined in \eqref{eq: ode non-random}, and $\sfZ_1$ is thus the endpoint map. Therefore, $\lb\bm Z_0,\sfZ_1\lb\bm Z_0\rb\rb$ is the coupling obtained after the first rectification and, in turn, its rectification generates the second RF. 

We will prescribe distributions for which 2-RF is provably not straight. We will detail the constructions for the case of $d=3$;  as explained in Section~\ref{sec:non-straight-log-concave} this setting can be extended to arbitrary $d\geq3$ via a simple tensorization argument.


\subsection{A criterion for non-straightness}
\label{ns:sec:criterion}

Throughout this subsection, we assume that the endpoint map $\sfZ_1:\bbR ^3\to\bbR ^3$ is $C^1$ and we  write $\mb S_{++}^k$ for the set of real symmetric positive-definite $k\times k$ matrices. Put
\[
     \bm{D} ( \bm{x} )=\sfZ_1( \bm{x} )- \bm{x} ,
    \qquad
    \sfH_t ( \bm{x} )=(1-t) \cdot  \bm{x} +t \cdot  \sfZ_1( \bm{x} )
      = \bm{x} + t \cdot  \bm{D} ( \bm{x} ).
\]
For a square-integrable coupling $\lb\bm X,\bm \sfZ_1\lb\bm X\rb\rb$, define the {\it straightness defect}  
\[
    S_t=S_t(\bm X)
    :=\bbE\norm{\bm D\lb\bm X\rb
       -\bbE\left[\bm D\lb\bm X\rb
          \mid\sfH_t \lb\bm X\rb\right]}^2.
\]
By Definition~\ref{def: Qiang straightness}, the coupling is straight if and only if 
$S_t=0$ for almost every $t \in(0,1)$; equivalently,
$\bm D\lb\bm X\rb$ is measurable with respect to
$\sigma\lb\sfH_t \lb \bm X\rb \rb$ for almost every such $t$.
The goal of this subsection is to propose a sufficient condition for this coupling to be non-straight. 

Since $ \bm{D} ( \bm{X} )=t^{-1}(\sfH_t ( \bm{X} )- \bm{X} )$, $S_t$ can be written in terms of the chord map $\sfH_t $ as
\begin{equation}
\label{ns:eq:defect-as-chord}
    S_t
    =t^{-2}\cdot \bbE \norm{ \bm{X} -\bbE \left[ \bm{X} \mid \sfH_t \lb  \bm{X} \rb \right]}^2 .
\end{equation}
The following result provides the main technical tool for our construction of non-straight couplings. It shows that it suffices to formulate a data distribution for which the Jacobian of the endpoint map has two distinct negative eigenvalues at some point.

\begin{prop}
\label{ns:prop:criterion}
Let $\sfZ_1:\bbR ^3\to\bbR ^3$ be $C^1$ such that
\begin{equation} \label{eq:properness}
    \sfZ_1( \bm{x} )=\bB  \bm{x} + \bm{R} ( \bm{x} ),
    \qquad \bB \in\Sp_{++}^3,
    \qquad \sup_{ \bm{x} }\norm{ \bm{R} ( \bm{x} )}<\infty .
\end{equation}
and  $D\sfZ_1$ has real
eigenvalues $\lambda_1<\lambda_2<0<\lambda_3$ at some point $\bm x_*$.  Then, if $ \bm{X} $ has a strictly positive density and finite second moment,
\[
S_t>0
\]
on the interval $\bigl((1-\lambda_1)^{-1},(1-\lambda_2)^{-1}\bigr)$, almost surely. 
As a result, the coupling
$( \bm{X} ,\sfZ_1( \bm{X} ))$ is non-straight.
\end{prop}

Proposition \ref{ns:prop:criterion} provides a route to proving that a coupling is {\it non-straight}. To the best of our knowledge,  the criterion introduced above is novel. The constructions below for the smooth, strongly log-concave data rely crucially on this result. We note that it is also relatively simple to generalize Proposition \ref{ns:prop:criterion} to arbitrary dimension, under the assumption that the Jacobian matrix $D \sfZ_1$ has exactly two distinct negative eigenvalues (counting algebraic multiplicities) at some point. 

\subsection{Non-straight couplings under smooth \& strongly log-concave data} \label{sec:non-straight-log-concave}

We will now use Proposition \ref{ns:prop:criterion} to show that the threshold $e^{4\pi}$ in Theorem \ref{2RF straight-logconcave} is sharp. In fact, our results show that $e^{4\pi}$ is the largest condition number possible to guarantee the straightness of $2$-RF. Specifically, while  Theorem \ref{2RF straight-logconcave} already shows that such a condition number must be no smaller than $e^{4\pi}$,  Theorem \ref{ns:thm:main} below shows that any condition number larger than this threshold will preclude straightness for one particular choice of the function $V$.

The proof of Theorem \ref{ns:thm:main} entails constructing a limiting ODE where a suitably normalized version of the Jacobian map has exactly two negative eigenvalues at some point, for all large time. We then use Proposition \ref{ns:prop:criterion} to finish the proof. 

\begin{thm} \label{ns:thm:main}
For every $\ve>0$ and every $d\geq3$, there exist
$V\in C^\infty\lb\bbR^d\rb$ and constants $0<m<L<\infty$ such that
\[
    m\bI_d\preceq\nabla^2V\lb\bm x\rb\preceq L\bI_d.
    \qquad \bm x\in\bbR^d,
    \qquad
    e^{4\pi}<\frac{L}{m}<e^{4\pi}+\ve,
\]
and two-step RF applied to the independent coupling
$\bcN\lb\bm0,\bI_d\rb\otimes p_1$ yields a non-straight flow, where
$p_1\lb\bm x\rb\propto\exp\lb-V\lb\bm x\rb\rb$.
\end{thm}

As noted above, we only need to consider the case $d=3$. Indeed, for the case $d\geq 4$, let $V_3$ be the potential function in Theorem \ref{ns:thm:main}. Then, we can simply choose the new potential
\[
V_d \lb \bm x, \bm y \rb = V_{3} \lb \bm x \rb + \frac{m}{2} \cdot  \| \bm y\|^2,
\qquad 
\bm x \in  \mb R^3, \bm y \in \mb R^{d-3}. 
\]
The target distribution is thus
\[
\exp \lb -V_3 \rb \otimes \bcN \lb \bm 0, m^{-1} \cdot  \bm I_{d-3} \rb.
\]
The distribution above has the same condition number as $V_3$, and is non-straight because its endpoint map has the form
\[
\sfZ^{V_d} \lb \bm x, \bm y \rb = \lb \sfZ_1^{V_3} \lb \bm x \rb, m^{-1/2} \cdot \bm y \rb, 
\qquad 
\bm x \in  \mb R^3, \bm y \in \mb R^{d-3},
\]
where $\sfZ_1^{V_3}$ is the endpoint map with respect to $V_3$. 
Clearly, the straightness defect of $V_d$ is the same as $V_3$, which is positive and yields non-straightness. 


 The construction has two parts. First, we construct a smooth, strongly
convex potential $W$ and an associated auxiliary flow whose endpoint
map has a Jacobian with two distinct negative eigenvalues. Next,
we show that the RF endpoint map associated with the law $p_1^\beta\propto\exp\lb-\beta W\rb$ also has the same eigenvalue property when $\beta$ is sufficiently large (but finite). A crucial ingredient of the proof is a gluing argument that is used to guarantee that the condition number stays close to $e^{4\pi}$. The detailed proofs are presented
in Appendix~\ref{sec:nonstraight-proofs}.

\subsection{Non-straight couplings under Gaussian mixture models} 
\label{sec:non-straight-GMM}
For Gaussian mixtures, the relevant quantity that determines straightness is the ratio between the maximum interpoint distance between the component means and their common standard deviation. The next result shows that the bound in Theorem~\ref{theorem: straightnes GMM}, which guarantees straightness of 2-RF whenever the maximal separation among the means $\Dmax^2$ is smaller than $16\pi \sigma^2$, is sharp. In detail, Theorem~\ref{ns:thm:gmm} below proves that, for any value of $\Dmax^2$ larger than this threshold, there exists an isotropic GMM with variance parameter $\sigma^2$ and at least $4$ components for which 2-RF provably fails to be straight.


 In particular, our result improves and refines the bound in \cite{bansal2025wasserstein}, who, in the same settings, show that 2-RF is straight as long as $\Dmax^2/\sigma^2 \leq 4$.
 

\begin{thm}\label{ns:thm:gmm}
Fix $d\geq3$ and $\sigma>0$. For any $\ve>0$, there exist finitely many means $\bm\mu_1,\ldots,\bm\mu_K\in\bbR^d$, $K\geq 4$ and
positive weights $\pi_1,\ldots,\pi_K$ summing to one such that the mixture 
\[
p_1=\sum_{j=1}^K\pi_j \bcN\lb\bm\mu_j,\sigma^2\bI_d\rb
\]
satisfies
\[
16\pi
<
\frac{\Dmax^2}{\sigma^2}
<
16\pi+\ve,
\qquad
\Dmax:=\max_{i\ne j}\norm{\bm\mu_i-\bm\mu_j}_2,
\]
while the 2-RF initialized from the independent coupling
$\bcN\lb\bm0,\bI_d\rb\otimes p_1$ produces a non-straight flow.
\end{thm}

The strategy of the proof is generally similar to that of Theorem \ref{ns:thm:main}, in that we construct a limiting ODE whose endpoint map's Jacobian has exactly two negative eigenvalues at some point.

The number of components $K$ chosen in the proof of Theorem \ref{ns:thm:gmm} is huge and we have not attempted to make it explicit for clarity of the proof. The construction underlying Theorem~\ref{ns:thm:gmm} is illustrated in
Figure~\ref{fig:gmm-construction}.
We arrange the Gaussian component means in antipodal pairs on a cylinder
around the $z$-axis, with their horizontal directions rotating as the
height increases (Figure~\ref{fig:gmm-construction-means}).
The mixture weights are chosen so that, along a distinguished RF
trajectory on the $z$-axis, the posterior mass progressively concentrates
on successive pairs of components.
This particular construction, coupled with the accumulated rotational effect of the  components along the RF trajectory, produces an endpoint Jacobian with two distinct negative
eigenvalues, corresponding to a net half-turn in the transverse
directions.
By Proposition~\ref{ns:prop:criterion}, this implies non-straightness
of the coupling obtained after the first rectification.
The proof heuristics and details are deferred to Appendices~\ref{sec: ns gmm proof heuristic} and~\ref{sec: ns gmm proof}.  

\section{Approximate straightness and choosing the number of rectification steps}
\label{sec: approximate straightness}

In the previous sections, Theorem \ref{theorem: straightness theorem} provided sufficient conditions for the endpoint coupling generated by $1$-RF to be straight, and hence for $2$-RF to produce a straight flow, when starting from an independent coupling. However, exact straightness is restrictive in practice as one often only has access to the data, and it is hard to verify the structural assumptions required for exact straightness just from the observed data.


From a practical perspective, each additional rectification requires training a
new velocity field, which can be computationally expensive. Consequently,
although additional rectification typically improves the straightness of the
learned flow (see \cite{liu2022rectified} for more details), continuing the procedure indefinitely is neither necessary nor
computationally desirable.{ Empirically, 
\cite{lee2024improving} demonstrate that a single
reflow iteration can produce nearly straight
trajectories, enabling accurate few-step generation.
Related work by \cite{wang2025rectified} has also explored efficient
generation without requiring exactly straight
trajectories.}
Motivated by these recent findings, we argue that a \emph{sufficiently straight} flow is
often adequate for achieving near-optimal sampling performance. This naturally
raises the following practical question.

\begin{quote}
\emph{Given a trained $K$th rectified flow, can we determine from finite data
whether the learned flow is sufficiently straight, or whether another
rectification step is statistically warranted?}
\end{quote}

To give an answer to this question, we first define the concept of {\it approximate straightness}. Let $(\bm Z_0,\sfZ_1\lb\bm Z_0\rb)$ be a deterministic population coupling between two densities $p_0$ and $p_1$, with \emph{finite second moments}. 
Recall that $\sfH_t\lb\bm x\rb=(1-t)\bm x+t\sfZ_1\lb\bm x\rb$.
\begin{defi}[Approximate straightness]
\label{def: approximate straightness}

For $\delta\geq0$, we say that the coupling $(\bm Z_0,\sfZ_1\lb\bm Z_0\rb)$ is $\delta$-approximately straight if
\[
\begin{aligned}
    S(\bm{\mc Z})
    &:=
    \bbE_{\bm Z_0,t}\left[
        \Big\|
            \sfZ_1\lb\bm Z_0\rb-\bm Z_0
            -
            \bbE\left[
                \sfZ_1\lb\bm Z_0\rb-\bm Z_0
                \,\middle|\,
                \sfH_t\lb\bm Z_0\rb
            \right]
        \Big\|_2^2
    \right]
    \leq\delta,
\end{aligned}
\]
almost surely with respect to independent draws of $\bm Z_0\sim p_0$ and
$t\sim\operatorname{Unif}[0,1]$.
\end{defi}


Clearly, when $\delta=0$, this reduces to a straight coupling in the sense of Definition~\ref{def: Qiang straightness}. For the population endpoint coupling obtained after $K$ rectifications,
we denote this defect by $S^{(K)}(\mc Z)$. 
\cite{liu2023flow} showed that successive rectifications lead to a vanishing straightness defect, i.e., $\lim_{K \to \infty} S^{(K)}(\mc Z) = 0$. To be precise, Theorem D.7 in \cite{liu2023flow} states that 
\[
\sum_{k = 0}^K S^{(k)}(\mc Z) \le  \bbE\norm{\bm X_1-\bm X_0}_2^2<\infty 
\quad \text{for all $K$},
\]
and hence, the  limiting straightness claim mentioned above follows. 
In this section we consider this setting, with the goal of minimizing the number of reflow steps before obtaining a solution that, though not straight, has low enough defect to lead to practicable sampling methodologies. To this end, the concept of approximate straightness in Definition \ref{def: approximate straightness} is useful since one can construct a consistent estimator for it, given the learned vector field in a given rectification round, and construct a statistically valid stopping rule. 
In the case where the data belong to the class of smooth, strongly log-concave distributions, one can see a clear connection between the straightness result in Theorem \ref{2RF straight-logconcave} and the approximate straightness concept in Definition \ref{def: approximate straightness}.
\begin{prop}[Approximate straightness under strong log-concavity]
\label{prop: approximate straightness logconcave}
Suppose that $p_1$ satisfies \eqref{strong convex+smoothness}.
Let $\bm{\mc Z}$ be the population rectified flow starting from
$\bcN\lb\bm 0,\bI_d\rb\otimes p_1$, with endpoint coupling
$(\bm Z_0,\sfZ_1\lb\bm Z_0\rb)$.
Then
\begin{equation}
\label{eq: approximate straightness logconcave}
    S(\bm{\mc Z})
    \leq
    \frac{2d}{(mL)^{1/4}} \cdot 
    \sinh\left(
        \sqrt{
            \left[ \,
                \frac{\log^2\lb L/m\rb}{16}-\pi^2
            \, \right]_+
        }
    \right)
\end{equation}
where $[a]_+:=\max\{a,0\}$.
\end{prop}

The main consequence of the last result is that, under the stated assumptions, the coupling is $\delta_{m,L}$-approximately straight,
where $\delta_{m,L}$ denotes the right-hand side of \eqref{eq: approximate straightness logconcave}.
It is immediate that when $\log (L/m) \leq 4\pi$, Proposition \ref{prop: approximate straightness logconcave} yields the straightness of the coupling obtained from $2$-RF. Interestingly, the bound in \eqref{eq: approximate straightness logconcave} is of order $O \lb d/\sqrt{m} \rb$ for all large $L,m$. In fact, we can check that
\[
\frac{2d}{(mL)^{1/4}} \cdot 
    \sinh\left(
        \sqrt{
            \left[ \,
                \frac{\log^2\lb L/m\rb}{16}-\pi^2
            \, \right]_+
        }
    \right) \leq \frac{d}{\sqrt{m}}
\]
when the ratio $L/m$ is large. This implies that even when exact straightness fails, approximate straightness  can still hold, as long as the data distribution is generated by a potential function with large strong-convexity parameter. The detailed proof is deferred to Appendix~\ref{sec: proof of approx-straightness results}.


An important consideration in designing such a certification procedure is that
its computational overhead should be negligible compared with the cost of
training an additional rectified flow. Since the objective of the stopping rule
is precisely to decide whether another expensive rectification step is
necessary, the certification procedure itself should avoid solving additional
optimization problems. Motivated by this principle, we introduce a
trajectory-based straightness functional that can be evaluated directly from
independently generated trajectories of the learned flow. Once the velocity
field has been trained, estimating this functional requires only forward
simulation of the learned ordinary differential equation together with
evaluations of the learned velocity field along the generated trajectories.
Consequently, the proposed certification procedure incurs only a modest
computational overhead relative to the cost of training the rectified flow
itself.



\begin{figure}[htbp]
\centering

\begin{tikzpicture}[scale = 0.7, 
transform shape,
>=latex,
every node/.style={font=\small},
box/.style={
draw,
rounded corners,
minimum width=2.6cm,
minimum height=0.7cm,
align=center
}
]


\node at (2.8,3.7) {\bf (a) Trajectory-based straightness};

\draw[thick,blue]
(0,0)
.. controls (0.8,0.8) and (1.7,1.6)
.. (2.8,2.0)
.. controls (3.7,2.2) and (4.7,1.8)
.. (5.8,1.0);

\fill (0,0) circle (2pt);
\fill (5.8,1.0) circle (2pt);

\node[below] at (0,0) {$\bm Z_0^{(K)}$};
\node[right] at (5.8,1.0) {$\bm Z_1^{(K)}$};

\fill[red] (3.4,2.05) circle (2pt);
\node[above] at (3.4,2.05) {$\bm Z_t^{(K)}$};

\draw[->,red,thick]
(3.4,2.05)--+(0.8,0.45);

\node[right] at (4.25,2.45)
{$\widehat{\bm v}^{(K)}\lb \bm Z_t^{(K)},t\rb$};

\draw[->,thick]
(0,-0.45)--(5.8,-0.45);

\node at (2.9,-0.85)
{$\bm Z_1^{(K)}-\bm Z_0^{(K)}$};

\node[align=center] at (2.9,-2.0)
{
Trajectory residual
$
\psi
=
\norm{
\bm Z_1^{(K)}
-
\bm Z_0^{(K)}
-
\widehat{\bm v}^{(K)}\lb \bm Z_t^{(K)},t\rb
}_2^2
$
};


\node at (12.3,3.7) {\bf (b) Statistical certification};

\node[box] (train) at (12.3,3.0)
{Train $K$th\\Rectified Flow};

\node[box] (traj) [below=0.9cm of train]
{Generate\\fresh trajectories};

\node[box] (psi) [below=0.9cm of traj]
{Compute\\trajectory residuals};

\node[box] (ci) [below=0.9cm of psi]
{Estimate $\widehat{\mc T}_N^{(K)}$\\and confidence bound};

\node[diamond,
draw,
aspect=2,
align=center,
below=1.1cm of ci]
(decision)
{$U_{N,K}\le\delta$?};

\node[box] (stop)
[left=1.9cm of decision]
{Stop};

\node[box] (next)
[right=2.0cm of decision]
{Train\\$(K+1)$th RF};

\draw[->] (train)--(traj);
\draw[->] (traj)--(psi);
\draw[->] (psi)--(ci);
\draw[->] (ci)--(decision);

\draw[->] (decision)--node[above]{Yes}(stop);
\draw[->] (decision)--node[above]{No}(next);

\end{tikzpicture}

\caption{
Overview of the proposed statistical certification procedure.
\textbf{(a)} The trajectory-based straightness functional compares the endpoint
displacement with the instantaneous velocity evaluated at a uniformly sampled
time along the learned trajectory.
\textbf{(b)} After training the $K$th rectified flow, fresh trajectories are
generated to estimate the straightness functional, construct a confidence
bound, and determine whether another rectification step is statistically
necessary.
}
\label{fig:trajectory_pipeline}
\end{figure}
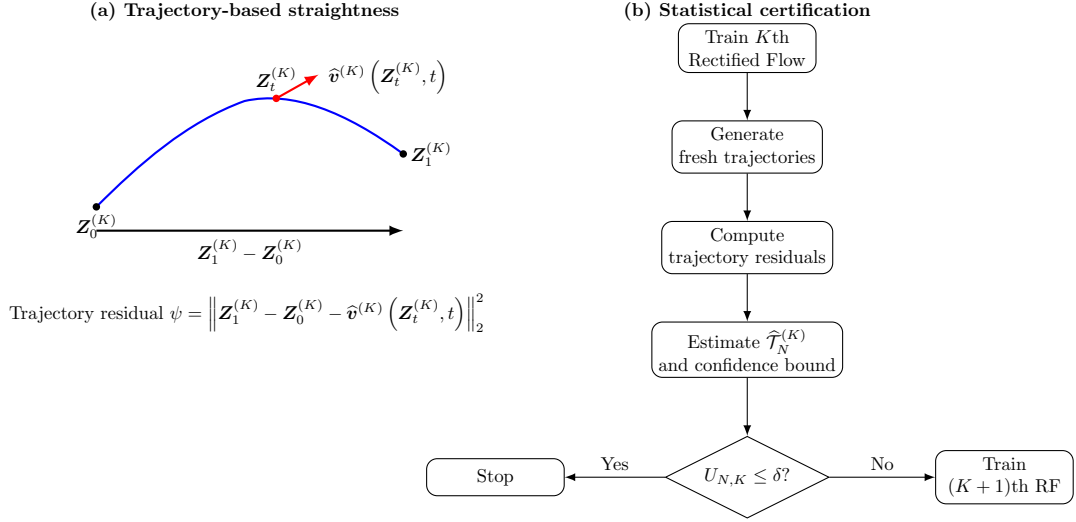

\subsection{A trajectory-based straightness functional}
\label{sec: straightness functional}
Suppose that the $K$th rectified flow has been trained, yielding the learned
velocity field $\widehat{\bm v}^{(K)}:\bbR^d\times[0,1]\rightarrow\bbR^d$.
Throughout this section, we condition on the learned velocity field
$\widehat{\bm v}^{(K)}$ during the $K$th rectification stage, and regard it as fixed. All probabilities and expectations
are therefore taken with respect to fresh trajectories generated from the
learned rectified flow. The inferential target is thus the straightness of the
\emph{learned} rectified flow rather than that generated by the unknown population-optimal
velocity field.

Figure~\ref{fig:trajectory_pipeline}(b) summarizes the proposed certification
procedure. Starting from a trained $K$th rectified flow, we generate fresh
trajectories from the learned dynamics, evaluate the proposed trajectory-based
straightness functional, construct a confidence interval for the resulting
straightness measure, and determine whether another rectification step is
statistically justified.

Starting from an initial condition
\[
\bm Z_0^{(K)}\sim \bcN\lb\bm 0,\bI_d\rb,
\]
let $\{\bm Z_{t_i}^{(K)}:0\le i\le N\}$ denote the solution to the following Euler discretized ODE:

\begin{equation}
\label{eq: learned ode}
\begin{aligned}
& 0 =t_0 <t_1< \ldots < t_N = 1, \\
\\
& \bm Z_{t_{i+1}}^{(K)}
=
\bm Z^{(K)}_{t_i} + \lb t_{i+1} - t_i\rb\widehat{\bm v}^{(K)}\lb \bm Z_{t_{i}}^{(K)},t_i\rb,
\qquad
\bm Z_0^{(K)}\sim \bcN\lb\bm 0,\bI_d\rb,\quad  i \in \{0, \ldots, N -1\}.
\end{aligned}
\end{equation}




Motivated by the geometric characterization of straight flows (\textit{cf.}~Definition \ref{def: approximate straightness}), we quantify the
remaining deviation from straightness by comparing the endpoint displacement
with the learned velocity evaluated at a uniformly sampled time along the
trajectory. Specifically, let $t\sim\mathrm{Unif}[0,1]$, independent of
$\bm Z_0^{(K)}$, and define

\begin{equation}
\label{eq: trajectory straightness functional}
\mc T^{(K)}
:=
\bbE
\left[
\norm{
\bm Z_1^{(K)}
-
\bm Z_0^{(K)}
-
\widehat{\bm v}^{(K)}
\lb
\bm Z_t^{(K)},
t
\rb
}_2^2
\right].
\end{equation}

The functional $\mc T^{(K)}$ admits a natural geometric interpretation and has also been studied in \cite{liu2023flow} and \cite{liu2022rectified} before.
It measures the variability of the velocity field along trajectories generated by
the learned rectified flow. A perfectly straight flow has constant velocity
along every trajectory, in which case these two quantities coincide almost
surely and $\mc T^{(K)}=0$. Larger values of
$\mc T^{(K)}$
indicate greater deviation from straightness.

An important advantage of the proposed functional is its computational
simplicity. Once the learned velocity field has been obtained, evaluating
$\mc T^{(K)}$ requires only generating independent trajectories from the
learned dynamics and evaluating the trained velocity field along these
trajectories. No additional optimization or retraining is required, making the
resulting certification procedure computationally lightweight relative to the
training of the rectified flow itself.

\subsection{Coupling straightness vs.\ trajectory straightness.}
We now clarify the relationship between the coupling-based
straightness defect $S^{(K)}(\mc Z)$
(Definition~\ref{def: approximate straightness}) and the
trajectory-based functional $\mc T^{(K)}$.
For simplicity, we first consider the exact population
velocity field $\bm v^{(K)}$ (\cf. Figure
\ref{fig:reflow-procedure}), rather than its learned
approximation $\widehat{\bm v}^{(K)}$. Accordingly, throughout this discussion,
$\mc T^{(K)}$ denotes the trajectory-based functional defined in
\eqref{eq: trajectory straightness functional}, with
$\widehat{\bm v}^{(K)}$ replaced by $\bm v^{(K)}$.

The functional $S^{(K)}(\mc Z)$ measures the deviation
from straightness of the endpoint coupling
$(\bm Z_0^{(K)},\bm Z_1^{(K)})$, using the linear
interpolation
\[
\sfH_t(\bm Z_0^{(K)})
=
(1-t)\bm Z_0^{(K)}+t\bm Z_1^{(K)}.
\]
In contrast, $\mc T^{(K)}$ measures the variability
of the instantaneous velocity
$\bm v_t^{(K)}(\bm Z_t^{(K)})$
along the actual ODE trajectories.
Although these functionals measure different aspects
of straightness, they are quantitatively related
under mild regularity conditions.
In particular, the following result shows that
small trajectory straightness defect guarantees
small coupling straightness defect.

\begin{thm}[Trajectory straightness controls coupling straightness]
\label{thm: trajectory controls coupling}
Let $\{\bm Z_t^{(K)}\}_{t\in[0,1]}$ be an
absolutely continuous solution of the RF ODE
\[
\frac{d}{dt}\bm Z_t^{(K)}
=
\bm v_t^{(K)}(\bm Z_t^{(K)}).
\]
Suppose that $\mc T^{(K)}<\infty$ and that the
velocity field $\bm v_t^{(K)}$ is uniformly $\ell_K$--Lipschitz in space.
Then
\begin{equation}
\label{eq: coupling trajectory comparison}
S^{(K)}(\mc Z)
\le
\left(1+\frac{\ell_K}{\pi}\right)^2
\mc T^{(K)}.
\end{equation}
In particular, $\mc T^{(K)}=0$ implies
$S^{(K)}(\mc Z)=0$.
\end{thm}


Theorem~\ref{thm: trajectory controls coupling}
provides a theoretical justification for using
$\mc T^{(K)}$ as a surrogate for coupling straightness.
In particular, controlling the trajectory-based defect
also controls the coupling-based defect, up to a
multiplicative factor determined by the spatial
regularity of the velocity field. This conforms with the results in Lemma \ref{lemma: straight coupling and flow}. More importantly, the same comparison applies to a
learned velocity field $\widehat{\bm v}^{(K)}$,
provided that the corresponding exact learned ODE
trajectories satisfy the assumptions of the theorem. The proof is deferred to Appendix~\ref{sec: proof of trajectory controls straightness}.

 Finally, we also point out that the Lipschitz condition, although strong, is a standard assumption in the existing RF literature \citep{bansal2025wasserstein, guan2026mirror, low_kumar2026flow, roy2026optimal}. In our settings, we start from a standard Gaussian distribution and the target distributions are isotropic GMMs or log-concave distributions absolutely continuous with respect to Lebesgue measure. Then, the Lipschitz condition holds in both cases. However, in general \cite{mena2025statistical} showed that the uniform Lipschitz conditions may fail for irregular distributions.

\begin{figure}[t]
    \centering
    \includegraphics[width=0.75\linewidth]{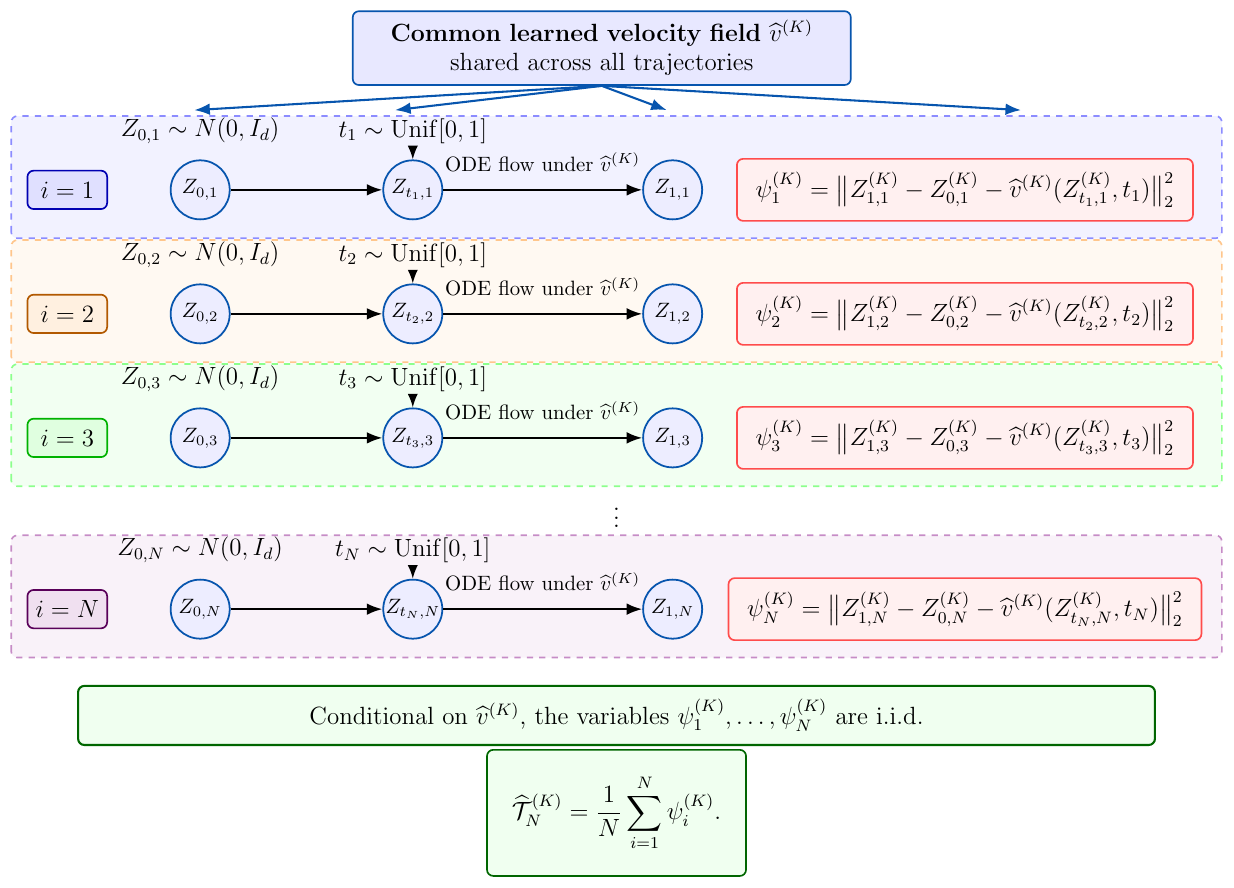}
    \caption{Construction of the trajectory functionals.}
    \label{fig: Trajectory_functionals}
\end{figure}

\subsection{Estimation and asymptotic inference}

The characterization above suggests a simple Monte Carlo estimator of
$\mc T^{(K)}$.
After training the $K$th rectified flow, we generate
$N$
independent trajectories using the discretized sampler
\eqref{eq: learned ode} starting from $N$ samples $\{\bm Z_{0,i}^{(K)}\}_{i = 1}^N \overset{i.i.d}{\sim}\bcN\lb\bm 0, \bI_d\rb$. For each trajectory, an independent time
$t_i\sim\mathrm{Unif}[0,1]$
is sampled (\cf. Figure \ref{fig: Trajectory_functionals}), and the corresponding trajectory state
$\bm Z_{t_i,i}^{(K)}$
is obtained by evaluating the numerical ODE solution at time
$t_i$.
The resulting trajectory residual is

\begin{equation}
\label{eq:trajectory_residual}
\psi_i^{(K)}
=
\norm{
\bm Z_{1,i}^{(K)}
-
\bm Z_{0,i}^{(K)}
-
\widehat{\bm v}^{(K)}
\lb
\bm Z_{t_i,i}^{(K)},
t_i
\rb
}_2^2.
\end{equation}

Since the learned velocity field is regarded as fixed, the only randomness
arises from the independently generated trajectories and the sampled times.
Consequently, conditional on $\widehat{\bm v}^{(K)}$, the residuals
$\{\psi_i^{(K)}\}_{i=1}^{N}$
are independent and identically distributed with mean
$\mc T^{(K)}$.

The trajectory-based straightness functional can therefore be estimated in a natural way by the
sample average

\begin{equation}
\label{eq:trajectory_estimator}
\widehat{\mc T}_N^{(K)}
=
\frac{1}{N}
\sum_{i=1}^{N}
\psi_i^{(K)}.
\end{equation}
We first establish consistency and asymptotic normality of the proposed estimator, which is a direct consequence of the weak law of large numbers and classical central limit theory.

\begin{thm}[Consistency and asymptotic normality]
\label{thm:trajectory_clt}

Let $\mc F_{K-1}$ be the sigma-field generated by the samples obtained via $(K-1)$th rectified flow. Then, conditioned on $\mc F_{K-1}$, we have

\[
\widehat{\mc T}_N^{(K)}
\xrightarrow{p}
\mc T^{(K)}.
\]
Additionally, if $\sigma^2_K :=\Var(\psi_i^{(K)} \mid \mc F_{K-1}) \in (0, \infty)$ almost surely, then conditional on
$\mc F_{K-1}$,

\[
\frac{\sqrt N
\lb 
\widehat{\mc T}_N^{(K)}
-
\mc T^{(K)}
\rb}{\sigma_K}
\xrightarrow{d}
\bcN\lb0,1\rb \quad \text{almost surely.}
\]
\end{thm}




The asymptotic normality property above eventually leads to the uncertainty quantification of the straightness functional. Towards that end, we begin with the natural plug-in estimator of the asymptotic variance.

\[
\widehat\sigma_{N,K}^2
=
\frac{1}{N-1}
\sum_{i=1}^{N}
\lb 
\psi_i^{(K)}
-
\widehat{\mc T}_N^{(K)}
\rb^2.
\]

The next result establishes its consistency.
\begin{prop}[Variance estimation]
\label{prop:variance}

Let the assumptions of
Theorem~\ref{thm:trajectory_clt} hold. Then, conditional on
$\mc F_{K-1}$, 

\[
\widehat\sigma_{N,K}^2
\xrightarrow{p}
\sigma_K^2.
\]

Consequently,

\[
\frac{
\widehat{\mc T}_N^{(K)}
-
\mc T^{(K)}
}
{\widehat\sigma_{N,K}/\sqrt N}
\xrightarrow{d}
\bcN\lb0,1\rb.
\]

\end{prop}


Let $z_{1-\alpha}$ denote the $(1- \alpha)$th quantile of the standard Gaussian distribution. Based on the above result, the one-sided upper confidence bound
\[
U_{N,K}
=
\widehat{\mc T}_N^{(K)}
+
z_{1-\alpha}
\frac{\widehat\sigma_{N,K}}{\sqrt N}
\]
will be used in the next subsection to construct a statistically principled
stopping rule for determining the required number of rectification steps.
\subsection{A statistically principled stopping rule}
\label{sec:straightness-stopping}

The confidence bound developed in the previous subsection naturally leads to
a statistically principled criterion for determining whether an additional
rectification step is warranted. Rather than seeking perfect straightness, we
instead aim to identify the smallest number of rectification steps for which
the learned flow is sufficiently straight.

Let $\delta>0$ denote a user-specified tolerance representing the maximum
acceptable deviation from trajectory straightness. At the $K$th rectification
step, we consider the one-sided hypothesis

\[
H_0:\;
\mc T^{(K)}>\delta,
\qquad
\text{versus}
\qquad
H_1:\;
\mc T^{(K)}\le\delta.
\]
Since $\mc T^{(K)}$ is unknown, we base the decision on the upper
one-sided confidence bound

\[
U_{N,K}
:=
\widehat{\mc T}_N^{(K)}
+
z_{1-\alpha}
\frac{\widehat\sigma_{N,K}}{\sqrt N}.
\]
The null hypothesis is rejected whenever $
U_{N,K}\le\delta$,
indicating that the learned rectified flow is sufficiently straight with
asymptotic confidence level $1-\alpha$.
\begin{prop}[Validity of the stopping criterion]
\label{prop:stopping}

Under the assumptions of
Theorem~\ref{thm:trajectory_clt}, the upper one-sided confidence bound

\[
U_{N,K}
=
\widehat{\mc T}_N^{(K)}
+
z_{1-\alpha}
\frac{\widehat\sigma_{N, K}}{\sqrt N}
\]
satisfies the coverage probability 

\[
\lim_{N \to \infty}\pr\!\lb 
\mc T^{(K)}
\le
U_{N,K}
\rb
=
1-\alpha.
\]
Consequently, the stopping rule
\[
U_{N,K}\le\delta
\]
ensures the following type-I error control:

\[
\lim_{N \to \infty}\pr_{H_0}\!\lb 
U_{N,K}\le\delta
 \rb
\le
\alpha,
\]
i.e., the proposed stopping rule falsely certifies a learned flow as
sufficiently straight with asymptotic probability at most
$\alpha$. 





\end{prop}



The proposed rule terminates the rectification procedure at the first iteration
for which the remaining trajectory straightness is certified to be no larger
than the prescribed tolerance. Consequently, the stopping decision explicitly
accounts for the sampling variability of the estimated straightness functional
, while avoiding unnecessary rectification steps. The steps are detailed in Algorithm \ref{alg:stopping} below.

\begin{remark}
Algorithm~\ref{alg:stopping} repeatedly evaluates the same hypothesis (that the remaining straightness is below $\delta$) across iterations. Using a fixed level $\alpha$ in every iteration generally does \emph{not} control the overall type-I error at $\alpha$.

If one enforces a maximum reflow budget $K_{\max}$, then family-wise type-I error can be controlled via an \emph{alpha-spending} schedule: choose (possibly adaptive, data-dependent) levels $\{\alpha_k\}_{k=1}^{K_{\max}}$ such that $\sum_{k=1}^{K_{\max}} \alpha_k \le \alpha$, and use $z_{1-\alpha_k}$ at iteration $k$. A simple choice is Bonferroni ($\alpha_k=\alpha/K_{\max}$), while geometric schedules (e.g., $\alpha_k=\alpha/2^k$) spend more alpha early.
\end{remark}

\begin{algorithm}[h]
\caption{Adaptive selection of the number of rectification steps}
\label{alg:stopping}
\begin{algorithmic}[1]
\STATE Choose confidence level $1-\alpha$ and tolerance $\delta$.
\STATE Initialize $K=1$.

\REPEAT
    \STATE Train the $K$th rectified flow and obtain $\widehat{\bm v}^{(K)}$.
    \STATE Generate independent trajectories from the learned ODE.
    \STATE Compute the trajectory residuals
    $\{\psi_i^{(K)}\}_{i=1}^N$
    using \eqref{eq:trajectory_residual}.
    \STATE Compute
    $\widehat{\mc T}_N^{(K)}$
    and
    $\widehat\sigma_{N,K}$.
    \STATE Form
    \[
    U_{N,K}
    =
    \widehat{\mc T}_N^{(K)}
    +
    z_{1-\alpha}
    \frac{\widehat\sigma_{N,K}}{\sqrt N}.
    \]
    \IF{$U_{N,K}>\delta$}
        \STATE Perform one additional rectification.
        \STATE Set $K\leftarrow K+1$.
    \ENDIF
\UNTIL{$U_{N,K}\le\delta$}

\RETURN $\widehat K=K$.
\end{algorithmic}
\end{algorithm}

\section{Numerical experiments}

\subsection{2D Gaussian mixtures experiment}
\label{sec:exp-gmm}

We first illustrate the effect of iterative rectification on a simple two-dimensional example, where the geometry of the learned trajectories can be directly visualized. We consider the source distribution
\(
p_0= \bcN(0, \bm I_2)
\)
and the symmetric two-component Gaussian mixture
\[
p_1
=
\frac12  \bcN(\bm\mu,\sigma^2 \bm I_2)
+
\frac12  \bcN(-\bm \mu,\sigma^2 \bm I_2),
\qquad
\bm \mu=(3,0)^\top,\qquad \sigma=0.5.
\]
The initial rectified flow (1-RF) is trained using independent samples
\(\bm X_0\sim p_0\) and \(\bm X_1\sim p_1\), with the usual linear interpolation
\[
\bm X_t=(1-t) \bm X_0+t \bm X_1.
\]
We then perform two additional rounds of rectification. In each round, we generate endpoint pairs by solving the ODE associated with the previously learned velocity field and train the next velocity field using the resulting coupling. 

\begin{figure}[htbp]
    \centering
    \includegraphics[width=\linewidth]{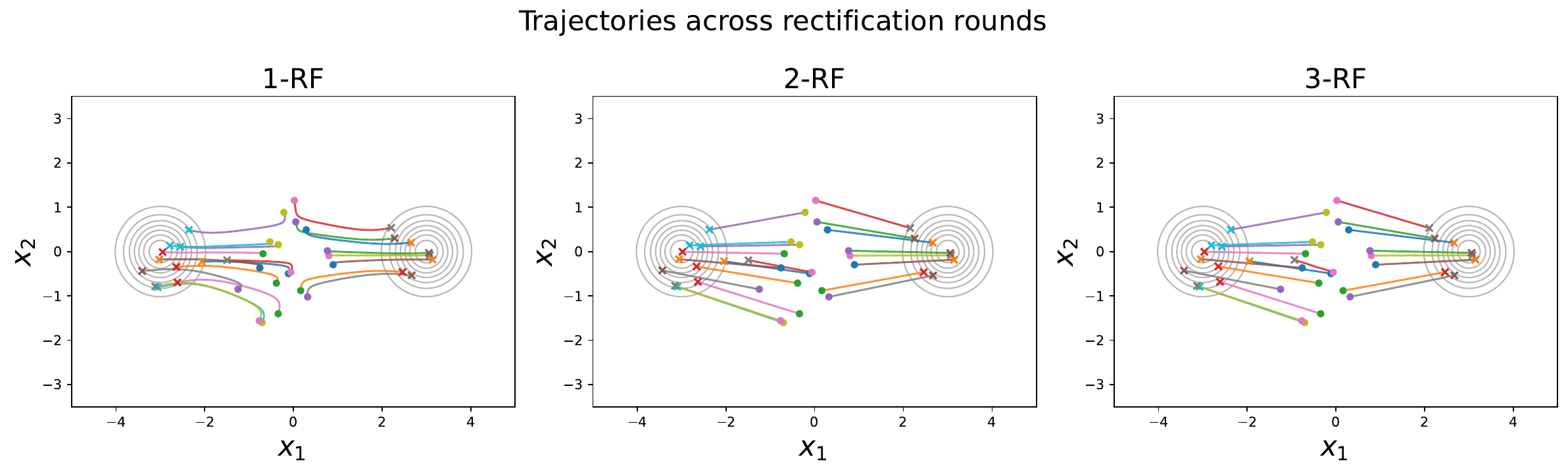}
    \caption{The figure shows that successive reflow steps straighten the trajectories for the 2D-GMM example.}
    \label{fig:gmm-trajectories}
\end{figure}

\paragraph*{Trajectory straightening}
Figure~\ref{fig:gmm-trajectories} displays the trajectories produced by the three learned flows. To facilitate comparison, we use the same initial points \(\bm Z_0\sim p_0\) in all three panels. The effect of rectification is visually pronounced. The trajectories of 1-RF exhibit substantial curvature as the mass originating near the center is separated toward the two modes of the target distribution. After one additional rectification, the trajectories become nearly linear. The third round produces only a minor additional geometric improvement.

We quantify this phenomenon using the trajectory-based straightness functional introduced in Section~\ref{sec: straightness functional}. For each \(K\), we generate \(N=5000\) fresh trajectories and compute
\[
\widehat {\mc T}_N^{(K)}
=
\frac1N
\sum_{i=1}^N
\left\|
\bm Z_{1,i}^{(K)}- \bm Z_{0,i}^{(K)}
-
\widehat {\bm v}^{(K)}
\bigl(\bm Z_{t_i,i}^{(K)},t_i\bigr)
\right\|_2^2,
\qquad
t_i\stackrel{\mathrm{i.i.d.}}{\sim}\operatorname{Unif}[0,1].
\]
We also report the corresponding standard error and the one-sided, per-round \(95\%\) upper confidence bound
\[
U_{N,K}
=
\widehat {\mc T}_N^{(K)}
+
z_{0.95}\frac{\widehat\sigma_{N,K}}{\sqrt N}.
\]

\begin{table}[t]
\centering
\caption{Straightness and generation quality for the two-dimensional Gaussian-mixture experiment. The empirical \(W_2\) distance is reported as mean $\pm$ standard error over five independent runs.}
\label{tab:gmm-straightness}
\small
\setlength{\tabcolsep}{6pt}
\renewcommand{\arraystretch}{1.15}
\begin{tabular}{lcccc}
\toprule
Method & $\widehat{\mc T}_N^{(K)}$ & $\mathrm{SE}$ & $U_{N,K}$ & Empirical $W_2$ \\
\midrule
1-RF & 1.766545 & 0.030910 & 1.817388 & $0.5319\pm0.0666$ \\
2-RF & 0.000575 & 0.000139 & 0.000804 & $0.4757\pm0.0387$ \\
3-RF & 0.000342 & 0.000100 & 0.000506 & $0.4882\pm0.0581$ \\
\bottomrule
\end{tabular}
\end{table}

The numerical results agree closely with the trajectory visualization. A single reflow step reduces the estimated straightness defect from $1.7665$
to $5.75\times10^{-4}$, 
a reduction of more than three orders of magnitude. A further rectification decreases the defect to \(3.42\times10^{-4}\), but the improvement is considerably smaller. Thus, in this example, most of the straightening effect occurs during the first reflow step, which conforms with the prior theoretical result in this paper.

\paragraph*{Stopping rule}
The confidence bounds lead to the same conclusion. In particular,
\[
U_{N,1}=1.8174,
\qquad
U_{N,2}=8.04\times10^{-4},
\qquad
U_{N,3}=5.06\times10^{-4}.
\]
Consequently, for each of the prescribed tolerances
\[
\delta\in\{0.50,0.20,0.10,0.05\},
\]
the stopping rule in Algorithm \ref{alg:stopping} rejects
\(H_0:\mc T^{(K)}>\delta\) for the first time at \(K=2\). Hence, the per-round confidence rule selects 2-RF, in agreement with the visual observation in Figure \ref{fig:gmm-trajectories}.

Finally, we examine whether the improvement in straightness comes at the expense of distributional accuracy. We compute an empirical Wasserstein-2 distance between the \(1000\) samples generated by each RF and the \(1000\) fresh samples from the target distribution, averaged over five independent runs. As shown in Table~\ref{tab:gmm-straightness}, the empirical \(W_2\) distances remain comparable across the three rounds, with 2-RF attaining the smallest average value. Thus, in this simple example, rectification dramatically straightens the trajectories without a corresponding deterioration in the quality of the generated distribution.

This observation is consistent with the population interpretation of rectification: the rectified coupling preserves the endpoint marginals while progressively simplifying the geometry of the transport. However, in finite samples, the learned velocity fields and numerical ODE solutions only approximate this idealized operation. The empirical \(W_2\) diagnostic therefore provides a complementary check that the observed reduction in the straightness defect is not merely obtained by sacrificing the target distribution.

\subsection{MNIST data experiment}
\label{sec:exp-mnist}

We next consider the MNIST dataset to study the proposed straightness criterion
in a higher-dimensional setting. The images are normalized to $[-1,1]$, so
that the target distribution is supported on $[-1,1]^{28\times28}$, while the
source distribution is
\(
p_0= \bcN(0,\bm I_{784}).
\)
We first train a rectified flow between independent samples
$\bm X_0\sim p_0$ and MNIST images $\bm X_1\sim p_1$, and subsequently perform
two additional rounds of rectification. At each round, the new RF is trained
using endpoint pairs generated from the preceding learned flow.

\begin{figure}[htbp]
    \centering
    \includegraphics[width=\linewidth]{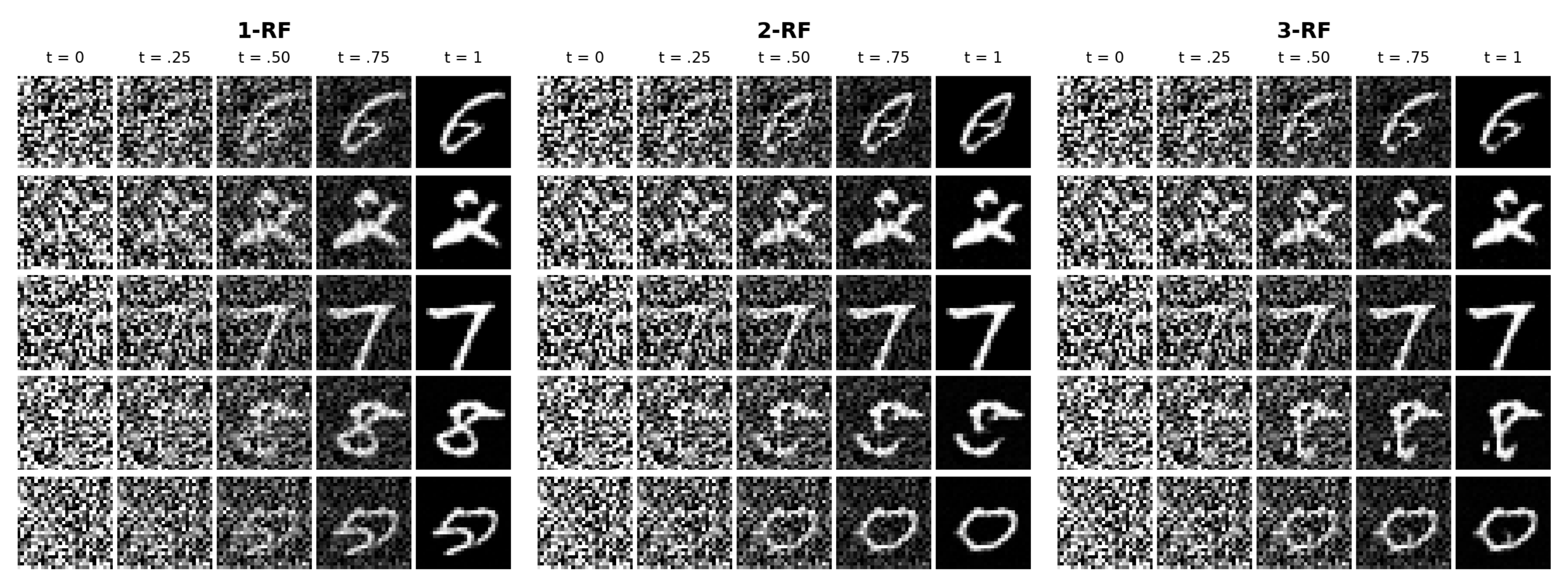}
    \caption{Evolution of the same eight initial Gaussian samples under 1-RF, 2-RF,
    and 3-RF. It is seen that successive rectification visually degrades the quality of the generated images in Row 4. For Row 1, the generation quality is visually worse for 2-RF compared to 1-RF and 3-RF.}
    \label{fig:mnist-trajectories}
\end{figure}

\paragraph*{Trajectory straightening and stopping}
Figure~\ref{fig:mnist-trajectories} shows the evolution of the same eight
initial Gaussian samples under 1-RF, 2-RF, and 3-RF. We quantify the
straightening using $N=2000$ fresh trajectories at each round. The resulting
straightness estimates, upper confidence bounds, and generation quality are
reported in Table~\ref{tab:mnist-straightness}.

\begin{table}[tbh]
\centering
\caption{Straightness and generation quality for MNIST. The MNIST feature FID is reported as mean $\pm$ standard error over ten independent runs.}
\label{tab:mnist-straightness}
\small
\setlength{\tabcolsep}{6pt}
\renewcommand{\arraystretch}{1.15}
\begin{tabular}{lcccc}
\toprule
Method & $\widehat{\mc T}_N^{(K)}$ & $\mathrm{SE}$ & $U_{N,K}$ & FID ($\downarrow$) \\
\midrule
1-RF & 56.8886 & 1.3000 & 59.0269 & $10.52\pm0.49$ \\
2-RF & 4.3075  & 0.6389 & 5.3584  & $13.40\pm0.51$ \\
3-RF & 3.2905  & 0.6393 & 4.3421  & $16.99\pm0.64$ \\
\bottomrule
\end{tabular}
\end{table}

The first reflow reduces the estimated straightness defect by approximately
$92.4\%$, from $56.89$ to $4.31$, whereas the second reflow gives a more
modest reduction to $3.29$. Thus, as in the Gaussian-mixture experiment, most
of the straightening occurs after a single reflow step.

In high dimensions, directly specifying a fixed tolerance on the total
straightness $\mc T^{(K)}$ can be difficult to achieve, since it aggregates
the squared discrepancy over all coordinates. We therefore consider the
dimension-normalized criterion
\[
\frac{\mc T^{(K)}}{d}\leq \delta_0,
\qquad\text{or equivalently}\qquad
\mc T^{(K)}\leq d\delta_0.
\]
Here $\delta_0$ represents the desired average squared straightness defect per
coordinate. Besides making the tolerance comparable across dimensions, such a
criterion can be more practically meaningful in high-dimensional generative
problems: one may only require the average coordinate-wise deviation from a
straight trajectory to be small, rather than imposing a dimension-independent
bound on the aggregate discrepancy.
For MNIST, $d=28\times28=784$ is the full dimension. Taking $\delta_0=0.01$, for example, gives
$\delta=7.84$. Since
\[
U_{N,1}=59.03>7.84,
\qquad
U_{N,2}=5.36<7.84,
\]
Algorithm~\ref{alg:stopping} stops at $K=2$. Similarly, if $\delta_0 = 6 \times 10^{-3}$, then $\delta 
\approx 4.7 > 4.34 = U_{N,3}$, and Algorithm \ref{alg:stopping} stops at $K = 3$.

\paragraph*{Generation quality}
Table~\ref{tab:mnist-straightness} also reveals an important trade-off. While successive reflow substantially improves straightness, the MNIST feature FID increases from $10.52$ for 1-RF to $13.40$ and $16.99$ for 2-RF and 3-RF, respectively. This is consistent with recent observations that recursive reflow can accumulate approximation errors and lead to distributional drift or model collapse \citep{zhu2026modelcollapse,kim2025balanced}. Several remedies have been proposed, including incorporating real data in subsequent reflow rounds \citep{zhu2026modelcollapse,kim2025balanced}. An alternative approach is optimal self-distillation \citep{roy2026optimal}, which avoids recursive replacement of endpoint pairs and instead mixes true and teacher-generated velocity targets on fixed interpolants, with the goal of improving velocity estimation and downstream generation.

Our stopping rule provides a complementary perspective: rather than modifying the training procedure for subsequent rounds, it asks whether an additional reflow is necessary at all. In the present experiment, most of the improvement in straightness is already achieved by 2-RF, whereas an additional reflow provides only a modest reduction in the straightness defect while further increasing FID. Thus, statistically certifying approximate straightness can provide a principled mechanism for avoiding unnecessary reflow and the associated accumulation of distributional error.
\section{Conclusion}

In this paper, we have developed a criterion for the straightness of the couplings obtained from applying $2$-RF to the independent couplings formed by the data and a standard Gaussian distribution. Our criterion reduces to interpretable conditions for some common classes of data distributions, such as the class of isotropic Gaussian mixture models and the class of smooth, strongly log-concave distributions.  Furthermore, we show that this criterion is indeed sharp in the worst case by developing techniques for constructing non-straight couplings. Finally, we also develop statistical tests based on a trajectory-based functional to design stopping rules for reflow procedures.

Our work opens up the possibility for understanding further statistical properties of RFs. In particular, it would be interesting to study the behavior of couplings obtained from a higher number of rectification steps. We note that understanding a higher number of rectification steps requires more tools and new ideas than what we develop in the present article. One challenge is due to the fact that the solution of the ODE in \eqref{eq: ode-true} might not be unique for higher rectification steps; see recent papers \cite{mena2025statistical} and \cite{hertrich2026relation} for more details. 

\section*{Acknowledgments}
We gratefully acknowledge NSF grants CCF-2019844, CCF-2505865, and 2217069.



\paragraph*{AI statement.} The authors use GPT 6-Astra for generating the construction of the limiting ODEs, for performing various calculations in the proof and for proof auditing. In particular, the use of Piola's identity (Lemma~\ref{ns:lem:Piota-identity}) was suggested by GPT 6-Astra as an alternative to using topological degree.  The authors take full responsibility for all content and its accuracy.

\bibliographystyle{plainnat}
\bibliography{paper-ref-url}

\clearpage
\IfFileExists{arxiv/arxiv_supp}{\input{arxiv/arxiv_supp}}{\input{arxiv_supp}}

\end{document}

%% file: arxiv_supp.tex
\section*{Appendix}

\medskip
This supplement contains all the proofs and technical results for the main article.

\renewcommand{\thefigure}{S\arabic{figure}}
\renewcommand{\thetable}{S\arabic{table}}
\renewcommand{\theHfigure}{supp.\arabic{figure}}
\renewcommand{\theHtable}{supp.\arabic{table}}
\renewcommand{\theHsubfigure}{supp.\arabic{figure}.\arabic{subfigure}}
\setcounter{figure}{0}
\setcounter{table}{0}
\setcounter{class}{0}
\setcounter{defi}{2}
\setcounter{question}{0}
\setcounter{secnumdepth}{3}
\appendix
\section{Proof of Lemma~\ref{lemma: straight coupling and flow}}
\label{sec: proof of straight flow and coupling}
We have already discussed that if $(\bm X_0, \bm X_1)$ is a straight coupling, then $\Rectflow{\bm X_0, \bm X_1}$ yields a straight flow. In fact, the flow $\bm X_t := (1-t)\bm X_0 + t \bm X_1$ is exactly the solution to the 2-RF ODE, and it has a constant velocity almost surely. 

For the other direction, we write $\{\bm Z_t\}_{t\in [0,1]} = \Rectflow{\bm X_0, \bm X_1}$ and assume that it is a straight flow. In other words, the flow has constant velocity and takes the form $\bm Z_t = (1-t)\bm Z_0 + t \bm Z_1$. Also, due to the marginal preserving property of RF, we have $\Law(\bm Z_t) = \Law(\bm X_t)$. Collecting all of these facts, we can conclude

\begin{align*}
\mathbb E\|\bm Z_t\|^2
&=(1-t)\mathbb E\|\bm Z_0\|^2
+t\mathbb E\|\bm Z_1\|^2
-t(1-t)\mathbb E\|\bm Z_1-\bm Z_0\|^2,
\end{align*}
and 
\begin{align*}
\mathbb E\|\bm X_t\|^2
&=(1-t)\mathbb E\|\bm X_0\|^2
+t\mathbb E\|\bm X_1\|^2-t(1-t)\mathbb E\|\bm X_1-\bm X_0\|^2.
\end{align*}
This shows that
$\bbE\norm{\bm X_1-\bm X_0}_2^2
=\bbE\norm{\bm Z_1-\bm Z_0}_2^2$.
By Jensen's inequality, marginal preservation, and the
conditional variance identity,
\[
\begin{aligned}
0
&\leq
\int_0^1
\bbE\Big\|
    \bm X_1-\bm X_0
    -\bbE[\bm X_1-\bm X_0\mid\bm X_t]
\Big\|_2^2\,dt
\leq
\bbE\norm{\bm X_1-\bm X_0}_2^2
-\bbE\norm{\bm Z_1-\bm Z_0}_2^2
=0.
\end{aligned}
\]
Since the integrand is nonnegative, it follows that
\[
\bbE[\bm X_1-\bm X_0\mid\bm X_t]
=\bm X_1-\bm X_0
\qquad\text{almost surely}
\]
for almost every $t\in(0,1)$.
Thus $(\bm X_0,\bm X_1)$ is a straight coupling,
which completes the proof. $\hfill$ $\square$

\section{Proof of the results in Section~\ref{sec: a sufficient condition for straightness}}

  Let us start with some useful connections between the velocity $\bm v_t\lb \bm x\rb$ and the score function  
$$\bm s_t\lb \bm x\rb:= \nabla \log p_t\lb \bm x\rb$$
{ where $p_t$ is defined as in \eqref{eq:p_t}.}

 For $t\in(0,1)$, Equation \eqref{eq: drift-true} can be rewritten as
\begin{equation}
\label{eq: v_tweedie}
\bm v_t\lb \bm x\rb = \bbE[\bm X_1 - \bm X_0 \mid \bm X_t = \bm x] = \frac{1}{1-t}\Big( \bbE[\bm X_1 \mid \bm X_t = \bm x] - \bm x \Big) 
= \frac{1}{1-t} \cdot \left[  \frac{\bm x}{t} + \frac{\lb 1-t\rb^2}{t} \cdot \bm s_t\lb \bm x\rb - \bm x \right].
\end{equation}
Indeed, the above follows from first-order Tweedie's formula \citep{efron2011tweedie, meng2021estimating}, which guarantees that for any choice of $t \in (0,1)$ and $\bm x \in \bbR^d$,
\[
\bbE[\bm X_1 \mid \bm X_t = \bm x]=   \frac{\bm x}{t} + \frac{\lb 1-t\rb^2}{t} \cdot \bm s_t\lb \bm x\rb,\] 
Also, the second-order Tweedie's formula gives 
\[
\Cov_{\mu_{\bm x}}\lb \bm X_1
\rb = \frac{\lb 1-t\rb^2}{t^2} \cdot \bI_d + \frac{\lb 1-t\rb^4}{t^2} \cdot  D_{\bm x} \left[ \bm s_t\lb \bm x\rb \right]
\]
where $\mu_{\bm x}:= \mbox{Law} \lb \bm X_1 \Big| \bm X_t = \bm x \rb$. 

Substituting this in \eqref{eq: v_tweedie}, we obtain that 
\begin{equation}\label{eq:tweedie2}
\bm v_t\lb \bm x\rb = \frac{\bm x}{t} + \frac{\lb 1-t\rb}{t} \cdot  \bm s_t\lb \bm x\rb, \qquad
\text{and} \qquad 
D_{\bm x} \left[ \bm v_t\lb \bm x\rb \right] = - \frac{\bI_d}{1-t} + \frac{t}{\lb 1-t\rb^3} \cdot \Cov_{\mu_{\bm x}}\lb\bm X_1\rb.
\end{equation}
  The reformulation of $\bm v_t$ essentially shows that estimating $\bm v_t$ is equivalent to estimating the score $\bm s_t$ for $t \in \lb 0,1\rb$.

\subsection{Injectivity of terminal flow map}
\label{sec: proof of injectivity of terminal flow map}

In this subsection, we establish a basic injectivity property of the terminal
flow map associated with an ODE. Specifically, given a (possibly time-dependent)
velocity field $\bm v:[0,1]\times\bbR^d\to\bbR^d$, we write
$\bm v_t(\bm x):=\bm v(t,\bm x)$ and let $\sfZ_t(\bm z)$ denote the solution at
time $t$ to the initial value problem starting from $\bm z$. Under mild local
integrability and spatial regularity assumptions ensuring local uniqueness of
solutions, we show that the endpoint map $\sfZ_1:\bbR^d\to\bbR^d$ is
injective.

\begin{lemma}[Injectivity of the endpoint map]
\label{lemma: endpoint injectivity}
Let $\bm v:[0,1]\times\bbR^d\to\bbR^d$ be Borel measurable
and continuously differentiable in space for almost every $t$.
Suppose that, for every $R>0$,
\[
    \int_0^1 \left(
        \|\bm v_t(\bm 0)\|_2
        + \sup_{\Vert\bm z\Vert_2\le R}
          \|D\bm v_t(\bm z)\|_{\mathrm{op}}
    \right)\,dt<\infty.
\]
Assume that, for every $\bm z\in\bbR^d$, 
the ODE
\[
    \frac{d}{dt}\sfZ_t(\bm z)
    =\bm v_t\bigl(\sfZ_t(\bm z)\bigr),
    \qquad \sfZ_0(\bm z)=\bm z,
\]
admits an absolutely continuous solution on $[0,1]$.
Then $\sfZ_1:\bbR^d\to\bbR^d$ is injective.
\end{lemma}

\begin{proof}
We begin by recalling the mean-value theorem for vector-valued functions:
\[
\bm v_t(\bm x) - \bm v_t(\bm y)
=
\left[
\int_0^1 D\bm v_t\bigl(\bm y+\xi(\bm x-\bm y)\bigr)\,d\xi
\right](\bm x-\bm y).
\]
Then, for a fixed $\bm z$ with $\Vert \bm z \Vert_2\le R$, the mean-value theorem, by setting $\bm x =
\bm z$ and $\bm y = \bm 0$, leads to
\[
\|\bm v_t(\bm z) - \bm v_t(\bm 0)\|_2
\le \|\bm z\|_2\sup_{\xi \in [0,1]}
       \|D\bm v_t(\xi\bm z)\|_{\mathrm{op}}.
\]
Therefore, a simple rearrangement of terms yields
\[
    \sup_{\|\bm z\|_2\le R}\|\bm v_t(\bm z)\|_2
    \le \|\bm v_t(\bm 0)\|_2
    +R\sup_{\|\bm z\|_2\le R}
       \|D\bm v_t(\bm z)\|_{\mathrm{op}},
    \qquad
    \operatorname{Lip}(\bm v_t;\{\|\bm z\|_2\le R\})
    \le \sup_{\|\bm z\|_2\le R}
       \|D\bm v_t(\bm z)\|_{\mathrm{op}},
\]
where $\operatorname{Lip}(\bm v_t;\{\|\bm z\|_2\le R\})$ denotes the uniform Lipschitz constant of $\bm v_t(\cdot)$ in the ball $\{\bm z: \Vert \bm z \Vert_2\le R\}$.
Hence, the local integrability hypothesis of
\citet[Lemma~8.1.4]{ambrosio2005gradient} holds.
In particular, local uniqueness holds through every
$(s,\bm x)\in[0,1]\times\bbR^d$, on a relatively open
time interval $I(s, \bm x)\ni s$.

Suppose $\sfZ_1(\bm x)=\sfZ_1(\bm y)$, and define
\[
    A:=\{t\in[0,1]:\sfZ_t(\bm x)=\sfZ_t(\bm y)\}.
\]
The continuity of the trajectories implies that $A$ is
relatively closed, and $1\in A$.
For any $s\in A$, the two trajectories solve the same
initial-value problem at time $s$. Local uniqueness therefore
gives an $\varepsilon>0$ such that
\[
    (s-\varepsilon,s+\varepsilon)\cap[0,1]\subseteq A.
\]
Thus $A$ is also relatively open. Since $[0,1]$ is connected, and $A$ is a non-empty \textit{clopen} subset of $[0,1]$, we can conclude
$A=[0,1]$. In particular,
\[
    \bm x=\sfZ_0(\bm x)=\sfZ_0(\bm y)=\bm y,
\]
which proves injectivity.
\end{proof}
\subsection{Proof of Lemma~\ref{lemma: injectivity lemma} and Theorem~\ref{theorem: straightness theorem}} \label{sec:proof-injectivity+straightness}

\noindent \textbf{Proof of Lemma~\ref{lemma: injectivity lemma}.}
Injectivity is obvious. If $\sfH_t$ is continuous and injective, then its inverse on its image is measurable. Hence the displacement
$\sfZ_1(\bm Z_0)-\bm Z_0$ is measurable with respect to
$\sfH_t(\bm Z_0)$, and most importantly the $\sigma$--field generated by $\sfH_t(\bm Z_0)$ coincides with that of $\bm Z_0$. Therefore, it follows that
\[
\bbE\!\left[
\sfZ_1(\bm Z_0)-\bm Z_0
\mid t, \sfH_t(\bm Z_0)
\right] = \bbE\!\left[
\sfZ_1(\bm Z_0)-\bm Z_0
\mid \bm Z_0
\right]  = 
\sfZ_1(\bm Z_0)-\bm Z_0
\quad\text{\it almost surely}.
\]
This is precisely the straightness property in
Definition~\ref{def: Qiang straightness}. $\hfill$ $\square$



Before proving the theorem, we state the following elementary bound.

    \begin{lemma} \label{lemma: eigen variance lemma}
Let $\bA$ be a symmetric matrix and $\bm{u}$ be a unit vector. Then, we have $\bm{u}^\top \bA^2 \bm{u} - (\bm{u}^\top \bA \bm{u})^2 \le \frac{\br{\lambda_{\max}(\bA) - \lambda_{\min}(\bA)}^2}{4}$.
\end{lemma}

\noindent \textbf{Proof of Lemma \ref{lemma: eigen variance lemma}.}
    Let us begin with the spectral form of $\bA$, i.e., $\bA =\sum_{j \in [d]} \lambda_j \bm{p}_j \bm{p}_j^\top$, where $\bc{\bm{p}_j}_{j \in [d]}$ is the orthonormal set of eigenvectors and $\lambda_1 \ge \ldots \ge \lambda_d$ are the corresponding eigenvalues. Let us define $\alpha_j = (\bm{u}^\top \bm{p}_j)^2$, and note that $\sum_{j \in [d]} \alpha_j = \bm{u}^\top (\sum_{j \in [d]} \bm{p}_j \bm{p}_j^\top) \bm{u}  = \bm{u}^\top \bm{u} = 1$.
    Using the above spectral form, we have 
    \begin{align*}
        \bm{u}^\top \bA^2 \bm{u} - (\bm{u}^\top \bA \bm{u})^2 & = \sum_{j \in [d]} \lambda_j^2 (\bm{u}^\top \bm{p}_j)^2  - \br{\sum_{j \in [d]} \lambda_j (\bm{u}^\top \bm{p}_j)^2}^2\\
        &  = \Var_\balpha(\{\lambda_j\}_{j \in [d]})\\
        & \le \frac{\br{\lambda_{\max}(\bA) - \lambda_{\min}(\bA)}^2}{4},
    \end{align*}
    where the last inequality follows from the well-known Popoviciu inequality \citep{popoviciu1935equations, bhatia2012some}. $\hfill$ $\square$

    We are now ready to prove Theorem~\ref{theorem: straightness theorem}.

\noindent \textbf{Proof of Theorem~\ref{theorem: straightness theorem}.} 
For notational brevity, we define $\bA_t := \bA_t(\bm x, \bm y)$ and $\bm{u}_t:= \bm{\delta}_t(\bm x, \bm y)$ for a fixed pair of distinct points $\bm x, \bm y$. Then, we recall ODE \eqref{eq: ode delta_t}:
\begin{equation}
    \label{eq: u_t ODE}
    \frac{d \bm{u}_t}{dt} = \bA_t \bm{u}_t, \quad \bm{u}_0 = \bm{\delta}_0(\bm x, \bm y) = \bm y - \bm x.
\end{equation}
The main strategy of the proof is to track the evolution of the angle $\theta_t^{\bm{u}} := \Angle(\bm{u}_t, \bm{u}_0)$. 
\paragraph{Assuming Condition (a):}
If we assume that $[\bA_t, \bA_s] = 0$, then using standard theory \citep{magnus1954exponential}, it follows that 
\[
\bm{u}_t = \underbrace{\exp \lb \int_0^t \bA_s \; ds\rb}_{ \succ \bm{0}} \bm{u}_0.
\]
Therefore, $\theta_1^{\bm{u}} < \pi/2 < \pi$. Then the proof concludes by using Lemma~\ref{lemma: injectivity lemma}.

\paragraph{Assuming Condition (b):}

We begin from the basic inequality, valid for all $t \in [0,1]$,
\begin{align} \label{eq:ode-norm-u}
    \frac{d \norm{\bm{u}_t}_2^2}{dt} & = 2 \bm{u}_t^\top \dot{\bm{u}}_t = 2 \bm{u}_t^\top \bA_t \bm{u}_t \ge 2\lambda_{\min}(\bA_t)\norm{\bm{u}_t}_2^2,
\end{align}
which, by Gronwall's lemma, yields that $\norm{\bm{u}_t}_2^2 \ge \norm{\bm{u}_0}_2^2 \exp \br{2 \int_0^t \lambda_{\min}(\bA_s) ds}>0$. This and \eqref{eq: u_t ODE} ensure that $t \mapsto \norm{\bm{u}_t}_2$ is a differentiable map over $[0,1]$. 
Now, we define the unit vector $\bm{n}_t:= \bm{u}_t/\norm{\bm{u}_t}_2$, so we can trivially write $ \bm{n}_t \norm{\bm{u}_t}_2 = \bm{u}_t$. 
Differentiating both sides of this identity and using \eqref{eq: u_t ODE}, we obtain
\begin{align*}
\dot{\bm{u}}_t  = \bA_t \bm{u}_t & =   \dot{\bm{n}}_t \norm{\bm{u}_t}_2 + \bm{n}_t \ddt \norm{\bm{u}_t}_2\\
& = \dot{\bm{n}}_t \norm{\bm{u}_t}_2 + \bm{n}_t \times \frac{1}{2 \norm{\bm{u}_t}_2} \times 2\innerprod{\dot{\bm{u}}_t, \bm{u}_t}\\
& =   \dot{\bm{n}}_t \norm{\bm{u}_t}_2 + \bm{n}_t \innerprod{\bA_t \bm{u}_t, \bm{n}_t}. 
\end{align*}
Dividing both sides by $\norm{\bm{u}_t}_2 > 0$ yields 
\[
\bA_t \bm{n}_t = \dot{\bm{n}}_t + \bm{n}_t \bm{n}_t^\top \bA_t \bm{n}_t
\]
and, rearranging, 
\begin{equation}\label{eq:n_t-ode}
 \dot{\bm{n}}_t = (\bI_d - \bm{n}_t \bm{n}_t^\top)\bA_t \bm{n}_t. 
\end{equation}
Next, the endpoint angle $\theta_t^{\bm{u}}$ is bounded above by the length of the normalized trajectory on the unit sphere; that is,
\[
\theta_t^{\bm{u}} \le \int_0^t \norm{\dot{\bm{n}}_s}_2 ds.
\]

Therefore, using the ODE \eqref{eq:n_t-ode},
\begin{equation}
\label{eq: bound on theta_t}
\begin{aligned}
    \theta_t^{\bm{u}}  &  \le \int_0^t \norm{(\bI_d - \bm{n}_s\bm{n}_s^\top) \bA_s \bm{n}_s}_2 \; ds\\
    & = \int_0^t \sqrt{\norm{(\bI_d - \bm{n}_s \bm{n}_s^\top) \bA_s \bm{n}_s}_2^2} \; ds\\
    & = \int_0^t \bc{ \bm{n}_s^\top \bA_s^2 \bm{n}_s - \br{\bm{n}_s^\top \bA_s \bm{n}_s}^2}^{1/2}\; ds\\
    & \le \int_{0}^t \sqrt{\frac{(\lambda_{\max}(\bA_s) - \lambda_{\min}(\bA_s))^2}{4}} \; ds\\
    & = \int_{0}^t\frac{\lambda_{\max}(\bA_s) - \lambda_{\min}(\bA_s)}{2}\; ds,
\end{aligned}
\end{equation}
where the inequality is due to Lemma \ref{lemma: eigen variance lemma}.
The claimed result under condition (b) follows from Lemma~\ref{lemma: injectivity lemma}. $\hfill$ $\square$

\subsection{Proofs of Theorems 3.2 and 3.3}

Before proving Theorem~\ref{2RF straight-logconcave}, we collect a useful Jacobian estimate. 
    \begin{lemma} \label{lem:Jacobian-bound-logconcave}
    Recall $\bm v( \bm x,t)$ in \eqref{eq: v_tweedie}. Then, for all $t \in [0,1]$, it holds that 
    \begin{align} \label{Jacobian est}
      \frac{t-L(1-t)}{t^2+L(1-t)^2}\cdot \bI_d \preceq  D_{\bm{x}} \bm{v}(\bm{x},t) \preceq       \frac{t-m(1-t)}{t^2+m(1-t)^2}\cdot \bI_d
    \end{align}
    for all $\bm{x} \in \mathbb{R}^d$.
\end{lemma}

\noindent \textbf{Proof of Lemma \ref{lem:Jacobian-bound-logconcave}.} Recall that $\bm v$ has the exact formula
\begin{align} \label{v formula}
\bm{v}(\bm{x},t) = \nabla_{\bm{x}} \left( \frac{\| \bm{x} \|^2}{2t} + \frac{1-t}{t} \log p_t(\bm{x}) \right)
\end{align}
where 
\[
p_t(\bm{x}) \propto  \int_{\mathbb{R}^d} \exp \left[ -\frac{\| \bm{x}- t\bm{y} \|^2}{2(1-t)^2} - V(\bm{y}) \right] d \bm{y}.
\]
From the expression above, we find that
\begin{align} \label{log p-t}
    \nabla^2_{\bm{x}} \log p_t(\bm{x}) = - \frac{1}{(1-t)^2} \bI_d + \frac{t^2}{(1-t)^4} \mbox{Cov}_{\mu_{\bm{x}}} \left( \bm{Y} \right)   
\end{align}
where $\bm{Y} \sim \mu_{\bm{x}}$ and $\mu_{\bm{x}}$ is the measure 
\[
\mu_{\bm{x}}(d\bm{y}) \propto  \exp \left[ -V(\bm{y}) -  \frac{\| \bm{x}- t\bm{y} \|^2}{2(1-t)^2}  \right] d\bm{y}. 
\]
It is easy to see that for every $\bm{x}$, the measure $\mu_{\bm{x}}(d\bm{y})$ is a strongly log-concave distribution, and 
\[
\nabla^2_{\bm{y}} \left[ V(\bm{y}) + \frac{\| \bm{x}- t\bm{y} \|^2}{2(1-t)^2}  \right] = \nabla^2 V(\bm{y}) + \frac{t^2}{(1-t)^2} \bI_d. 
\]
Therefore, the Brascamp-Lieb inequality and the Cram\'er--Rao lower bound (see Lemma 7 in \cite{chewi2023entropic}) give
\begin{align*}
 \left[  \frac{(1-t)^2}{t^2+L(1-t)^2} \right] \cdot \bI_d  \preceq \mbox{Cov}_{\mu_{\bm{x}}} \left( \bm{Y} \right)  \preceq \left[  \frac{(1-t)^2}{t^2+m(1-t)^2} \right] \cdot \bI_d.
\end{align*}
The proof is completed by plugging the above expression into \eqref{log p-t} and then \eqref{v formula}. 
$\hfill$ $\square$

\noindent \textbf{Proof of Theorem~\ref{2RF straight-logconcave}.} Using Lemma \ref{lem:Jacobian-bound-logconcave}, we deduce that 
\begin{align*}
 \Angle(\bm{\delta}_1(\bm{x},\bm{y}), \bm{\delta}_0(\bm{x},\bm{y})) &\leq \frac{1}{2} \left[ \int_0^1 \frac{s-m(1-s)}{s^2 +m(1-s)^2} ds  - \int_0^1 \frac{s-L(1-s)}{s^2 +L(1-s)^2} ds \right] = \frac{1}{4} \log \left( \frac{L}{m} \right).
\end{align*}
As a result, Condition (b) in Theorem~\ref{theorem: straightness theorem} is verified provided $L/m < \exp(4\pi)$. $\hfill$ $\square$

\noindent \textbf{Proof of Theorem~\ref{theorem: straightnes GMM}.} In this setting the score takes the form
\[
\bm{s}_t(\bm{x}) = \sum_{i=1}^K \pi_{i, t}(\bm{x}) \br{\frac{t\bm{\mu}_i-\bm{x}}{\sigma^2_{t}}},\]
where $\sigma^2_{t} = (1-t)^2 + t^2\sigma^2$ and 

\[\pi_{i, t}(\bm{x}) = \frac{\pi_i \exp\br{\frac{-\norm{\bm{x}-t\bm{\mu}_i}^2}{2\sigma_{t}^2}}}{\sum_{j=1}^K \pi_j\exp\br{\frac{-\norm{\bm{x}-t\bm{\mu}_j}^2}{2\sigma_{t}^2}}}.\]
Substituting in \eqref{eq:tweedie2}, we  get 
\begin{equation}
    \label{eq: velocity equation}
    \bm{v}_t(\bm{x}) = \frac{(1+\sigma^2)t -1}{\sigma_t^2} \bm{x} + \frac{1-t}{\sigma_t^2} \sum_{i \in [K]} \pi_{i,t}(\bm{x}) \bm{\mu}_i,
\end{equation}

\begin{equation}
    \label{eq: grad_velocity}
    D_{\bm{x}} \bm{v}_t(\bm{x}) = \frac{(1+\sigma^2)t-1}{\sigma_t^2} \bI_d +  \frac{t(1-t)}{2 \sigma_t^4} \underbrace{\sum_{i \ne j} \pi_{i,t}(\bm{x}) \pi_{j,t}(\bm{x}) (\bm{\mu}_i - \bm{\mu}_j) (\bm{\mu}_i - \bm{\mu}_j)^\top}_{\bB_t(\bm{x})}.
\end{equation}
Next, note that $\bB_t(\bm{x})$ in \eqref{eq: grad_velocity} is positive semi-definite uniformly over $\bm{x} \in \bbR^d$. Therefore
\begin{align*}
    \lambda_{\max} \lb D_{\bm x} \bm{v}_t(\bm{x}) \rb - \lambda_{\min} \lb D_{\bm x} \bm{v}_t(\bm{x}) \rb \le \frac{t(1-t)}{\sigma_t^4} \sup_{\bm{\alpha} : \norm{\bm{\alpha}}_2 =1}\Var_{\bpi_t}(\bc{\bm{\alpha}^\top \bm{\mu}_i}_{i \in [K]}) \le \frac{t(1-t)}{4 \sigma_t^4} \Dmax^2.
\end{align*}
Using the identity
 $$\int_0^1 \frac{t(1-t)}{\sigma_t^4}dt = \frac{1}{2 \sigma^2},$$ 
  Condition (b) in Theorem~\ref{theorem: straightness theorem} holds as long as $\Dmax^2 < 16 \pi \sigma^2$. The result then follows from Theorem~\ref{theorem: straightness theorem}. $\hfill$ $\square$

\section{Proof of Proposition~\ref{ns:prop:criterion}}
Before proving Proposition~\ref{ns:prop:criterion}, we need a technical result.  We assume that the endpoint map $\sfZ_1$ is $C^1$ and that the data's law has a strictly positive density with finite second moment. 
\begin{lemma}\label{ns:lem:degree-defect}
Suppose that the endpoint map has the form
\begin{equation}
\label{ns:eq:T-shape}
    \sfZ_1( \bm{x} )=\bB  \bm{x} + \bm{R} ( \bm{x} ),
    \qquad \bB \in\Sp_{++}^3,
\end{equation}
where $\sup_{ \bm{x} }\norm{ \bm{R} ( \bm{x} )}<\infty $.
Recall the chord map $\sfH_t$ and the straightness defect $S_t$ in \eqref{eq:chord-map} and \eqref{ns:eq:defect-as-chord}, respectively. Assume that the law of $\bm X$ admits a strictly positive density. If $\det D\sfH_t ( \bm{x} _*)<0$ for some $ \bm{x} _*$ and some
$t\in(0,1)$, then $S_t>0$. 
\end{lemma}

\noindent \textbf{Proof of Lemma~\ref{ns:lem:degree-defect}.}
All the diffeomorphisms in the proof below are understood as $C^1$. The argument below is based on topological degree, although our proof is elementary and does not make use of such a concept. The argument below is adapted from \cite{simader2005homotopy}.  Write

$$
 \bm{F} :=\sfH_t ,
\qquad
\bA _t:=(1-t)\bI _3+t\bB ,
\qquad
 \bm{q} _t:=t \bm{R} ,
$$
so that
\begin{align} \label{eq:F-requirement}
 \bm{F} ( \bm{x} )=\bA _t \bm{x} + \bm{q} _t( \bm{x} ),
\qquad
\bA _t\in\Sp_{++}^3,
\qquad
\sup_{ \bm{x} }\norm{ \bm{q} _t( \bm{x} )}<\infty.
\end{align}

\underline{\it Step 1: Signed change of variables.} We first prove the signed change-of-variables 
\begin{equation}
\label{ns:eq:signed-cov}
\int_{\bbR ^3}
\varphi( \bm{F} ( \bm{x} ))
\cdot \det D \bm{F} ( \bm{x} )\dd\bm{x} 
=
\int_{\bbR ^3}\varphi( \bm{z} )\dd\bm{z} ,
\qquad \text{for all} \qquad 
\varphi\in C_c^\infty(\bbR ^3),
\end{equation}
and for all $\bm F$ satisfying \eqref{eq:F-requirement}.

 Let us pause for a moment before proving \eqref{ns:eq:signed-cov} and explain what it means. If $\bm F$ is a diffeomorphism on $\mb R^3$, then the usual change of variable formula (Theorem 2.47, section 2.6 in \cite{folland1999real}) gives
\[
\int_{\bbR ^3}
\varphi( \bm{F} ( \bm{x} ))
\cdot \left| \det D \bm{F} ( \bm{x} ) \right| \dd\bm{x}
=
\int_{\bbR ^3}\varphi( \bm{z} )\dd\bm{z} ,
\qquad \text{for all} \qquad 
\varphi\in C_c^\infty(\bbR ^3).
\]
The above display differs from \eqref{ns:eq:signed-cov} in the  absolute value sign, and that it requires $\bm F$ to be a $C^1$ diffeomorphism. The point of \eqref{ns:eq:signed-cov} is that the identity holds without the absolute value sign and the diffeomorphism requirement on $\bm F$. We instead assume only \eqref{eq:F-requirement}.

To prove \eqref{ns:eq:signed-cov}, keep $t \in (0,1)$ fixed and define the interpolation
$$
 \bm{F} _s( \bm{x} )
=
\bA _t \bm{x} +s \bm{q} _t( \bm{x} ),
\qquad 0\leq s\leq1,
$$
and set
$$
I(s) = \int_{\bbR ^3} \varphi( \bm{F} _s( \bm{x} )) \det D \bm{F} _s( \bm{x} )\,d \bm{x} .
$$
Because $ \bm{q} _t$ is bounded and $\bA _t$ is invertible, the
integrands are supported in one fixed bounded set, uniformly in $s$. Thus $I(s)<\infty$ for all $s \in [0,1]$. On the other hand, Lemma \ref{ns:lem:Piota-identity} gives
$$
\frac{d}{ds}
\Big[
\varphi( \bm{F} _s)\det D \bm{F} _s
\Big]
=
\operatorname{div} \left[ \varphi( \bm{F} _s) \operatorname{cof}(D \bm{F} _s)^{\top}  \bm{q} _t \right].
$$
Hence 
\[
I'(s)= \int_{\bbR ^3} \frac{d}{ds} \Big[ \varphi( \bm{F} _s)\det D \bm{F} _s \Big] \, d  \bm{x}  = \int_{\bbR ^3} 
\operatorname{div} \left[ \varphi( \bm{F} _s) \operatorname{cof}(D \bm{F} _s)^{\top}  \bm{q} _t \right] d  \bm{x} =0.
\]

 Let us supply more details on the last equality. The function $\bm G(\bm x):=  \varphi( \bm{F} _s) \operatorname{cof}(D \bm{F} _s)^{\top}  \bm{q} _t$ inside the divergence has compact support because $\varphi$ has compact support and 
\[
\| \bm F_s(\bm x) \| \geq \lambda_{\rm min} \lb \bm A_t \rb \cdot \| \bm x \| - \| \bm q_t \|_\infty \to \infty, \qquad 
\text{as} \qquad  \| \bm x\| \to \infty.
\]
Therefore, one can pick a ball $B \lb \bm 0, R \rb$ that contains the support of $\bm G$ and use integration by parts to obtain
\[
\int_{\mb R^3} \mbox{div} \lb G(\bm x) \rb \, d\bm x = \int_{ \la \bm x \in \mb R^3: \| \bm x \|= R \ra} G \cdot \bm n \, dS = 0
\]

Thus $I'(s)=0$ for all $s \in [0,1]$. Consequently,
$$
I(1)=I(0)
=
\int_{\bbR ^3}
\varphi(\bA _t \bm{x} )
\det\bA _t\,d \bm{x} 
=
\int_{\bbR ^3}\varphi( \bm{z} )\,d \bm{z} ,
$$
which proves \eqref{ns:eq:signed-cov}. 

Note that the argument above implicitly uses that $q_t$ is $C^2$ (so that Lemma \ref{ns:lem:Piota-identity} applies). This does not affect the conclusion in \eqref{ns:eq:signed-cov} even if $q_t $ is only $C^1$ since one can smooth out $q_t$ by taking the convolution with a sequence of smooth mollifiers, see the argument in Lemma \ref{ns:lem:Piota-identity}. 

\underline{\it Step 2: Non-injectivity of $ \bm{F} $.} Since $\det D \bm{F} ( \bm{x} _*)<0$, the inverse function theorem gives the existence of open sets
$U\ni \bm{x} _*$ and $O\ni \bm{F} ( \bm{x} _*)$ such that

$$
 \bm{F} :U\longrightarrow O \qquad \text{is a diffeomorphism},
$$
and
$$
\det D \bm{F} <0 \qquad\text{on }U.
$$
Choose $\varphi\in C_c^\infty(O)$ with $\varphi\geq0$ and
$\varphi\not\equiv0$. The usual change of variables formula gives
$$
\int_U
\varphi( \bm{F} ( \bm{x} ))
\det D \bm{F} ( \bm{x} )\,d \bm{x} 
= - \int_U
\varphi( \bm{F} ( \bm{x} ))
\left| \det D \bm{F} ( \bm{x} ) \right|\,d \bm{x} 
=
-\int_O\varphi( \bm{z} )\,d \bm{z} .
$$
Combining this with \eqref{ns:eq:signed-cov},
$$
\int_{\bbR ^3\setminus U}
\varphi( \bm{F} ( \bm{x} ))
\det D \bm{F} ( \bm{x} )\,d \bm{x} 
=
2\int_O\varphi( \bm{z} )\,d \bm{z} 
>0.
$$
Consequently, there exists $ \bm{x} _+\notin U$ such that

$$
\varphi( \bm{F} ( \bm{x} _+))>0,
\qquad
\det D \bm{F} ( \bm{x} _+)>0
$$
and $ \bm{z} _0:= \bm{F} ( \bm{x} _+)\in O$.
Next, let
$$
 \bm{x} _-:=( \bm{F} |_U)^{-1}( \bm{z} _0),
$$
which is well defined since $\bm{F} |_U$ is a diffeomorphism into $O$.
Observe that 
$$
 \bm{x} _-\neq \bm{x} _+,
\qquad
 \bm{F} ( \bm{x} _-)= \bm{F} ( \bm{x} _+)= \bm{z} _0,
$$
and both Jacobians are invertible. The inverse
function theorem once again gives, after possibly shrinking the neighborhoods, three open sets $U_1, U_2, O_1$ such that
\[
U_1 \subset U, \qquad U_2 \subset \bbR ^3\setminus U, \qquad O_1 \subset O
\]
and 
\[
\bm{F}: U_1 \longrightarrow O_1 \qquad \text{and} \qquad \bm{F}: U_2 \longrightarrow O_1
\]
are diffeomorphisms.
Thus, $ \bm{F} $ is non-injective on an open set.

\underline{\it Step 3: Non-injectivity yields positive straightness defect.}
Reasoning by contradiction, assume that $S_t=0$. By \eqref{ns:eq:defect-as-chord}, there exists a Borel map
$\bm g: \mb R^3 \to \mb R^3$ such that
\[
\bm X=\bm g \lb \bm F(\bm X) \rb = \bm g \lb \sfH_t \lb \bm X \rb \rb, \qquad \text{almost surely}.
\]
Since $\bm X$ has a strictly positive density, $\bm x=\bm g(\bm F(\bm x))$ for Lebesgue-almost every $\bm x$ in each $U_i$, $i=1,2$.
The fact that $\bm F:U_1 \to O_1$ and $\bm F: U_2 \to O_1$ are both diffeomorphisms forces $\bm g(\bm z) \in U_1 \cap U_2$ for almost every $\bm z \in O_1$, which is a contradiction. Therefore, $S_t>0$. $\hfill$ $\square$

\noindent \textbf{Proof of Proposition~\ref{ns:prop:criterion}}
For $t\in(0,1)$,
\[
    \det D\sfH_t\lb\bm x_*\rb
    =\prod_{j=1}^3\lb1-t+t\lambda_j\rb.
\]
If $(1-\lambda_1)^{-1}<t<(1-\lambda_2)^{-1}$, the first factor is
negative and the other two are positive. Hence, the determinant is negative
throughout this nonempty open interval. Lemma~\ref{ns:lem:degree-defect}
therefore gives $S_t>0$ on that interval, so the coupling is
not straight. $\hfill$ $\square$

\section{Proof of Theorem~\ref{ns:thm:main}} \label{sec:nonstraight-proofs}

Throughout, $\Sp_{++}^k$ denotes the set of real symmetric positive-definite $k\times k$ matrices. Write points of $\bbR ^3$ as $( \bm{x} ,z)\in\bbR ^2\times\bbR $ and set
\begin{align} \label{ns:eq:zeta(r)}
\zeta(r)=r/\sqrt{1+r^2}.
\end{align}
Our construction has two ingredients. We first show that a sufficiently nice potential function with a desirable condition number exists, namely an ``admissible potential'' as in Definition \ref{ns:def:admissible} below. The second step is to show that starting with an admissible potential, one can get the desired potential function by taking its ``zero-temperature limit'', see Theorem \ref{ns:thm:zero-temperature-nonstraight} for the precise statement. Theorem~\ref{ns:thm:main} follows as a corollary of Theorems \ref{ns:thm:existence} and \ref{ns:thm:zero-temperature-nonstraight}.

\begin{defi}[Admissible potential]
\label{ns:def:admissible}
We say that a potential function $W\in C^\infty(\bbR ^3)$ is $(\kappa,\Lambda)$-\emph{admissible} if
\begin{enumerate}[label=\textnormal{(\Roman*)},ref=\Roman*,leftmargin=2.8em]
\item \label{cond:W-curvature}
$\kappa\bI _3\preceq\nabla^2W\preceq\Lambda\bI _3$ on $\bbR ^3$, for some
$0<\kappa\leq\Lambda<\infty$;
\item \label{cond:W-quadratic}
for some $\bQ \in\Sp_{++}^3$, it holds that
\[\sup_{\bm y}\left|W\lb\bm y\rb-\tfrac12\bm y^\top\bQ\bm y\right|<\infty.\]

\item \label{cond:W-axis}
$\nabla W( \bm{0} ,z)=( \bm{0} ,z)$ for all $z\in\bbR $;
\item \label{cond:W-negative-eigenvalues}
for $\bM_W$ the solution to the differential equation
\begin{equation}
\label{ns:eq:A4-ode}
    \bM _W'(r)=r\left[  \bG _W \lb \zeta(r) \rb+r^2 \cdot \bI _2\right]^{-1} \cdot \bM _W(r),
    \qquad \bM _W(0)=\bI _2,
\end{equation}
with $\bG_W$ defined by the identity
\begin{equation}
\label{ns:eq:axis-Hessian}
    \nabla^2W( \bm{0} ,z)=\begin{pmatrix}\bG _W(z)& \bm{0} \\ \bm{0} ^{\top}&1\end{pmatrix},
    \qquad z\in\bbR,
\end{equation}
it holds that the matrix $\bPi_W=\lim_{r\to\infty}\bM_W(r)/r$ is well-defined and has two distinct negative
eigenvalues.
\end{enumerate}
\end{defi}



The first condition \eqref{cond:W-curvature} requires $W$ to be smooth and strongly convex with parameters $\Lambda$ and $\kappa$, respectively. The second condition \eqref{cond:W-quadratic} requires that $W$ is within an $O(1)$ distance from a quadratic form, which will later be used to show that the endpoint map corresponding to $\bcN\lb\bm0,\bI_3\rb\otimes \exp \lb -\beta W(\bm x) \rb$ is proper, for all $\beta>0$. Conditions \eqref{cond:W-axis} and \eqref{cond:W-negative-eigenvalues} are needed to guarantee the existence of two distinct negative eigenvalues so that Proposition~\ref{ns:prop:criterion} is applicable.

\begin{thm} \label{ns:thm:existence}
    For any $S>2\pi$ and $\ve  \in (0,1/10)$, there exists a potential function $W \in C^\infty(\bbR ^3)$ that is $(\kappa, \Lambda)$-admissible in the sense of Definition \ref{ns:def:admissible} with 
    \begin{align} \label{ns:eq:condition-number-bound}
    e^{2S}< \frac{\Lambda}{\kappa} < (1+\ve ) e^{2S}. 
    \end{align}
\end{thm}
Given an admissible potential $W$, let $p_1^\beta\lb\bm x\rb\propto\exp\lb-\beta W\lb\bm x\rb\rb$. The following theorem transfers non-straightness from the zero-temperature limit to a finite value of $\beta$.
\begin{thm} \label{ns:thm:zero-temperature-nonstraight}
    Suppose $W$ is $(\kappa, \Lambda)$-admissible in the sense of Definition \ref{ns:def:admissible} and satisfies \eqref{ns:eq:condition-number-bound}. Then, for all sufficiently large $\beta>0$, $2$-RF yields a non-straight flow from the independent coupling $N \lb  \bm{0} , \bI _3 \rb \otimes p_1^\beta $, and the condition number of the potential $\beta W$ satisfies \eqref{ns:eq:condition-number-bound}.
\end{thm}

Theorems~\ref{ns:thm:existence} and \ref{ns:thm:zero-temperature-nonstraight} give Theorem~\ref{ns:thm:main} by taking $S$ sufficiently close to $2\pi$. It remains to prove Theorems \ref{ns:thm:existence} and \ref{ns:thm:zero-temperature-nonstraight}. Before doing so, we need some technical results.

\subsection{Technical preliminaries}

We need some technical results for the construction of admissible potential functions. The first result is the construction of an appropriate convex, smooth function that is needed for the gluing argument used later on.

\begin{lemma}
\label{ns:lem:crease}
For every $\alpha>0$ and $\beta\in(0,1/10)$, there exist $\delta_\pm>0$ and
$\chi\in C^\infty(\bbR )$ that depend only on $\alpha,\beta$ such that
\begin{enumerate}[label=\textnormal{(\roman*)},leftmargin=2.6em]
\item $\chi''\geq0$ on $\bbR $, so that $\chi$ is convex;
\item $\chi(s)=|s|$, for $s\notin(-\delta_-,\delta_+)$;
\item $|\chi''(s)|\cdot |\alpha-s|\leq\beta$ for every $s\in\bbR $,
\end{enumerate}
and with
\begin{equation}
\label{ns:eq:crease-inner}
    \frac{\alpha}{\beta\,e^{2/\beta}}\;\leq\;\alpha-\delta_+,
    \qquad 
      \delta_-\;\leq\;\frac{4\alpha}{\beta}\, .
\end{equation}
\end{lemma}

\noindent \textbf{Proof of Lemma \ref{ns:lem:crease}.}
Put $p(x)=\chi''(x)$. Suppose we have chosen $\delta_-$ and $\delta_+$, to satisfy conditions (i) and (ii), we only need
\begin{align} \label{ns:eq:p-constraint}
p \geq 0, \qquad \int_{\bbR } p(x) \, dx=2, \qquad \int_{\bbR } xp(x) \, dx=0, \qquad \supp \lb p \rb \subset \lb -\delta_{-}, \delta_{+} \rb.
\end{align}
Indeed, (i) follows from $p \geq 0$, and 
\[
\chi(s) =  -s + \int_{-\infty}^s (s-t) \cdot p(t) \, dt =|s|, \qquad \text{if $s \notin  \lb -\delta_{-}, \delta_{+} \rb$}.
\]
It suffices to pick $\delta_{\pm}$ and $p$ so that \eqref{ns:eq:crease-inner} and (iii) hold. Let $\eta \in C_c^\infty \lb (-1,1) \rb$ be an even, non-negative mollifier. For example, we can take 
\[
\eta(x)=c_\eta
\begin{cases}
\exp\left[-\dfrac{1}{1-4x^2}\right],& |x|<1/2,\\
0,& |x|\geq1/2,
\end{cases}
\qquad
c_\eta:=
\left(
\int_{-1/2}^{1/2}
\exp\left[-\frac{1}{1-4u^2}\right]\,du
\right)^{-1}.
\]
Pick $a,b,\gamma>0$ with $b<\alpha$, all to be chosen later, and choose $p$ of the form 
\[
p(s)=  q * \eta_\ve (s)
\]
where
\[
\eta_\ve (x):= \frac{1}{\ve } \cdot \eta \lb \frac{x}{\ve } \rb,
\qquad 
q(s):= \frac{\gamma}{\alpha-s}\cdot \mathbf 1 _{[-a,b]}(s).
\]
Now, observe that, since convolution preserves $L_1$ norm and also first moment, if the mollifier is even, we have
\[
\int_{\bbR } p(s) \, ds =  \int_{\bbR } q * \eta_\ve (x) \, dx = \int_{-a}^b  \frac{\gamma}{\alpha-s} \, ds = \gamma \cdot \log \lb \frac{\alpha+a}{\alpha-b} \rb,
\]
and
\[
\int_{\bbR } s\cdot p(s) \, ds = \int_{\bbR } s \cdot  q(s) \, ds
= \gamma \lb \alpha \cdot \log \frac{\alpha+a}{\alpha-b} - a -b   \rb
\]
Thus, to get the second and third constraints in \eqref{ns:eq:p-constraint}, we need  
\[
\gamma \cdot \log \lb \frac{\alpha+a}{\alpha-b} \rb=2, \qquad \alpha\cdot \log \frac{\alpha+a}{\alpha-b} - a -b=0.
\]
Plugging the first equation into the second one, we have $a+b=(2\alpha)/\gamma$. Solving these equations, we get
\begin{align} \label{ns:eq:a-b}
b=\alpha\cdot \left[ 1 - \frac{2}{\gamma(e^{2/\gamma}-1)} \right], \qquad a= \alpha\cdot \left[ \frac{2}{\gamma} + \frac{2}{\gamma \lb e^{2/\gamma}-1 \rb}-1 \right].
\end{align}
Note that $\supp\lb \eta_\ve  \rb \subset (-\ve ,\ve )$ and  $\supp \lb p \rb \subset [-a-\ve ,b+\ve ]$. It suffices to choose $a,b,\ve $ for given $\alpha,\beta$. The term $\ve $ is harmless since we can make it as small as possible once a and b are chosen. We will then set 
$\delta_-=a+\ve $ and $\delta_+=b+\ve $.

To get (iii), note that
\begin{align*}
    (\alpha-s)p(s) &= \gamma \int_{\bbR } \eta_\ve (u) \cdot \frac{\alpha-s}{\alpha-s+u} \cdot \mathbf 1 _{[-a,b]}(s-u) \, du \\
    &\leq \gamma \int_{\bbR } \eta_\ve (u) \cdot \lb 1 +  \frac{\ve }{\alpha-b} \rb \, du = \gamma \cdot \lb 1 +  \frac{\ve }{\alpha-b} \rb,
\end{align*}
where the last inequality follows from the fact that the function inside the integral takes non-zero value only when $|u| \leq \ve $ and $s-u \in [-a,b]$, which in turn yields
\[
\frac{\alpha-s}{\alpha-s+u}=1 - \frac{u}{\alpha-s+u} \leq 1 + \left| \frac{u}{\alpha-s+u} \right| \leq 1+ \frac{\ve }{\alpha-b}
\]
on the set $\la u: \eta_\ve (u) \cdot \frac{\alpha-s}{\alpha-s+u} \cdot \mathbf 1 _{[-a,b]}(s-u) \neq 0 \ra$.

Finally, with $a,b$ chosen as in \eqref{ns:eq:a-b}, for (i)--(iii) and \eqref{ns:eq:crease-inner} to hold, we need
\[
\gamma \cdot \lb 1 +  \frac{\ve }{\alpha-b} \rb \leq \beta, \qquad
\alpha-b-\ve  \geq \frac{\alpha}{\beta \cdot e^{2/\beta}}, \qquad 
a+\ve  \leq \frac{4\alpha}{\beta}.
\]
Such a choice of $\lb \gamma, \ve  \rb$ exists because we can first choose $\gamma$ to approach $\beta$ from the left and then take $\ve $ sufficiently small. The proof is completed. $\hfill$ $\square$

We will also need the following elementary ODE result. 

\begin{lemma} \label{ns:lem:ODE-sol-construction}
    For all $S>2\pi$, there exists a smooth function $\bm{H} : \lb 0, \infty \rb \to \Sp _{++}^2$, which can be constructed explicitly as in \eqref{ns:eq:H} below, such that the ODE
    \[
    \bm{M} '(r) = r \lb \bm{H} (r) + r^2 \bm{I} _2 \rb^{-1} \cdot \bm{M} (r), \qquad \bm{M} (0)= \bm{I} _2 
    \]
    has a unique solution $\bm{M}  \in C^1 \Big( [0,\infty),  \rm GL_2 \lb \mb R \rb \Big)$ 
    with 
    \[
    \bm{\Pi}  = \lim_{r \to \infty} \frac{\bm{M} (r)}{r} = -e^{-S/2} \cdot \exp \begin{pmatrix}
        S & \pi \\
        -\pi &0
    \end{pmatrix}.
    \]
    In particular, $\bm{\Pi} $ has two distinct negative eigenvalues: $- \exp \lb -\frac{\sqrt{S^2-4\pi^2}}{2} \rb$ and $- \exp \lb \frac{\sqrt{S^2-4\pi^2}}{2} \rb$.
\end{lemma}

We note that the assumption $S>2\pi$ is crucial here. For $0<S<2\pi$, the matrix $\bm{\Pi}$ has nonreal eigenvalues; at $S=2\pi$, it has the repeated negative eigenvalue $-1$. Neither case is useful for our construction of an admissible potential function. It is also worth noting that for the construction of an admissible potential function to work, the solution $\bm{M}$ of the ODE above has to be non-symmetric, at least for some $r>0$.\\

\noindent \textbf{Proof of Lemma \ref{ns:lem:ODE-sol-construction}.} Consider the functions over $[0,\infty)$
  \[
    b(r) := \frac{r}{r^2+e^{S}}, \qquad \delta(r):= \frac{r}{r^2+e^{-S}} - \frac{r}{r^2 +e^{S}} \qquad \text{and} \qquad  \vartheta(r) = \frac{\pi}{S} \int_0^r \delta(s) \, ds .
    \]
Then,
\[
 \vartheta(0) = 0, \qquad \vartheta(\infty)=\pi, \qquad \vartheta'(r) = \frac{\pi}{S} \cdot \delta(r)
\]
and
\[
\int_0^\infty \delta(x) \, dx= \lim_{r \to \infty} \int_0^r \delta(x) \, dx = \lim_{r \to \infty} \la S + \frac{1}{2}\log \lb \frac{r^2+e^{-S}}{r^2+e^S} \rb \ra = S.
\]
Let $ \bm{u} _\phi =\lb \cos \phi, \sin \phi \rb^\top$ and define the rank-one matrix 
\begin{align} \label{ns:eq:rank-one}
\bm{P} _\phi =  \bm{u} _\phi  \bm{u} _\phi^\top.
\end{align}
Set
\begin{align} \label{ns:eq:H}
\bm{H} (r) = e^{-S} \cdot \bm{P} _{\vartheta(r)} + e^S \cdot \lb \bm{I} _2 - \bm{P} _{\vartheta(r)} \rb.
\end{align}
It is easy to check that for all $r \geq 0$, $\operatorname{spec}  \bm{H} (r)  = \la e^{-S}, e^{S} \ra$, so that
\[
 e^{-S} \bm{I} _2 \preceq \bm{H} (r) \preceq e^S \bm{I} _2, \qquad \text{for all } r \geq 0.
\]
Thus, Lemma \ref{ns:lem:Pi-exists} yields the existence and uniqueness of a solution $\bm{M}  \in C^1 \Big( [0,\infty), \mbox{GL}_2(\mb R) \Big)$. Next, using the Sherman--Morrison formula
\[
\lb a \bm{I}  - b \,  \bm{u}   \bm{u} ^\top \rb^{-1} = \frac{1}{a} \cdot \bm{I}  + \frac{b}{a^2-ab} \cdot  \bm{u}   \bm{u} ^\top,
\]
we get
\begin{align*}
r\cdot \lb \bm{H} (r) + r^2\cdot \bm{I} _2 \rb^{-1} &= r \cdot \left[ \lb r^2 + e^S \rb\cdot \bm{I} _2 - \lb e^S -e^{-S} \rb \cdot \bm{P} _{\vartheta(r)}  \right]^{-1} \\
&= \frac{r}{r^2+e^{S}} \cdot \bm{I} _2 + \lb \frac{r}{r^2+e^{-S}} - \frac{r}{r^2 +e^{S}} \rb \cdot \bm{P} _{\vartheta(r)} \\
&= b(r) \cdot \bm{I} _2 + \delta(r) \cdot \bm{P} _{\vartheta(r)}. 
\end{align*}
Thus, the ODE for $\bm{M} $ becomes 
\[
\bm{M} '(r) = \left[ b(r) \cdot \bm{I} _2 + \delta(r) \cdot \bm{P} _{\vartheta(r)}  \right] \cdot \bm{M} (r), \qquad r \geq 0.
\]
Its solution is
\[
\bm{M} (r) = \sqrt{1+e^{-S} r^2} \cdot \bm{R} _{\vartheta(r)} \cdot \exp \left[ \vartheta(r) \lb \frac{S}{\pi} \bm{P} _0 - \bm{\Omega}  \rb \right], \qquad r \geq 0,
\]
where
\begin{align} \label{eq:Omega+R-theta}
\bm{\Omega} = \begin{pmatrix}
0 &-1 \\
1 & 0
\end{pmatrix} \qquad \text{and} \qquad
\bm{R} _\phi = \exp \lb \phi \bm{\Omega}  \rb=\begin{pmatrix}
    \cos \phi & -\sin \phi \\
    \sin \phi & \cos \phi
\end{pmatrix}.
\end{align}
Indeed, it is easy to check that $\bm{M} (0)=\bm{I} _2$ and,  setting $\bm{Q} := \frac{S}{\pi} \bm{P} _0 - \bm{\Omega} $,
\begin{align*}
\bm{M} '(r) &=  \frac{d}{dr} \left[ \sqrt{1+e^{-S}r^2} \, \right] \cdot \bm{R} _{\vartheta(r)} \cdot \exp \left[ \vartheta(r) \, \bm{Q}  \right]
+ \sqrt{1+e^{-S}r^2} \cdot  \frac{d}{dr} \left[ \bm{R} _{\vartheta(r)} \right] \cdot \exp \left[ \vartheta(r) \, \bm{Q}  \right] \\
&+ \sqrt{1+e^{-S}r^2} \cdot  \bm{R} _{\vartheta(r)} \cdot \frac{d}{dr} \la  \exp \left[ \vartheta(r) \, \bm{Q}  \right] \ra\\
&= b(r) \cdot \bm{M} (r) + \sqrt{1+e^{-S}r^2} \cdot \vartheta'(r) \, \bm{\Omega}   \, \bm{R} _{\vartheta(r)} \cdot \exp \left[ \vartheta(r) \, \bm{Q}  \right] \\
&+ \sqrt{1+e^{-S}r^2} \cdot  \bm{R} _{\vartheta(r)} \cdot \vartheta'(r) \bm{Q}  \cdot \exp \left[ \vartheta(r) \, \bm{Q}  \right] \\
&= b(r) \cdot \bm{M} (r) + \sqrt{1+e^{-S}r^2} \cdot \vartheta'(r) \, \bm{R} _{\vartheta(r)} \lb \bm{\Omega}  + \bm{Q}  \rb
\cdot \exp \left[ \vartheta(r) \, \bm{Q}  \right] \\
&= \left[ b(r) \cdot \bm{I} _2 + \delta(r) \cdot \bm{P} _{\vartheta(r)}  \right] \cdot \bm{M} (r),
\end{align*}
where the last equality follows from the fact that $\bm{R} _\phi$ commutes with $\bm{\Omega} $ and from $\bm{P} _{\vartheta(r)} =  \bm{R} _{\vartheta(r)} \bm{P} _0 \bm{R} _{-\vartheta(r)}$.
The proof is completed by noting that
\[
\lim_{r \to \infty} \frac{\bm{M} (r)}{r} = e^{-S/2} \cdot \bm{R} _{\vartheta(\infty)} \cdot \exp \left[ \vartheta(\infty) \cdot \lb \frac{S}{\pi} \bm{P} _0 - \bm{\Omega}   \rb \right] =   -e^{-S/2} \cdot \exp \begin{pmatrix}
        S & \pi \\
        -\pi &0
    \end{pmatrix},
\]
because $\vartheta(\infty)=\pi$ and $\bm{R} _\pi=-\bm{I} _2$. $\hfill$ $\square$\\

We will also need the following perturbation result for the endpoint map.  Let $\sfZ_1^{\beta,W}$ be the flow generated by the ODE 
\[
\frac{d}{dt} \sfZ ^{\beta,W}_t\lb   \bm{x} \rb  =  \bm{v}_{t}^\beta\lb  \sfZ ^{\beta,W}_t\lb   \bm{x} \rb  \rb, \qquad \sfZ ^{\beta,W}_0\lb   \bm{x} \rb  =  \bm{x}   
\]
where
\[
 \bm{v}_{t}^\beta\lb  \bm{x} \rb := \bbE _{\lb  \bm{X} , \widetilde{ \bm{Y} } \rb \sim N \lb  \bm{0} , \bI _3 \rb \otimes p_1^\beta } \left[ \widetilde{ \bm{Y} } -  \bm{X}   \Big|  t \widetilde{ \bm{Y} }+\lb 1-t\rb  \bm{X}  =  \bm{x}   \right].
\]
The next lemma asserts that bounded perturbation of the potential function results in bounded perturbation of the endpoint map.

\begin{lemma} \label{ns:lem:affine-tail}
Let $W$ be a function on $\mathbb{R}^3$ of the form
$$W\lb  \bm{y} \rb =\tfrac12  \bm{y} ^{\top}\bQ  \bm{y} +U\lb  \bm{y} \rb, \qquad \bm y \in \mathbb{R}^3, $$ 
with 
\[
\norm{U}_\infty<\infty, \qquad \bQ \in\Sp_{++}^3 \qquad \text{and} \qquad q_-=\lambda_{\min}\lb \bQ \rb.
\]
Assume also that $W$ satisfies \textnormal{\eqref{cond:W-curvature}} and \textnormal{\eqref{cond:W-quadratic}} in Definition~\ref{ns:def:admissible}. Then, for each fixed $\beta\geq1$,
\begin{equation}
\label{ns:eq:T-asymptotic}
    \sfZ_1^{\beta,W} \lb  \bm{x} \rb =\lb \beta\bQ \rb ^{-1/2} \bm{x} +\bm{R}_\beta \lb  \bm{x} \rb 
    \qquad \text{and} \qquad
    \sup_{ \bm{x} \in\bbR ^3}\norm{\bm{R}_\beta \lb  \bm{x} \rb }
      <\infty.
\end{equation}
\end{lemma}

\noindent \textbf{Proof of Lemma \ref{ns:lem:affine-tail}.} The idea of the proof is to compare the end-point maps obtained for the Gaussian distribution $N\lb\bm0,(\beta\bQ)^{-1}\rb$ and the distribution $p_1^\beta\lb\bm y\rb\propto\exp\lb-\beta W\lb\bm y\rb\rb$.

When $\bm{X} , \widetilde{ \bm{Y} }  \sim N \lb  \bm{0} , \bI _3 \rb \otimes p_1^\beta$ and $t < 1$, the posterior distribution of $\widetilde{ \bm Y}$ given $t \widetilde{ \bm Y} +\lb 1-t\rb  \bm{X} = \bm x$ takes the form
\begin{align*}
\pi_{t, \bm{x} } \lb d  \bm{y}  \rb &\propto \exp \la -\beta W\lb  \bm{y} \rb  - \frac{\norm{\bm{x}  - t \bm{y}}^2}{2\lb 1-t\rb ^2}  \ra \, d \bm{y} \\
&\propto e^{ - \beta U\lb  \bm{y} \rb  } \cdot \exp \la  - \frac{\beta }{2} \cdot  \bm{y} ^\top \bQ   \bm{y}  - \frac{\norm{\bm{x}  - t \bm{y}}^2}{2\lb 1-t\rb ^2} \ra \, d \bm{y}.
\end{align*}
Using this posterior, we can conveniently express $\bm{v}_{t}^{\beta}$ as
\[
\bm{v}_{t}^{\beta}\lb   \bm{x}  \rb = \frac{1}{1-t} \bbE _{ \bm{Y}  \sim \pi_{t, \bm{x} }} [ \bm{Y} ] - \frac{\bm x}{1-t}, \qquad 0 \leq t < 1.
\]

Next, let $ \bm{Y}  \sim N \lb  \bm{0} , \lb \beta \bQ  \rb^{-1}  \rb$ and let $\sfZ^{G} _t \in C^1 \lb [0,1], \bbR ^3 \rb$ be the flow generated by the ODE
\[
\frac{d}{dt} \sfZ^{G} _t\lb   \bm{x} \rb  =  \bm{v}_{t}^{G}\lb  \sfZ^{G} _t\lb   \bm{x} \rb  \rb, \qquad \sfZ^{G} _0\lb   \bm{x} \rb  =  \bm{x},
\]
where 
\begin{align*}
 \bm{v}_{t}^{G}\lb   \bm{x}  \rb:&= \bbE _{\lb  \bm{X} ,  \bm{Y}  \rb \sim N \lb  \bm{0} , \bI _3 \rb \otimes  N \lb  \bm{0} , \lb \beta \bQ  \rb^{-1}  \rb} \left[  \bm{Y}  -  \bm{X}   \,\Big| \,  t \bm{Y} +\lb 1-t\rb  \bm{X}  =  \bm{x}   \right]  \\
 &=  \bbE \left[ \bm Y \, \Big| \,  t \bm{Y} +\lb 1-t\rb  \bm{X}  =  \bm{x}  \right] 
 -  \bbE \left[ \bm X \, \Big| \,  t \bm{Y} +\lb 1-t\rb  \bm{X}  =  \bm{x}  \right] \\
 &=  \frac{t}{\beta} \bm Q^{-1} \cdot \left[ (1-t)^2 \bm I_3 + \frac{t^2}{\beta} \bm Q^{-1} \right]^{-1} \cdot \bm x - (1-t)\left[ (1-t)^2 \bm I_3 + \frac{t^2}{\beta} \bm Q^{-1} \right]^{-1} \cdot \bm x   \\
&= \left[ \frac t \beta \cdot  \bQ ^{-1} - \lb 1-t\rb  \bI _3 \right] \cdot \left[ \lb 1-t\rb^2 \bI _3 + \frac{t^2}{\beta} \bQ ^{-1} \right]^{-1}  \bm{x} . 
\end{align*}
 Above, the third line follows from the Gaussian conditional mean formula
\begin{align*}
\bbE \left[ \bm Y \, \Big| \,  t \bm{Y} +\lb 1-t\rb  \bm{X}  =  \bm{x}  \right] 
&= \mbox{Cov} \left[ \bm Y, t \bm{Y} +\lb 1-t\rb  \bm{X}  \right] \cdot \la \mbox{Cov} \left[ t \bm{Y} +\lb 1-t\rb  \bm{X} \right]  \ra^{-1} \cdot \bm x \\
&= \frac{t}{\beta} \bm Q^{-1} \cdot \left[ (1-t)^2 \bm I_3 + \frac{t^2}{\beta} \bm Q^{-1} \right]^{-1} \cdot \bm x.
\end{align*}
and
\begin{align*}
    \bbE \left[ \bm X \, \Big| \,  t \bm{Y} +\lb 1-t\rb  \bm{X}  =  \bm{x}  \right]
    &=  \mbox{Cov} \left[ \bm X, t \bm{Y} +\lb 1-t\rb  \bm{X}  \right] \cdot \la \mbox{Cov} \left[ t \bm{Y} +\lb 1-t\rb  \bm{X} \right]  \ra^{-1} \cdot \bm x \\
    &= (1-t) \bm I_3 \cdot \left[ (1-t)^2 \bm I_3 + \frac{t^2}{\beta} \bm Q^{-1} \right]^{-1} \cdot \bm x \\
    &= (1-t) \cdot \left[ (1-t)^2 \bm I_3 + \frac{t^2}{\beta} \bm Q^{-1} \right]^{-1} \cdot \bm x.
\end{align*}

In particular, for some $L_{\beta, \bQ }<\infty$, we have 
\[
\sup_{t \in [0,1]}\norm{\bm{v}_{t}^{G}\lb   \bm{x}  \rb -  \bm{v}_{t}^{G}\lb   \bm{y}  \rb} \leq L_{\beta, \bQ } \cdot \norm{\bm{x}  -  \bm{y}}
\]
for all $ \bm{x} ,  \bm{y}  \in \bbR ^3$. 

For $\bm{X} ,  \bm{Y}  \sim N \lb  \bm{0} , \bI _3 \rb \otimes  N \lb  \bm{0} , \lb \beta \bQ  \rb^{-1}  \rb$ and $t < 1$, denote by $\widetilde{\pi}_{t, \bm{x} }$ the posterior distribution of $\bm Y$ given $t \bm{Y} +\lb 1-t\rb  \bm{X} = \bm x$, which is given by
\[
\widetilde{\pi}_{t, \bm{x} } \lb  \bm{y}  \rb \propto \exp \la -\frac{\beta }{2} \cdot  \bm{y} ^\top \bQ   \bm{y}  - \frac{\norm{\bm{x}  - t \bm{y}}^2}{2\lb 1-t\rb ^2}\ra
\, d \bm{y}  \sim N \lb  \bm{\mu} _{t, \bm{x} }, \bA _{t, \bm{x} }^{-1} \rb,
\]
where
\[
\bA _{t,  \bm{x} }:= \beta \bQ  + \frac{t^2}{\lb 1-t\rb ^2} \bI _3, \qquad  \bm{\mu} _{t, \bm{x} } = \lb \bA _{t, \bm{x} }  \rb^{-1} \cdot  \frac{t \bm{x} }{\lb 1-t\rb ^2},
\]
using the explicit formulae for the conditional distribution of a sub-vector of a Gaussian vector. Similarly, for all $0 \leq t < 1$,
\[
\bm{v}_{t}^{G}\lb   \bm{x}  \rb = \frac{1}{1-t} \bbE _{ \bm{Y}  \sim \widetilde{\pi}_{t, \bm{x} }} [ \bm{Y} ] - \frac{\bm x}{1-t}.
\]





From the form of the posterior distributions,  we obtain that their Radon-Nikodym derivative at any $\bm y \in \mathbb{R}^3$ is
\[
\frac{d \, \pi_{t, \bm{x} }}{d \, \widetilde{\pi}_{t, \bm{x} }} \lb  \bm{y}  \rb = \frac{e^{-\beta \, U\lb  \bm{y} \rb }}{\bbE _{\, \widetilde{\pi}_{t, \bm{x} }} \left[ 
e^{-\beta U \lb  \bm{Y}  \rb}
\right]}
\leq e^{2\beta \norm{U}_\infty}.
\]
Putting all the pieces together, for all $t<1$ and $ \bm{x}  \in \bbR ^3$, 
\begin{align*}
    \norm{\bm{v}_{t}^{\beta}\lb   \bm{x}  \rb -  \bm{v}_{t}^G\lb   \bm{x}  \rb} &= \frac{1}{1-t} \cdot \norm{\bbE _{ \bm{Y}  \sim \pi_{t,  \bm{x} }} [ \bm{Y} ] - \bbE _{ \bm{Y}  \sim \widetilde{\pi}_{t, \bm{x} }} [ \bm{Y} ]} \\
    &\leq \frac{e^{2\beta \, \norm{U}_\infty}+1}{1-t} \cdot   \bbE _{ \bm{Y}  \sim \widetilde{\pi}_{t, \bm{x} }} \Big[  \norm{\bm{Y}  - \bbE _{ \bm{Y}  \sim \widetilde{\pi}_{t, \bm{x} }} [ \bm{Y} ]} \Big]  \\
    &\leq \frac{2e^{2\beta \, \norm{U}_\infty}}{1-t} \cdot \sqrt{\operatorname{tr} \lb \bA _{t, \bm{x} }^{-1} \rb} \\
    &\leq \frac{2e^{2\beta \, \norm{U}_\infty}}{1-t} \cdot \frac{\sqrt{3}}{\sqrt{\beta q_- + t^2\lb 1-t\rb ^{-2}}} 
     \leq 6e^{2\beta \norm{U}_\infty} \cdot \sqrt{1+ \frac{1}{\beta q_-}}:=\psi\lb \beta,q_-\rb,
\end{align*}
where $q_-= \lambda_{\rm min}(\bm Q)$. The second inequality follows from  Jensen's inequality, while the last follows from the fact that
\[
\bm A_{t, \bm x} \succeq \left[ \beta q_{-} + \frac{t^2}{(1-t)^2} \right] \cdot \bm I_3
\]
or, equivalently, 
\[
\bm A_{t, \bm x}^{-1} \preceq \left[ \beta q_{-} + \frac{t^2}{(1-t)^2} \right]^{-1} \cdot \bm I_3,
\]
which implies
\[
\mbox{tr} \lb \bm A_{t,\bm x}^{-1} \rb \leq \frac{3}{\beta q_- + t^2\cdot (1-t)^{-2}}.
\]

To finish the proof, put $E_t:=  \norm{ \sfZ _t^{\beta,W}\lb  \bm{x} \rb  - \sfZ^{G} _t \lb  \bm{x}  \rb}  $ and write 
\begin{align*}
    E_t &=  \left\| \sfZ _t^{\beta,W}\lb  \bm{x} \rb - \bm x - \left[ \sfZ^{G} _t \lb  \bm{x}  \rb - \bm x \right]   \right\| \\
    &=  \left\|  \int_0^t \bm{v}_{s}^\beta\lb  \sfZ ^{\beta,W}_s \rb \,ds - \int_0^t \bm{v}_{s}^{G}\lb  \sfZ ^{G}_s \rb \, ds \right\|  \\
    &\leq \int_0^t \norm{ \bm{v}_{s}^\beta\lb  \sfZ ^{\beta,W}_s \rb -  \bm{v}_{s}^{G}\lb  \sfZ ^{\beta,W}_s \rb} \, ds
    + \int_0^t \norm{\bm{v}_{s}^{G}\lb  \sfZ ^{\beta,W}_s \rb -  \bm{v}_{s}^{G}\lb  \sfZ^{G} _s \rb} \, ds \\
    &\leq \psi\lb \beta,q_-\rb  + L_{\beta,\bQ } \cdot \int_0^t E_s \, ds 
\end{align*}
Gronwall's inequality then yields 
\[
E_t \leq \psi\lb \beta,q_-\rb  \cdot e^{L_{\beta,\bQ }}, \qquad \forall t \in [0,1).
\]
Letting $t\uparrow1$ finishes the proof.
  $\hfill$ $\square$

\subsection{Proof of Theorem \ref{ns:thm:existence}}
 Take $S>2\pi$ and $\ve  \in (0,1/10)$. We first use Lemma \ref{ns:lem:ODE-sol-construction} to construct the $2\times 2$ block $\bm{G} _W(z)$ in \eqref{ns:eq:axis-Hessian}. For simplicity, we will drop the subscript $W$ and simply write $\bm{G} (z)$. Define
 \[
 \Theta(s) = \frac{\pi}{S} \la S + \frac 12 \cdot \log \left[  \frac{s+e^{-S}(1-s)}{s+e^S(1-s)}  \right] \ra, \qquad  - e^{-S} < s < 1 + e^{-S}.
 \]
 Take a smooth, compactly supported function $\widetilde{\Theta}(s)$ that agrees with $\Theta(s)$ in a neighborhood of $[0,1]$ and define
 \begin{align} \label{ns:eq:G}
 \bm{G} (z):= e^{-S} \cdot \bm{P} _{\widetilde{\Theta} \lb z^2 \rb} + e^S \cdot \lb \bm{I} _2 - \bm{P} _{\widetilde{\Theta} \lb z^2 \rb} \rb
 \end{align}
 where $\bm{P} _\phi$ is the rank-one map defined in \eqref{ns:eq:rank-one}. In other words, with $\zeta(r)$ as in \eqref{ns:eq:zeta(r)}, our choice $\bm{G} (z)$ gives
 \begin{align} \label{ns:eq:G=H}
 \bm{G} (\zeta(r)) =  \bm{H} (r)
 \end{align}
where $\bm{H} (r)$ is the matrix-valued function constructed in Lemma \ref{ns:lem:ODE-sol-construction}, see \eqref{ns:eq:H}. Note that this construction gives $\operatorname{spec} \, \bm{G} (z) = \la e^{-S}, e^S \ra$ for all $z \in \bbR $.
 
 The reason we choose $\widetilde{\Theta}(s)$ instead of $\Theta(s)$ is to extend $\bm{G} (z)$ smoothly to all $z \in \bbR $. Note that the construction yields $\operatorname{spec} \, \bm{G} (z) = \la e^{-S}, e^S \ra$ for all $z \in \bbR $, and thus
 \begin{align} \label{ns:eq:spectrum-G}
 e^{-S} \cdot \bm{I} _2 \preceq \bm{G} (z) \preceq e^S \cdot \bm{I} _2, \qquad \forall z \in \bbR . 
 \end{align}
 For $ \bm{x}  \in \bbR ^2$ and $\alpha,\widetilde{\ve },\widetilde{\ve }_1>0$ to be chosen later, define 
\[
W_0( \bm{x} ,z):= \frac{z^2}{2} +  \frac{1}{2}\bm{x} ^\top \, \bm{G} (z) \,  \bm{x} , \qquad 
Q \lb  \bm{x} , z \rb:= \lb 1+\widetilde{\ve } \rb \cdot e^{S} \cdot \frac{\norm{\bm{x}}^2}{2} + \frac{z^2}{2} - \alpha,
\]
and 
\[
\Delta \lb  \bm{x} , z \rb:= W_0( \bm{x} ,z)- Q \lb  \bm{x} , z \rb= \alpha - \frac 12 \cdot  \bm{x} ^\top \, \left[ \lb 1+\widetilde{\ve } \rb \cdot e^{S} \cdot \bm{I} _2 - \bm{G} (z) \right] \,  \bm{x} .
\]
We will take $W$ to be of the form
\[
W = \frac{1}{2} \left[ W_0 + Q + \chi \lb \Delta \rb  \right],
\]
where $\chi $ is the function from Lemma \ref{ns:lem:crease}, applied with parameters $\alpha$ and $\beta=\widetilde{\ve }_1$.

Before going further, let us explain why $W$ is chosen to be of this form. Firstly, $W_0$ satisfies \eqref{cond:W-axis},
but is not necessarily convex globally because $\bm G(z)$ varies with $z$. Fortunately, it is strongly convex with condition number roughly equal to $e^{2S}$ in a small cylinder $\la \lb  \bm{x} , z \rb: \norm{\bm{x}} \leq \varsigma \ra$, for some small $\varsigma>0$, due to \eqref{ns:eq:spectrum-G} and the fact that
\begin{align} \label{ns:eq:Hessian-W0}
\nabla^2 W_0 \lb  \bm{x} , z \rb =  \begin{pmatrix}
    \bm{G} (z)   & \bm{G} '(z) \,  \bm{x}  \\
     \bm{x} ^\top \, \bm{G} '(z) & 1 +\frac 12 \cdot  \bm{x} ^\top \, \bm{G}'' (z) \,  \bm{x}  
\end{pmatrix}
= \begin{pmatrix}
    \bm{G} (z) &  \bm{0} _{2\times 1} \\
     \bm{0} _{1\times 2} & 1 
\end{pmatrix}
+ O \lb \norm{\bm{x}} + \norm{\bm{x}}^2 \rb.
\end{align}
Therefore, we will let $W$ equal $W_0$ only in the cylinder where it can be shown to be strongly convex, and let $W$ equal $Q$ outside a larger neighborhood. Generally speaking, $W$ glues $W_0$ and $Q$ together in such a way that \eqref{cond:W-curvature}--\eqref{cond:W-axis} are preserved and the condition number of $W_0$ is only slightly perturbed (these properties are justified below).   

We visualize this gluing construction of $W$ in Figure \ref{ns:fig:cylinder} below. With $\delta_{\pm} = \delta_{\pm} \lb \alpha, \widetilde{\ve }_1 \rb$ as in Lemma \ref{ns:lem:crease} with $\delta_+ < \alpha < \delta_-$, we have
\begin{align*}
    W= \begin{cases}
        W_0, \qquad &\text{if $\Delta \geq \delta_+$ },  \\
         Q, \qquad &\text{if $\Delta \leq -\delta_-$ }.
    \end{cases}
\end{align*}
The region $\la \lb  \bm{x} , z \rb: \Delta \geq \delta_+ \ra$ is the green  twisted elliptical tube depicted in Figure \ref{ns:fig:cylinder}. In the region $\la \lb  \bm{x} , z \rb: -\delta_- \leq \Delta(\lb  \bm{x} , z \rb) \leq \delta_+ \ra$, $W$ interpolates smoothly between $W_0$ and $Q$ and is depicted as the complement of the green region in the cylinder. 
Put
\[
 \bm{D} (z) := \lb 1+\widetilde{\ve } \rb \cdot e^{S} \cdot \bm{I} _2 - \bm{G} (z)
\]
and note that for all $z \in \bbR $,
\begin{align} \label{ns:eq:spectrum-bound}
 \widetilde{\ve }e^S \cdot \bm{I} _2  \preceq \bm{D} (z) \preceq \left[  \lb 1+\widetilde{\ve } \rb \cdot e^{S}- e^{-S} \right] \cdot \bm{I} _2
 .
\end{align}

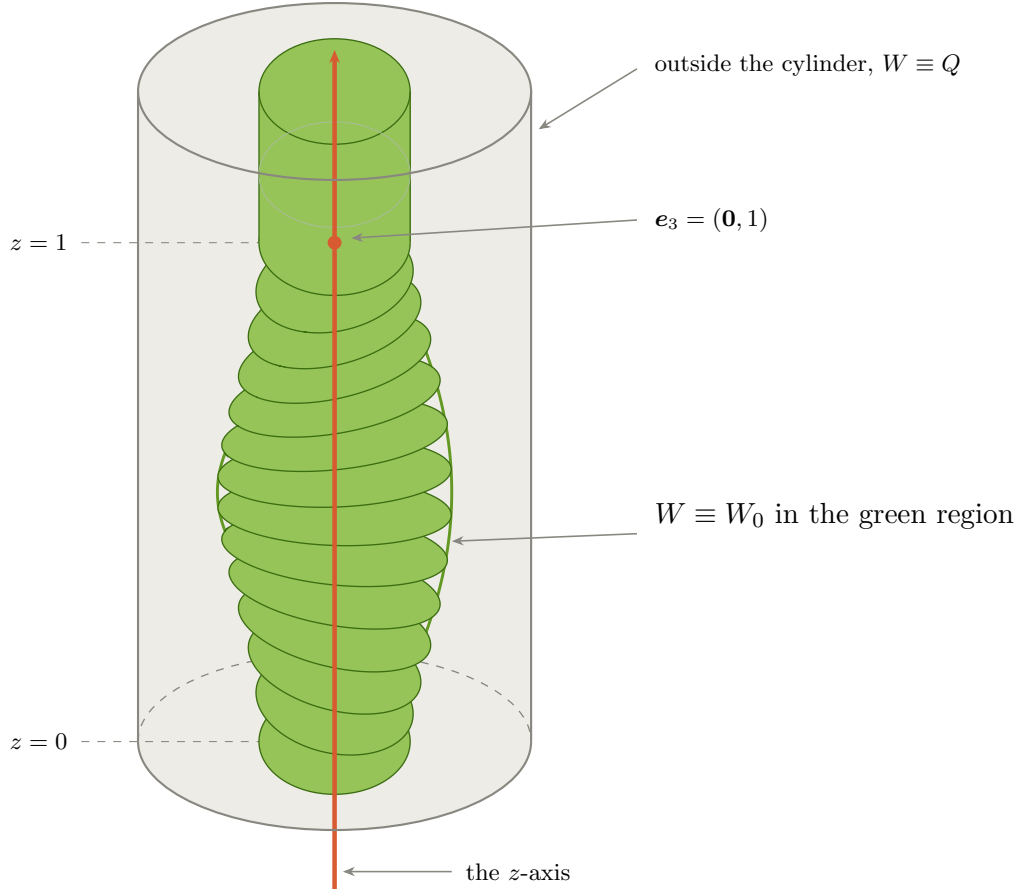
\begin{figure}[h]
\centering
\begin{tikzpicture}[x=1cm,y=1cm,>=Stealth]
 
\pgfmathsetmacro{\K}{0.45}      
\pgfmathsetmacro{\Hgt}{6.6}     
\pgfmathsetmacro{\Htop}{8.6}    
\pgfmathsetmacro{\Rad}{2.6}     
\pgfmathsetmacro{\Aa}{1.55}     
\pgfmathsetmacro{\Bb}{1.00}     
\pgfmathsetmacro{\Ns}{13}       
 
\tikzset{
  lead/.style={
    cylline,
    line width=.6pt,
    -{Stealth[length=4.5pt,width=3.5pt]}
  }
}
 
\fill[cylfill] (-\Rad,0) rectangle (\Rad,\Htop);
\fill[cylfill] (0,0)
  ellipse [x radius=\Rad, y radius=\K*\Rad];
\fill[cylfill] (0,\Htop)
  ellipse [x radius=\Rad, y radius=\K*\Rad];
 
\draw[cylline,dashed,line width=.4pt] (-3.35,0) -- (0,0);
\draw[cylline,dashed,line width=.4pt] (-3.35,\Hgt) -- (0,\Hgt);
\node[anchor=east,font=\footnotesize] at (-3.4,0) {$z=0$};
\node[anchor=east,font=\footnotesize] at (-3.4,\Hgt) {$z=1$};
 
\draw[cylline,dashed,line width=.5pt]
  (-\Rad,0)
  arc[
    start angle=180,
    end angle=0,
    x radius=\Rad,
    y radius=\K*\Rad
  ];
 
\foreach \s in {1,-1}{%
  \draw[bladeedge,line width=1.1pt]
    plot[variable=\f,domain=0:1,samples=80,smooth]
      ({\s*\Aa*cos(90+180*\f)},
       {\Hgt*\f+\K*\s*\Aa*sin(90+180*\f)});
}
 
\foreach \i in {0,...,13}{%
  \pgfmathsetmacro{\f}{\i/\Ns}
  \pgfmathsetmacro{\Zz}{\Hgt*\f}
  \pgfmathsetmacro{\Ph}{90+180*\f}
  \draw[bladeline,line width=.5pt,fill=bladefill]
    plot[variable=\t,domain=0:360,samples=61,smooth cycle]
      ({\Aa*cos(\t)*cos(\Ph)-\Bb*sin(\t)*sin(\Ph)},
       {\Zz+\K*(\Aa*cos(\t)*sin(\Ph)+\Bb*sin(\t)*cos(\Ph))});
}
 
\fill[bladefill] (-\Bb,\Hgt) rectangle (\Bb,\Htop);
\fill[bladefill] (0,\Htop)
  ellipse [x radius=\Bb, y radius=\K*\Aa];
\draw[bladeline,line width=.5pt]
  (-\Bb,\Hgt) -- (-\Bb,\Htop);
\draw[bladeline,line width=.5pt]
  (\Bb,\Hgt) -- (\Bb,\Htop);
\draw[bladeline,line width=.5pt]
  (0,\Htop) ellipse [x radius=\Bb, y radius=\K*\Aa];
\draw[bladeline!45,line width=.4pt]
  (0,\Hgt+0.9) ellipse [x radius=\Bb, y radius=\K*\Aa];
 
\draw[
  axiscol,
  line width=1.7pt,
  -{Stealth[length=5pt,width=4pt]}
]
  (0,-1.95) -- (0,\Htop+0.55);
\fill[axiscol] (0,\Hgt) circle (2.6pt);
 
\draw[cylline,line width=.8pt]
  (-\Rad,0) -- (-\Rad,\Htop);
\draw[cylline,line width=.8pt]
  (\Rad,0) -- (\Rad,\Htop);
\draw[cylline,line width=.8pt]
  (-\Rad,0)
  arc[
    start angle=180,
    end angle=360,
    x radius=\Rad,
    y radius=\K*\Rad
  ];
\draw[cylline,line width=.8pt]
  (0,\Htop) ellipse [x radius=\Rad, y radius=\K*\Rad];
 
\draw[lead] (4.0,8.9) -- (\Bb+1.7,8.1);
\node[anchor=west,align=left] at (4.1,8.95)
  {\footnotesize outside the cylinder, $W \equiv Q$
   };
 
\draw[lead] (4.0,6.9) -- (0.22,\Hgt+0.06);
\node[anchor=west,align=left] at (4.1,6.9)
  {\footnotesize $ \bm{e} _3=( \bm{0} ,1)$};
 
 
\draw[lead] (4.0,2.75) -- (1.55,2.65);
\node[anchor=west,align=left] at (4.1,2.75)
  {$W\equiv W_0$ in the green region \\[2pt]
   };
 
\draw[lead] (1.55,-1.72) -- (0.11,-1.72);
\node[anchor=west,font=\footnotesize] at (1.6,-1.72)
  {the $z$-axis};
 
 
\end{tikzpicture}
\caption{
Geometry of the gluing construction. 
}
\label{ns:fig:cylinder}
\end{figure}

To finish the proof, we check that $W$ satisfies \eqref{cond:W-curvature}--\eqref{cond:W-negative-eigenvalues} with a condition number satisfying \eqref{ns:eq:condition-number-bound}.

\underline{\it Checking \eqref{cond:W-negative-eigenvalues}.} $\bm{G} _W(z)$ in \eqref{ns:eq:axis-Hessian} has the form as in \eqref{ns:eq:G} by our construction, so \eqref{cond:W-negative-eigenvalues} follows from \eqref{ns:eq:G=H} and Lemma \ref{ns:lem:ODE-sol-construction}.

\underline{\it Checking \eqref{cond:W-axis}.} This is obvious because $W=W_0$ along a cylinder such that its interior contains the $z$-axis, which yields
\[
\nabla W( \bm{0} , z)= \nabla W_0( \bm{0} , z) = \lb  \bm{0} , z  \rb.
\]

\underline{\it Checking \eqref{cond:W-quadratic}.} Take $\bm{Q}  =  \operatorname{diag} \la \lb 1+\widetilde{\ve } \rb \cdot e^{S}; \lb 1+\widetilde{\ve } \rb \cdot e^{S}; 1  \ra$. Writing $ \bm{y}  = \lb  \bm{x} , z \rb \in \bbR ^3$, we have 
\[
Q \lb  \bm{x} , z \rb = \frac{1}{2}  \bm{y} ^\top \bm{Q}   \bm{y}  -\alpha
\]
for all $ \bm{y}  \in \bbR ^3$. We will verify that
\[
\sup_{ \bm{y}  \in \bbR ^3} \left| W( \bm{y} ) - \frac{1}{2}  \bm{y} ^\top \bm{Q}   \bm{y}   \right|<\infty.
\]
In the region $\la  \bm{y} : \Delta(\bm y) \leq -\delta_- \ra$, we have $ W( \bm{y} ) = Q( \bm{y} )$, so 
\[
\sup_{ \bm{y}  \in \la  \bm{y} : \Delta \leq -\delta_- \ra} \left| W( \bm{y} ) - \frac{1}{2}  \bm{y} ^\top \bm{Q}   \bm{y}   \right| = \alpha.
\]
In the region $\la  \bm{y} : \Delta(\bm y) \geq \delta_+  \ra$,  we have 
\[
 0 \leq  \bm{x} ^\top \, \bm{D} (z) \,  \bm{x}  \leq  2(\alpha-\delta_+)
\]
Thus
\[
\sup_{ \bm{y}  \in \la  \bm{y} : \Delta \geq \delta_+ \ra} \left| W( \bm{y} ) - \frac{1}{2}  \bm{y} ^\top \bm{Q}   \bm{y}   \right|
\leq \alpha+ \sup_{ \bm{y}  \in \la  \bm{y} : \Delta \geq \delta_+ \ra}  \bm{x} ^\top \, \bm{D} (z) \,  \bm{x}  
\leq \alpha + 2(\alpha-\delta_+). 
\]
Finally, in the region $\la  \bm{y} : -\delta_- \leq \Delta \leq \delta_+ \ra$, we have 
\[
 2(\alpha+\delta_-) \geq     \bm{x} ^\top \, \bm{D} (z) \,  \bm{x}  \geq 2 \lb \alpha - \delta_+ \rb.
\]
Consequently,
\begin{align*}
    \sup_{ \bm{y}  \in \la  \bm{y} : -\delta_- \leq \Delta \leq \delta_+ \ra} \left| W( \bm{y} ) - \frac{1}{2}  \bm{y} ^\top \bm{Q}   \bm{y}   \right|
    &=   \sup_{ \bm{y}  \in \la  \bm{y} : -\delta_- \leq \Delta \leq \delta_+ \ra} \left| \frac{1}{2} \lb W_0 + Q +\chi(\Delta) \rb - Q - \alpha  \right| \\
    &\leq \frac{1}{2} \cdot \sup_{ \bm{y}  \in \la  \bm{y} : -\delta_- \leq \Delta \leq \delta_+ \ra} \left| W_0( \bm{y} ) - Q( \bm{y} ) \right| +\alpha + \frac{1}{2} \sup_{-\delta_- \leq t \leq \delta_+} |\chi(t)| \\
    &\leq \frac{1}{2} \cdot \sup_{ \bm{y}  \in \la  \bm{y} : -\delta_- \leq \Delta \leq \delta_+ \ra}
    \left|  \bm{x} ^\top \, \bm{D} (z) \,  \bm{x}   \right| + \frac{3\alpha}{2} + \frac{1}{2} \sup_{-\delta_- \leq t \leq \delta_+} |\chi(t)| \\
    &\leq \alpha + \delta_- + \frac{3\alpha}{2} + \frac{1}{2} \sup_{-\delta_- \leq t \leq \delta_+} |\chi(t)|<\infty.
\end{align*}

\underline{\it Checking \eqref{cond:W-curvature}.}  
We first note that due to the construction of $\bm{G} $ and the fact that $\widetilde{\Theta}$ has compact support, we have
\[
\max \la \sup_{z \in \bbR } \norm{\bm{G} '(z)}_{\op}; \, \sup_{z \in \bbR } \norm{\bm{G} ''(z)}_{\op}   \ra =C_S<\infty,
\]
where $C_S$ depends only on $S$.

With $ \bm{y}  = \lb  \bm{x} ,z \rb \in \bbR ^3$, and omitting  for readability the argument $ \bm{y}$ in the functions $W_0$, $Q$ and $\Delta$,
\begin{align*}
    & \nabla^2 W \lb  \bm{y}  \rb \\
    =& \frac{1}{2} \cdot \nabla^2 \left[ W_0( \bm{y} ) + Q( \bm{y} ) + \chi \lb W_0(\bm y) - Q (\bm y)\rb \right] \\
    =& \frac 12 \cdot \nabla \left[ \nabla W_0 + \nabla Q + \chi' \lb W_0 - Q \rb \cdot \nabla \lb W_0 -Q \rb  \right] \\
    =& \frac 12 \cdot \left[ \nabla^2 W_0 + \nabla^2Q + \chi' \lb W_0 - Q \rb \cdot \nabla^2 \lb W_0 -Q \rb + 
    \chi'' \lb W_0-Q  \rb \cdot \nabla \lb W_0- Q \rb \cdot \nabla \lb W_0- Q \rb^\top   \right] \\
    =& \frac 12 \cdot \la \nabla^2 W_0 + \nabla^2Q  +\chi' \lb \Delta \rb \cdot \nabla^2 \lb \Delta  \rb 
    + \chi''(\Delta) \cdot \left[ \nabla  \Delta  \cdot \lb \nabla  \Delta \rb^\top \right] \ra \\
    =& \frac{1+\chi'\lb \Delta \rb}{2} \cdot \nabla^2 W_0 + \frac{1-\chi'\lb \Delta \rb}{2} \cdot \nabla^2 Q
    + \frac 12 \chi''(\Delta) \cdot \left[ \nabla  \Delta  \cdot \lb \nabla \Delta \rb^\top \right] \\
    =& \bm{L} _{ \bm{y} } + \bm{N} _{ \bm{y} },
\end{align*}
where
\[
\bm{L}_{ \bm{y} } := \frac{1+  \chi'\lb \Delta( \bm{y} ) \rb}{2} \cdot \nabla^2 W_0( \bm{y} ) + \frac{1-  \chi'\lb \Delta( \bm{y} ) \rb}{2} \cdot \nabla^2 Q( \bm{y} )
\]
and
\[
\bm{N} _{ \bm{y} } :=  \frac 12 \chi''(\Delta( \bm{y} )) \cdot \left[ \nabla \Delta( \bm{y} )  \cdot \lb\nabla  \Delta ( \bm{y} )\rb^\top \right].
\]
The term $\bm{L} _{ \bm{y} }$, the convex combination of two PSD matrices, is the dominant term in this sum since, due to our construction, $-1 \leq \chi'(s) \leq 1$ for all $s \in \bbR $. We will separately bound the spectrum of $\bm{L} _{ \bm{y} }$ and $\bm{N} _{ \bm{y} }$ in the following. In particular, we will show that 
\begin{align} \label{ns:eq:bound-L_y}
    \operatorname{spec} \, \bm{L} _{ \bm{y} } \subset \left[ e^{-S} - O_{\widetilde{\ve }, \widetilde{\ve }_1}(\sqrt{\alpha}), \lb 1+ \widetilde{\ve } \rb \cdot e^{S} + O_{\widetilde{\ve }, \widetilde{\ve }_1} \lb \sqrt{\alpha} \rb \right]
\end{align}
and 
\begin{align} \label{ns:eq:bound-N_y}
    \operatorname{spec} \, \bm{N} _{ \bm{y} } \subset \left[ 0, \lb 1 + \widetilde{\ve }\rb \, \widetilde{\ve }_1 \, e^S + O_{\widetilde{\ve }, \widetilde{\ve }_1}(\alpha) \right].
\end{align}
Given \eqref{ns:eq:bound-L_y} and \eqref{ns:eq:bound-N_y}, both \eqref{cond:W-curvature} and \eqref{ns:eq:condition-number-bound} follow because
\[
\frac{ \sup_{\bm y\in\bbR^3} \lambda_{\max}\lb\nabla^2W\lb\bm y\rb\rb}{ \inf_{\bm y\in\bbR^3}
\lambda_{\min}\lb\nabla^2W\lb\bm y\rb\rb}
\leq \frac{ \lb 1 + \widetilde{\ve }\rb \, \lb 1+\widetilde{\ve }_1 \rb \, e^S + O_{\widetilde{\ve }, \widetilde{\ve }_1}(\sqrt{\alpha})  }{e^{-S} - O_{\widetilde{\ve }, \widetilde{\ve }_1}(\sqrt{\alpha})}
\stackrel{\alpha \to 0}{\longrightarrow} \lb 1 + \widetilde{\ve }\rb \, \lb 1+\widetilde{\ve }_1 \rb \, e^{2S}.
\]
We then get \eqref{ns:eq:condition-number-bound} by picking $\widetilde{\ve }, \widetilde{\ve }_1$ such that 
\[
\lb 1 + \widetilde{\ve }\rb \, \lb 1+\widetilde{\ve }_1 \rb < 1+\ve .
\]
For the strict lower bound, choose $\kappa$ and $\Lambda$ as the global infimum and supremum of the Hessian eigenvalues, respectively. At a point on the $z$-axis the Hessian has the eigenvalue $e^{-S}$, whereas in the exterior region where $W=Q$ it has the eigenvalue $(1+\widetilde\ve)e^S$. Hence
\[
\frac{\Lambda}{\kappa}
\geq (1+\widetilde\ve)e^{2S}>e^{2S}.
\]
Together with the preceding upper bound, this proves \eqref{ns:eq:condition-number-bound}.
It remains to prove \eqref{ns:eq:bound-L_y} and \eqref{ns:eq:bound-N_y}. Let us start with \eqref{ns:eq:bound-L_y}. Note that on the set $\la  \bm{y} : \chi'(\Delta)\neq -1 \ra$, we have $\Delta>-\delta_-$, or equivalently
\[
 \bm{x} ^\top \, \bm{D} (z) \,  \bm{x}  \leq 2 \lb \alpha + \delta_- \rb.
\]
The above, \eqref{ns:eq:spectrum-bound} and \eqref{ns:eq:crease-inner} yield
\[
 \norm{\bm{x}}^2 \leq \frac{2 \lb \alpha + \delta_- \rb}{\widetilde{\ve }e^S} \leq \frac{2\lb \alpha + 4\alpha/\widetilde{\ve }_1  \rb}{\widetilde{\ve } e^S}
 \qquad \text{on $\la \Delta > -\delta_- \ra$}.
\]
Moreover, 
\[
\operatorname{spec} \, \bm{G} (z) \cup \operatorname{spec} \, \nabla^2 Q( \bm{y} ) = \la e^{-S}, 1, e^S, \lb 1+\widetilde{\ve }\rb\cdot e^S \ra
\]
Thus Lemma \ref{ns:lem:spectrum-bound}, the above display and \eqref{ns:eq:Hessian-W0} yield  \eqref{ns:eq:bound-L_y}.

We now prove \eqref{ns:eq:bound-N_y}. It is clear that $\bm{N} _{ \bm{y} } \succeq  \bm{0} $ and that $\bm{N} _{ \bm{y} } \neq  \bm{0} $ only when  $\Delta \in \lb -\delta_-, \delta_+ \rb$. Observe that
\[
\nabla \Delta \lb  \bm{x} , z \rb = \begin{pmatrix}
    -\bm{D} (z)  \bm{x}  \\
    \frac 12  \bm{x} ^\top \, \bm{G} '(z) \,  \bm{x} 
\end{pmatrix}
\]
Thus
\begin{align*}
\norm{\nabla \Delta( \bm{x} , z)}^2 = \norm{\bm{D} (z) \,  \bm{x}}^2 + \frac 14 \cdot \left[  \bm{x} ^\top \, \bm{G} '(z) \,  \bm{x}  \right]^2
&\leq  \bm{x} ^\top  \bm{D} (z)^2   \bm{x}  + O \lb \norm{\bm{x}}^4 \cdot \mathbf 1 _{\la \Delta \in \lb -\delta_-, \delta_+ \rb \ra} \rb \\
&\leq 2 \lb 1+ \widetilde{\ve } \rb \cdot e^S \cdot \lb \alpha - \Delta \rb + O_{\widetilde{\ve }, \widetilde{\ve }_1} \lb \alpha^2 \rb,
\end{align*}
 where the last inequality follows from $ \bm{x} ^\top \, \bm{D} (z) \,  \bm{x}  = 2 \lb \alpha-\Delta \rb$ and
 \[
 \bm{D} (z)^2 \preceq \lb 1+ \widetilde{\ve } \rb \cdot e^S \cdot \bm{D} (z),
 \]
which is due to \eqref{ns:eq:spectrum-bound}.

Consequently,
\begin{align*}
\frac 12 \cdot  \chi'' \lb \Delta \rb \cdot \norm{\nabla \Delta}^2
&\leq \chi'' \lb \Delta \rb \cdot \left[ \lb 1+ \widetilde{\ve } \rb \cdot e^S \cdot  \lb \alpha - \Delta \rb
+ O_{\widetilde{\ve }, \widetilde{\ve }_1} \lb \alpha^2 \rb \right] \\
&\leq \widetilde{\ve }_1 \cdot \lb 1+ \widetilde{\ve } \rb \cdot e^S + O_{\widetilde{\ve }, \widetilde{\ve }_1} \lb 
\frac{\alpha^2}{\alpha-\Delta}
\rb
\end{align*}
where the second line follows from property (iii) of $\chi$ in Lemma \ref{ns:lem:crease}. 

Furthermore, on  $\la \lb  \bm{x} ,z \rb: \Delta \in \lb -\delta_-, \delta_+ \rb \ra$, we get, by using \eqref{ns:eq:crease-inner},
\[
\frac{\alpha^2}{\alpha-\Delta} \leq \frac{\alpha^2}{\alpha-\delta_+} \leq \widetilde{\ve }_1 \, e^{2/\widetilde{\ve }_1} \cdot \alpha.
\]
Thus 
\[
\operatorname{spec} \, \bm{N} _{ \bm{y} } = \la 0; \frac 12 \cdot  \chi'' \lb \Delta \rb \cdot \norm{\nabla \Delta}^2 \ra
\subset \left[ 0, \lb 1 + \widetilde{\ve }\rb \, \widetilde{\ve }_1 \, e^S + O_{\widetilde{\ve }, \widetilde{\ve }_1}(\alpha) \right],
\]
which is \eqref{ns:eq:bound-N_y}. The proof is completed. $\hfill$ $\square$


\subsection{Proof of Theorem \ref{ns:thm:zero-temperature-nonstraight}}
   Let $ \bm{e} _3=\lb 0,0,1\rb ^{\top}\in\bbR ^3$ and put 
$$\bm{\mc{P}} _\beta\lb  \bm{x} \rb  :=  \sfZ_1^{\beta,W} \lb \sqrt\beta\, \bm{x} \rb .$$
where $ \sfZ_1^{\beta,W} $ is the endpoint of the flow obtained from $N \lb  \bm{0} , \bI _3 \rb \otimes p_1^\beta $. The reader is referred to Section \ref{sec:zero-temperature-ode} for some basic terminologies and properties of $\bm \sfZ_1^{\beta,W}$.

By Proposition~\ref{ns:prop:criterion} and Lemma \ref{ns:lem:affine-tail}, it suffices to show that for all $\beta>0$ sufficiently large, the Jacobian $D\sfZ_1^{\beta,W} \lb \sqrt\beta\, \bm{e} _3\rb $ has two distinct negative eigenvalues. By using the fact 
$$D \bm{\mc{P}} _\beta \lb  \bm{e} _3 \rb=\sqrt{\beta} \cdot D \sfZ_1^{\beta,W} \lb \sqrt\beta\, \bm{e} _3\rb  $$  
along with Proposition \ref{ns:prop:zero-temperature-endpoint}(iii), we only need to check that $D \bm{\mc{P}}  \lb  \bm{e} _3 \rb$ has two distinct negative eigenvalues. 

Indeed, given such a property on $D \bm{\mc{P}}  \lb  \bm{e} _3 \rb$, Proposition \ref{ns:prop:zero-temperature-endpoint}(iii) implies that $D \sfZ_1^{\beta,W} \lb \sqrt\beta\, \bm{e} _3\rb $ also has two distinct negative eigenvalues (with the remaining eigenvalue positive), for all $\beta$ sufficiently large. We then get non-straightness from Lemma  \ref{ns:lem:affine-tail}, \eqref{cond:W-quadratic} and Proposition~\ref{ns:prop:criterion}.

Define
\[
\bK \lb r, \bm{x} \rb := D_{ \bm{x} }  \bm{h}  \lb r,  \bm{x}  \rb \qquad \text{and} \qquad \bK \lb r\rb := D_{ \bm{x} }  \bm{h}  \lb r,  \bm{x}  \rb |_{ \bm{x}  =  \bm{e} _3}.
\]
Now, by Proposition \ref{ns:prop:zero-temperature-endpoint}(ii), 
\begin{align} \label{ns:DP-limit}
D \bm{\mc{P}}  \lb  \bm{e} _3 \rb =  \lim_{r \to \infty} \frac{\bK \lb r\rb }{r}.
\end{align}
Next, let us get an ODE for $\bK \lb r\rb $. Observe that
\begin{align*}
    \frac{d}{dr} \bK \lb r,  \bm{x} \rb  =  \frac{d}{dr} \left[ D_{ \bm{x} }  \bm{h}  \lb r,  \bm{x}  \rb  \right]
    &= D_{ \bm{x} } \left[ \frac{d}{dr}  \bm{h}  \lb r,  \bm{x}  \rb \right] \\
    &= D_{ \bm{x} } \left[ \bm{y} \lb r,  \bm{h}  \lb r,  \bm{x}  \rb \rb \right] \\
    &= D_{ \bm{h} } \bm{y} \lb r,  \bm{h}  \rb \cdot D_{ \bm{x} } \bm{h} \lb r,  \bm{x}  \rb \\
    &= r \cdot \left[ \nabla^2W\lb  \bm{y} \lb r, \bm{h} \rb \rb+r^2\bI _3 \right]^{-1} \cdot \bK  \lb r,  \bm{x}  \rb,
\end{align*}
where the last equality follows from \eqref{ns:eq:Dh-y}.

By setting $ \bm{x}  =  \bm{e} _3$, we deduce that
\[
\bK '\lb r\rb  =  r \cdot \left[ \nabla^2W\lb  \bm{y} \lb r, \bm{h} \rb \rb+r^2\bI _3 \right]^{-1} \cdot \bK \lb r\rb , 
\qquad
\bK \lb 0\rb  =  \bI _3.
\]
By \textnormal{\eqref{cond:W-axis}}, the trajectory $\lb  \bm{h} \lb r\rb ,  \bm{y} \lb r\rb   \rb$ starting at $ \bm{e} _3$ is precisely
$$
 \bm{h} \lb r, \bm{e} _3\rb 
=
\lb  \bm{0} ,\sqrt{1+r^2}\rb = \sqrt{1+r^2} \cdot  \bm{e} _3
\qquad \text{and} \qquad
 \bm{y} \lb r\rb 
=
\lb  \bm{0} ,\frac{r}{\sqrt{1+r^2}}\rb=\frac{r}{\sqrt{1+r^2} }\cdot  \bm{e} _3.
$$
Indeed, one can check from $\nabla W \lb  \bm{0} , z \rb =  \lb  \bm{0} , z  \rb$ that 
\[
\frac{d}{dr}  \bm{h}  \lb r,  \bm{e} _3 \rb =  \bm{y}  \lb r,  \bm{h}  \lb r,  \bm{e} _3 \rb \rb \qquad \text{and} \qquad  \nabla W \lb  \bm{y} \lb r\rb  \rb + r^2  \bm{y} \lb r\rb  - r \bm{h} \lb r\rb =\bm{0}.
\]
Together with \eqref{cond:W-negative-eigenvalues}, we find that
\[
\left[ \nabla^2W\lb  \bm{y} \lb r, \bm{h} \rb \rb+r^2\bI _3 \right]^{-1} =  \left[ \nabla^2W \lb \frac{r}{\sqrt{1+r^2}} \cdot  \bm{e} _3 \rb +r^2\bI _3 \right]^{-1} 
= \begin{pmatrix}
    \lb \bG_W \lb \zeta\lb r\rb  \rb + r^2 \bI _2 \rb^{-1}  &  \bm{0}  \\
     \bm{0}     &  \frac{1}{r^2+1}
\end{pmatrix}
\]
where $\zeta\lb r\rb =r/\sqrt{1+r^2}$.

Consequently, if we write $\bK \lb r\rb  =  \lb \bM _W\lb r\rb , k\lb r\rb   \rb$ with $\bM _W \in \bbR ^{2\times 2}$, then 
\[
\bM '_W\lb r\rb  = r \cdot \lb \bG_W \lb \zeta\lb r\rb  \rb + r^2 \bI _2 \rb^{-1} \cdot \bM _W\lb r\rb , \qquad \bM _W\lb 0\rb  =  \bI _2.
\]
and
\[
k'\lb r\rb  = \frac{r}{1+r^2} \cdot k\lb r\rb , \qquad k\lb 0\rb =1.
\]
It is easy to check that $k\lb r\rb  \equiv \sqrt{1+r^2}$, which yields the block form
\[
\bK \lb r\rb  =  \begin{pmatrix}
    \bM _W\lb r\rb   &  \bm{0}  \\
     \bm{0}    & \sqrt{1+r^2}
\end{pmatrix}.
 \]
By \eqref{cond:W-negative-eigenvalues}, $\bK \lb r\rb /r$ has two distinct negative eigenvalues for all $r$ large enough. This finishes the proof. $\hfill$ $\square$

\section{Proof of Theorem~\ref{ns:thm:gmm}}

\subsection{Heuristic construction}
\label{sec: ns gmm proof heuristic}

{
We explain the construction in dimension three with $\sigma=1$.
Choose an integer $N\geq2$, a horizontal radius $R>0$, and a
vertical scale $\eta>0$. The parameters $R,\eta$ will be used to control how close $D_{\rm max}^2$ is to $16\pi$. Arrange the means of the $2N$ components as
\[
\bm\mu_j^\pm
=
\bigl(\pm R\bm u_{\pi\theta_j^2},\,\eta\theta_j\bigr),
\qquad
\theta_j=\frac{2j-1}{2N},
\qquad
\bm u_\phi=
\begin{pmatrix}\cos\phi\\ \sin\phi\end{pmatrix},
\qquad 1 \leq j \leq N.
\]
The mixture is chosen to be 
\[
\sum_{j=1}^N  \left( p_j^{+} \cdot \bcN  \lb \bm \mu_j^+, \bm I_d \rb + p_j^- \cdot \bcN \lb \bm \mu_j^-, \bm I_d \rb \right),
\]
for some appropriately chosen weights $\la p_j^{\pm}; 1 \leq j \leq N \ra$. 

By pairwise symmetry of the means, any trajectory that begins on the $z$-axis stays on the $z$-axis at all times. Moreover, it can be checked that $4R^2 \leq D_{\rm max}^2 \leq 4R^2+\eta^2$, and we pick $R^2$ to be close to $4\pi$ so that the requirement that $D^2_{\rm max}$ be close to $16\pi$ holds. 
\begin{figure}[h]
    \centering
    \subfloat[Rotating antipodal component means.\label{fig:gmm-construction-means}]{%
    \begin{minipage}[t]{0.48\linewidth}
        \centering
        \includegraphics[width=\linewidth]{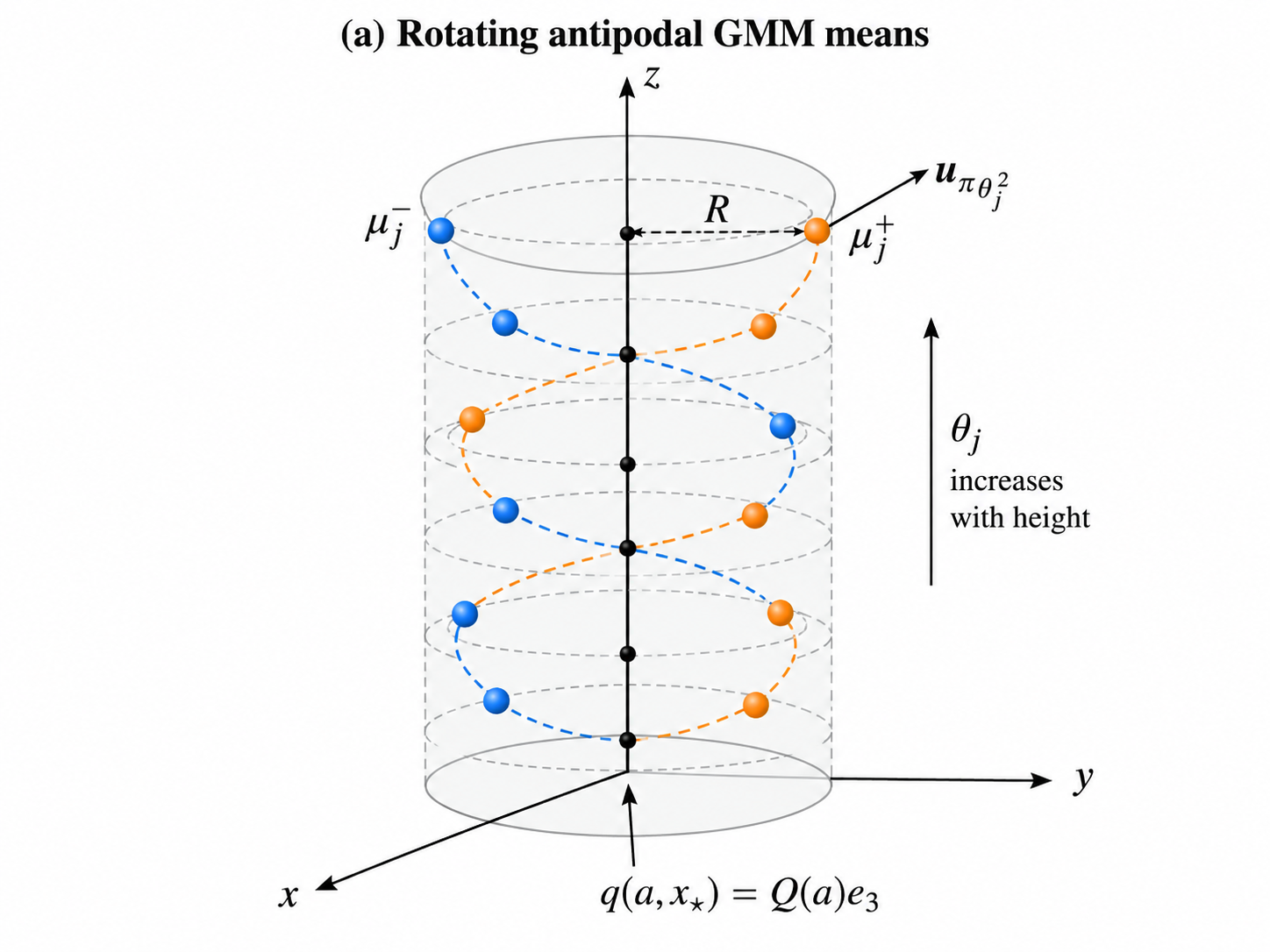}
    \end{minipage}}
    \hfill
    \subfloat[Rotating posterior covariance along
        the distinguished axial trajectory.\label{fig:gmm-rotating-covariance}]{%
    \begin{minipage}[t]{0.48\linewidth}
        \centering
        \includegraphics[width=\linewidth]{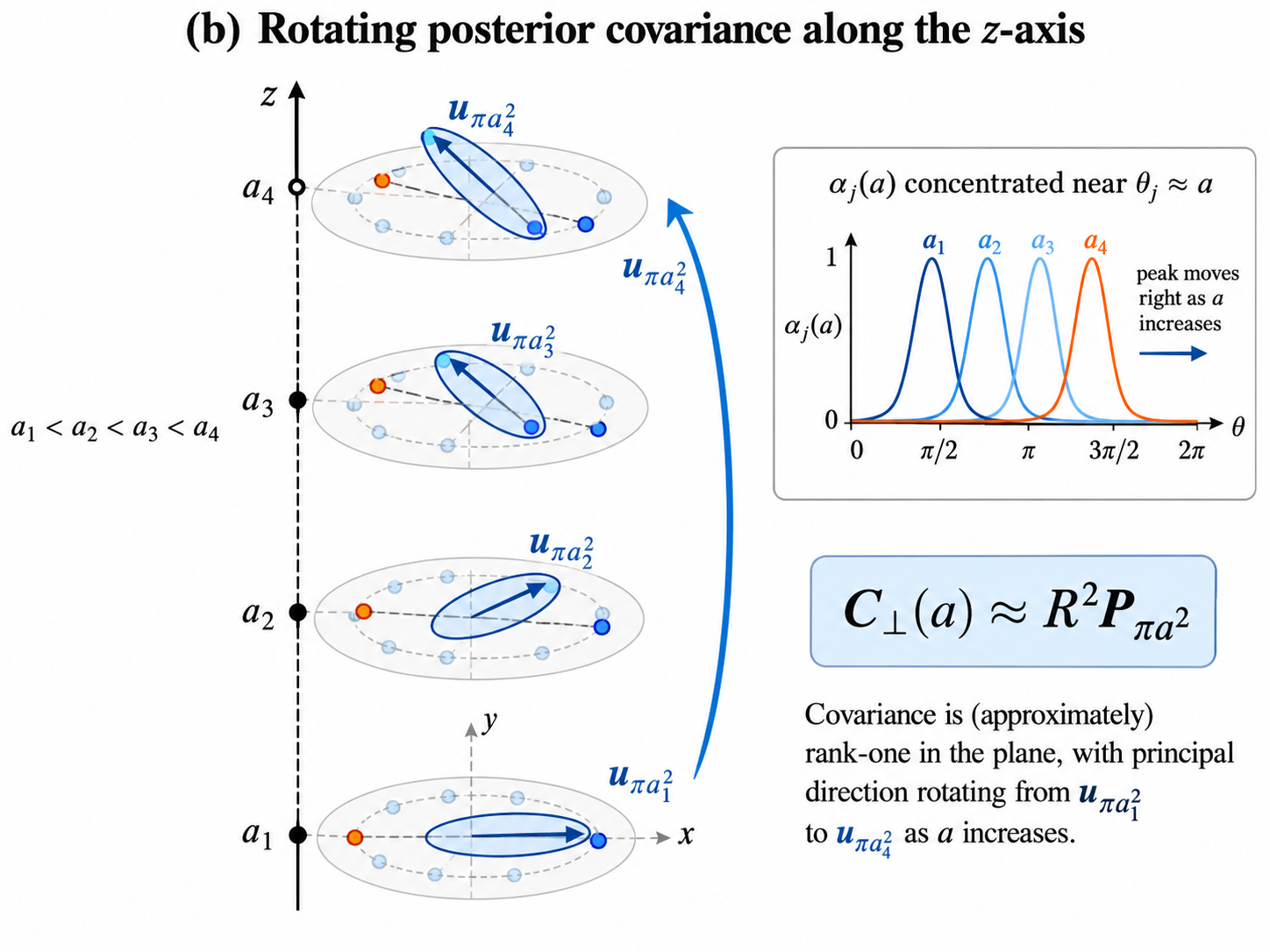}
    \end{minipage}}

    \caption{Geometric construction of the GMM.
    The mixture means are antipodal pairs with
    rotating horizontal directions. Along a chosen RF
    trajectory on the $z$-axis, the posterior progressively
    concentrates on successive pairs, inducing a rotating
    covariance in the transverse plane.}
    \label{fig:gmm-construction}
\end{figure}

\begin{figure}[h]
    \centering
    \includegraphics[width=0.5\linewidth]{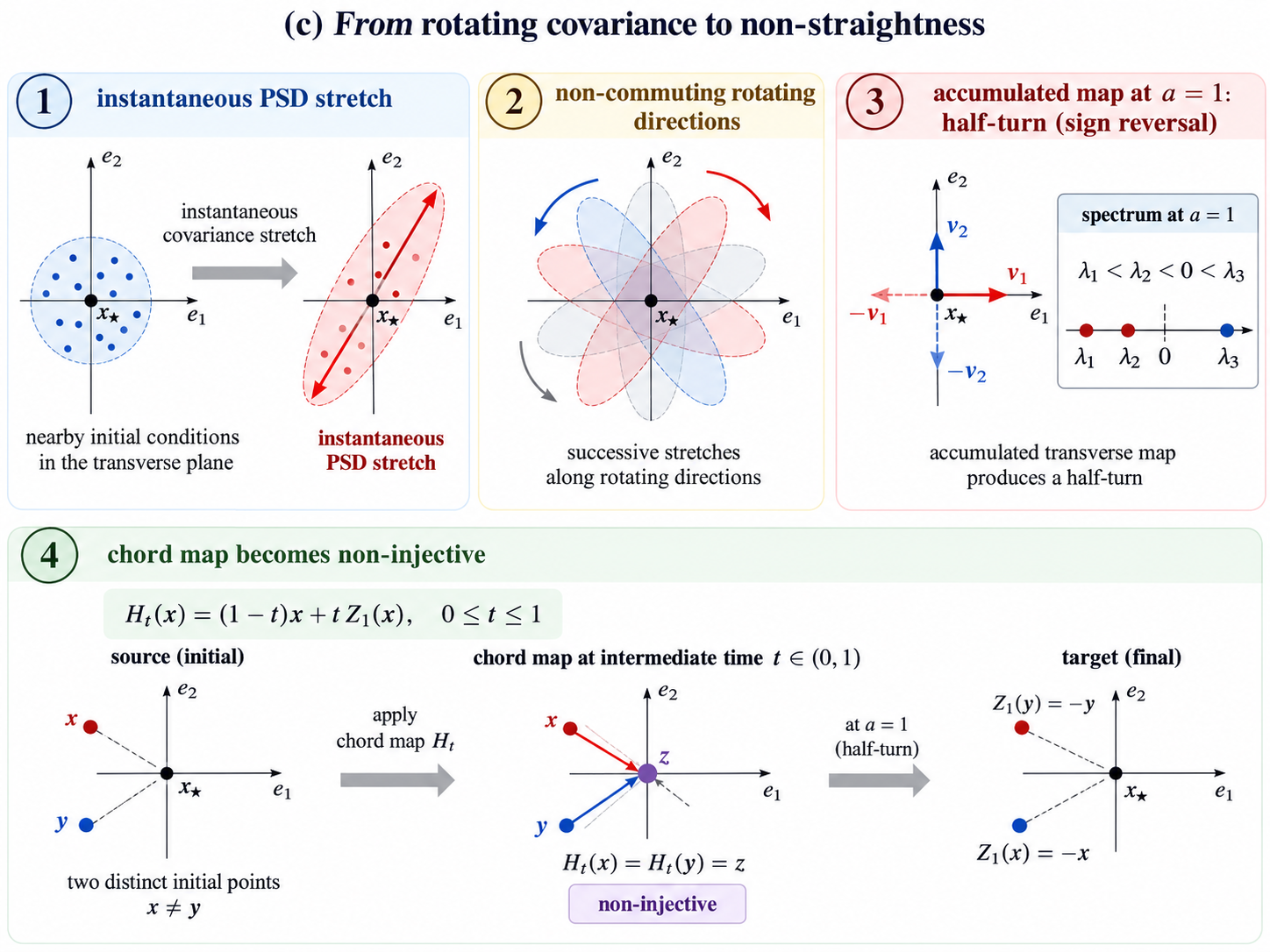}

    \caption{Mechanism of non-straightness.
    The rotating posterior covariance induces
    non-commuting positive semidefinite transverse
    stretches. Their accumulation produces an endpoint
    Jacobian with two distinct negative eigenvalues.
    By Proposition~\ref{ns:prop:criterion}, the resulting
    coupling is non-straight.}
    \label{fig:gmm-nonstraightness}
\end{figure}

Write $p_1$ for this mixture and let $\sfZ_t^{p_1}$ denote the
corresponding RF flow map. To get a cleaner form of the endpoint Jacobian, we introduce
the rescaled flow $\bm q$ satisfying
\[
\sfZ_t^{p_1}(\bm x)
=\rho_t\,\bm q(a(t),\bm x),
\]
where
\[
\rho_t:=\sqrt{(1-t)^2+t^2} \qquad \text{and} \qquad a(t):=\frac{t}{\rho_t}.
\]
As shown in \eqref{eq:ODE-K-gmm} below, the Jacobian of this rescaled flow
$\bm K(a,\bm x):=D_{\bm x}\bm q(a,\bm x)$ solves the matrix ODE
\[
\partial_a\bm K(a,\bm x)
=
a\,\bm C\bigl(a,\bm q(a,\bm x)\bigr) \cdot \bm K(a,\bm x),
\qquad
\bm K(0,\bm x)=\bI_3.
\]
Here,
\[
\bm C(a,\bm y)
:=
\Cov\bigl(\bm\mu\mid a\bm\mu+\bm G=\bm y\bigr)
\]
where $\bm\mu$ is the random component mean with
$\mb P(\bm\mu=\bm\mu_j^\pm)=p_j^\pm$, and
$\bm G\sim\mathcal N(\bm0,\bI_3)$ is independent of $\bm\mu$.
Since $\rho_1=a(1)=1$,  the endpoint Jacobians of the two flows coincide, i.e.
\[
D\sfZ_1^{p_1}(\bm x)=\bm K(1,\bm x).
\]
Using Proposition~\ref{ns:prop:criterion}, it is sufficient to choose the mixture weights $\la p_j^{\pm}; 1\leq j \leq N \ra$ so that the matrix
ODE of $\bm K$ approximates one whose solution has two distinct negative
eigenvalues at the endpoint $a=1$, for some initial location $\bm x_\star$. Towards that goal, we choose
\[
p_j^\pm
=
\frac{\exp(-B\theta_j^2)}
{2\sum_{\ell=1}^N\exp(-B\theta_\ell^2)},
\qquad
\bm x_\star=\frac{2B}{\eta}\bm e_3
\qquad \text{and} \qquad
\bm e_3=(0,0,1)^\top,
\]
for a sufficiently large $B>0$.

The evolution of the posterior covariance $\bm C(a, . )$ over the time variable $a>0$ can be described via the posterior distribution
\[
\alpha_j(a)
:=
\mb P\left(
\bm\mu\in\{\bm\mu_j^+,\bm\mu_j^-\}
\,\middle|\,
a\bm\mu+\bm G=\bm q(a,\bm x_\star)
\right),
\qquad \sum_{j=1}^N \alpha_j(a)=1. 
\]
More precisely, due to the fact that the trajectory stays on the $z$-axis, we get
\[
\bm C\bigl(a,\bm q(a,\bm x_\star)\bigr)
=
\begin{pmatrix}
\bm C_\perp(a)&\bm0_{2\times1}\\
\bm0_{1\times2}&c_3(a)
\end{pmatrix},
\]
where
\[
\bm C_\perp(a)
=
R^2\sum_{j=1}^N\alpha_j(a)\bm P_{\pi\theta_j^2},
\qquad
\bm P_\phi:= \bm u_\phi\bm u_\phi^\top
=
\begin{pmatrix}
\cos^2\phi & \cos\phi\sin\phi\\
\cos\phi\sin\phi & \sin^2\phi
\end{pmatrix},
\]
and
\[
c_3(a)
=
\eta^2
\left[
\sum_{j=1}^N\alpha_j(a)\theta_j^2
-
\left(\sum_{j=1}^N\alpha_j(a)\theta_j\right)^2
\right].
\]
See  the two panels in Figure~\ref{fig:gmm-construction}. The choice of the initial mixture weights $\la p_j^{\pm}; 1\leq j \leq N \ra$ gives,
as $N,B\to\infty$,
\[
\bm C_\perp(a)\to R^2\bm P_{\pi a^2},
\qquad
c_3(a)\to 0,
\]
uniformly for $a\in[0,1]$. 

The limiting $2\times2$ upper left block Jacobian therefore solves
\[
\bm M'(a)
=
aR^2\bm P_{\pi a^2} \cdot \bm M(a),
\qquad
\bm M(0)=\bI_2.
\]
We solve this ODE explicitly in Lemma \ref{lem:ODE-gmm} and show that,
when $R^2>4\pi$, $\bm M(1)$ has real, distinct and negative eigenvalues 
\[
-\exp\left(
\frac{R^2}{4}
\pm\sqrt{\frac{R^4}{16}-\pi^2}
\right).
\]

Note that the negative eigenvalues happen because the projections in the family $\la \bm P_a; a \in \mb R \ra$ generally do not commute. By choosing $N,B$ large (but finite), we get (using standard perturbation theory for ODEs) that 
\[
D\sfZ_1^{p_1}(\bm x_\star)
=
\bm K(1,\bm x_\star)
\approx
\begin{pmatrix}
\bm M(1)&\bm0\\
\bm0&1
\end{pmatrix}.
\]
Here, both the mixture $p_1$ and the starting point $\bm x_\star$ depend on $N,B$.
The final step connects this endpoint reversal to the folding of straight endpoint chords. That is, for
\[
\sfH_t(\bm x)
=
(1-t)\bm x+t\sfZ_1^{p_1}(\bm x),
\]
we have
\[
\det D\sfH_t(\bm x_\star)
=
\prod_{i=1}^3(1-t+t\lambda_i)<0
\qquad\text{when}\qquad
\frac{1}{1-\lambda_1}<t<\frac{1}{1-\lambda_2}.
\]
Lemma~\ref{lem:properness} verifies that
$\sfZ_1^{p_1}(\bm x)-\bm x$ is bounded.
Under this global condition,
Proposition~\ref{ns:prop:criterion} shows that the negative
determinant forces overlapping inverse branches of $\sfH_t$
and hence $S_t>0$. This gives the non-straightness mechanism, as illustrated in Figure~\ref{fig:gmm-nonstraightness}.
}

\subsection{Detailed proof of Theorem~\ref{ns:thm:gmm}}
\label{sec: ns gmm proof}



Before proving Theorem~\ref{ns:thm:gmm}, we first construct an explicit ODE whose solution has a Jacobian with two distinct, negative eigenvalues. Recall $\bm P_\theta$, $\bm R_\theta$ and $\bm \Omega$ in \eqref{ns:eq:rank-one} and \eqref{eq:Omega+R-theta}, respectively. Also recall the relation 
\[
\bm P_\theta = \bm R_\theta \bm P_0 \bm R_\theta^\top.
\]

\begin{lemma} \label{lem:ODE-gmm}
    Suppose $R^2>4\pi$. The equation
    \[
    \bm M'(a) = aR^2 \bm P_{\pi a^2} \cdot  \bm M(a) , \qquad \bm M(0)=\bm I_2, \qquad 0 \leq a \leq 1,
    \]
    has a unique solution $\bm M \in C^1 \Big( [0,1], \rm GL_2 \lb \mb R \rb \Big)$ such that $\bm M(1)$ has two distinct negative eigenvalues.  
\end{lemma}

\noindent \textbf{Proof of Lemma \ref{lem:ODE-gmm}.} The existence, uniqueness and $C^1$ regularity follow easily from Theorem 3.9 in \cite{teschl2012ordinary} since the coefficients are smooth. Its solution can be computed  explicitly as
\[
\bm M (a):= \bm R_{\pi a^2} \cdot \exp \left[ \pi a^2 \lb \frac{R^2}{2\pi} \bm P_0 - \bm \Omega \rb \right]. 
\]
Indeed, this can be checked directly by differentiating the above expression. Using the representation $\bm R_\theta = \exp \lb \theta \bm  \Omega \rb$ (for all real $\theta$) and setting $\bm K:= \lb R^2/2\pi \rb \bm P_0 - \bm \Omega $, we can write
\begin{align*}
     \bm M' (a) &= \frac{d}{da} \left[  \exp \lb \pi a^2 \bm \Omega \rb \cdot \exp \lb \pi a^2 \cdot \bm K \rb  \right] \\
    &= 2\pi a \cdot \bm \Omega \cdot \bm M(a) + \exp \lb \pi a^2 \bm \Omega  \rb \cdot 2\pi a \cdot \bm K \cdot \exp \lb \pi a^2  \cdot \bm K \rb \\
    &= 2\pi a \cdot \left[ \bm \Omega +  \exp \lb \pi a^2 \bm \Omega  \rb \cdot \bm K \cdot  \exp \lb -\pi a^2 \bm \Omega  \rb \right] \cdot \bm M(a) \\
    &= 2\pi a \cdot \left[ \bm \Omega +  \bm R_{\pi a^2} \bm K \bm R_{\pi a^2}^\top \right] \cdot \bm M(a) \\
    &= 2\pi a \cdot \left[ \bm \Omega + \frac{R^2}{2\pi} \bm P_{\pi a^2} - \bm \Omega \right] \cdot \bm M(a) = aR^2 \bm P_{\pi a^2} \cdot  \bm M(a)
\end{align*}
where the second last equality follows from the fact that $\bm \Omega$ commutes with $\bm R_\theta$, for all $\theta  \in \mb R$.

At $a=1$, we have $\bm R_\pi=-\bm I_2$, so 
\[
\bm M(1) = -\exp \left[ \pi \begin{pmatrix}
    R^2/(2\pi) & 1 \\
    -1 & 0
\end{pmatrix}  \right].
\]
The eigenvalues of $\bm M(1)$ are 
\[
-\exp \left[ \frac{\pi}{2} \lb \frac{R^2}{2\pi} + \sqrt{\frac{R^4}{4\pi^2}-4} \rb \right], \qquad -\exp \left[ \frac{\pi}{2} \lb \frac{R^2}{2\pi} - \sqrt{\frac{R^4}{4\pi^2}-4} \rb \right].
\]
$\hfill$ $\square$

\noindent \textbf{Proof of Theorem~\ref{ns:thm:gmm}.}  Without loss of generality, we may assume that $\sigma=1$. Take $\ve>0$ and choose
\[
R^2 := 4 \pi + \frac{\ve}{8} \qquad \text{and} \qquad \eta := \frac{\sqrt{\ve}}{2}.
\]
Recall $\bm u_\theta$ from \eqref{ns:eq:rank-one}. Let $N \in \mb N$ and $B >0$ be two parameters sufficiently large to be chosen later and define
\begin{align} \label{eq:theta-mu-p}
 \theta_j:= \frac{2j-1}{2N};  \qquad   \bm \mu^{\pm}_j:= \begin{pmatrix}
     \pm R \bm u_{\pi \theta_j^2} \\  
     \eta \theta_j 
 \end{pmatrix}
 \in \mb R^3, 
 \qquad 
 p_j^{\pm}:= \frac{ \exp \lb - B \theta_j^2 \rb}{2\sum_{k=1}^N  \exp \lb - B \theta_k^2 \rb}.
\end{align}
The desired target distribution is 
\[
p_1:= \sum_{j=1}^N \left[ p_j^{+} \cdot \bcN \lb \bm \mu_j^+, \bm I_3 \rb + p_j^- \cdot \bcN \lb \bm \mu_j^-, \bm I_3 \rb \right]
\]
Roughly speaking, $p_1$ is a mixture with $2N$ components, with equal weights within each antipodal pair. By Lemma \ref{lem:maximum-interpoint} and the definitions of $R,\eta$, we get 
\[
16\pi \leq D_{\rm max}^2 \leq 16 \pi + \ve. 
\]
Therefore, by using Proposition~\ref{ns:prop:criterion}, we only need to show that the endpoint map's Jacobian of $\sfZ_1^{p_1}$, where
\begin{align} \label{eq:sfZ_1}
\la \sfZ_t^{p_1} \ra_{0\leq t \leq 1}:= \Rectflow{\bcN \lb \bm 0, \bm I_3 \rb, p_1},
\end{align}
has exactly two distinct negative eigenvalues at some point. Note that the condition \eqref{eq:properness} is met with $\bm B = \bm I_3$ and by using Lemma \ref{lem:properness}.

We split the proof into a few steps.

\underline{\it Step 1: Rescale the endpoint map.} Make the change of variable 
\begin{align} \label{eq:rescaled-q}
\rho_t:= \sqrt{(1-t)^2+t^2}, \qquad a(t):= \frac{t}{\rho_t} \qquad \text{and} \qquad \sfZ_t^{p_1} \lb \bm x \rb:= \rho_t \cdot \bm q \lb a(t), \bm x \rb.
\end{align}
The last change of variable is valid because $a(t)$ is a $C^1$ homeomorphism on $[0,1]$. Roughly speaking, we will change the time scale from $t$ to $a$. The point of this change of variable is that $\bm q(a,\bm x)$ satisfies a convenient ODE that is easier to analyze.

Letting
\[
\bm K(a,\bm x):= D_{\bm x} \bm q \lb a, \bm x \rb,
\]
observe that
\[
D_{\bm x} \left[ \sfZ_t^{p_1} \lb \bm x \rb \right] = \rho_t \cdot \bm K \lb a(t), \bm x \rb.
\]
Hence, $D_{\bm x} \sfZ^{p_1}_1 \lb \bm x \rb =  \bm K(1, \bm x)$. 

Let us derive the ODE for $\bm K(., \bm x)$. A direct calculation shows that the velocity field in \eqref{eq: v_tweedie} with respect to  $\bcN \lb  \bm 0, \bm I_3  \rb \otimes p_1$ has the form
\[
\bm v^{p_1}_t(\bm z) = \frac{\rho'_t}{\rho_t} \cdot \bm z + \frac{1-t}{\rho_t^2 } \cdot \bm m \lb a(t), \frac{\bm z}{\rho_t} \rb ,
\]
where $\bm m$ is the posterior mean
\[
\bm m \lb a, \bm q \rb := \sum_{\ell=1}^N \left[ w_{\ell}^{+} \lb a, \bm q \rb \cdot \bm \mu_l^{+} + w_{\ell}^{-} \lb a, \bm q \rb \cdot \bm \mu_l^- \right],
\]
with
\[
w^{\pm}_{\ell} \lb a, \bm q \rb := \frac{p_{\ell}^{\pm} \cdot \exp \lb a\bm q^\top \bm \mu_\ell^{\pm} - \frac{a^2}{2} \cdot \| \bm \mu_\ell^{\pm} \|^2 \rb}
{\sum_{k=1}^N  \sum_{\ve \in \la +,- \ra} p_k^\ve \cdot \exp \lb a\bm q^\top \bm \mu_k^{\ve} - \frac{a^2}{2} \cdot \| \bm \mu_k^{\ve} \|^2 \rb }.
\]
Here the probability weights $p_j^{\pm}$'s are as in \eqref{eq:theta-mu-p}. One can think of $\bm m $ as the posterior mean since $\sum_{\ell=1}^N w_\ell^+(a,\bm q)+ w_\ell^-(a,\bm q)=1$ for all $(a,\bm q)$. 

From these representations, we get
\begin{align} \label{eq:ODE-q-gmm}
\partial_a \bm q(a,\bm x) = \bm m \lb a, \bm q(a,\bm x) \rb, \qquad \bm q(0,\bm x)= \bm x. 
\end{align}
Thus, taking the Jacobian $D_{\bm x}$ in the above yields 
\begin{align} \label{eq:ODE-K-gmm}
\partial_a K \lb a, \bm x \rb = a \cdot \bm C \lb a, \bm q(a,\bm x) \rb \cdot K(a,\bm x), \qquad \bm K(0,\bm x)= \bm I_3
\end{align}
where 
\[
\bm C \lb a, \bm q \rb= \sum_{\ell=1}^N \sum_{\ve \in \la +,- \ra}  w_\ell^{\ve}(a,\bm q) \cdot \left[ \bm \mu_\ell^\ve - \bm m(a,\bm q) \right] \cdot \left[ \bm \mu_\ell^\ve - \bm m(a,\bm q) \right]^\top. 
\]
To finish the proof, we will show that $\bm K(1,. )$ has two distinct negative eigenvalues at some point, for all sufficiently large $N,B$.

\underline{\it Step 2: Covariance concentration.}  Define 
\[
L:= \frac{2B}{\eta}, \qquad \bm x_\star = \lb 0, 0, L \rb=L \bm e_3. 
\]
Starting at $\bm x_\star$, the ODE in \eqref{eq:ODE-q-gmm} has an explicit solution given by 
\begin{align} \label{eq:q-solution}
    \bm q \lb a, \bm x_\star \rb = \lb 0,0, Q(a) \rb=Q(a)\cdot \bm e_3 
\end{align}
 where $Q$ solves the one-dimensional ODE
\begin{align} \label{eq:Q-ode-gmm}
Q'(a) = \frac{\eta}{2N} \cdot \sum_{j=1}^N  \left[ (2j-1) \cdot \alpha_j \lb a, Q(a) \rb  \right]
= \eta \cdot \sum_{j=1}^N \left[ \theta_j \cdot \alpha_j \lb a, Q(a) \rb \right],
\qquad Q(0)=L,
\end{align}
with
\[
\alpha_j(a,Q(a)):= 2 w_j^{+} \lb a, Q(a)\cdot \bm e_3 \rb = \frac{ \exp \lb -B\theta_j^2 + a\eta \cdot Q(a)\cdot \theta_j - \frac{a^2\eta^2}{2} \cdot \theta_j^2 \rb}{\sum_{k=1}^N \exp \lb -B\theta_k^2 + a\eta \cdot Q(a)\cdot\theta_k - \frac{a^2\eta^2}{2} \cdot \theta_k^2 \rb } 
\]
In fact, one can derive the ODE in \eqref{eq:Q-ode-gmm} and the solution in \eqref{eq:q-solution} by observing that
\begin{align*}
\partial_a \bm q(a,\bm x_\star)=\bm m \lb a, Q(a) \cdot \bm e_3 \rb &= \sum_{j=1}^N \frac{\alpha_j(a,Q(a))}{2} \cdot  \left[ \,
\begin{pmatrix}
    R \bm u_{\pi \theta_j^2 } \\
    \eta \theta_j
\end{pmatrix}
+  \begin{pmatrix}
    -R \bm u_{\pi \theta_j^2 } \\
    \eta \theta_j
\end{pmatrix} \,
\right] \\
&= \begin{pmatrix}
    0 \\ 0 \\ \eta \cdot \sum_{j=1}^N \alpha_j \lb a, Q(a) \rb \cdot \theta_j
\end{pmatrix}.
\end{align*}
Plug \eqref{eq:q-solution} into the formula of $\bm C$, we get
\begin{align*}
    \bm C \lb a, \bm q(a,\bm x_\star) \rb = \begin{pmatrix}
        R^2 \cdot \sum_{j=1}^N \alpha_j(a) \cdot \bm P_{\pi \theta_j^2} & \bm 0_{2 \times 1}  \\
        \bm 0_{1 \times 2}    & \eta^2 \cdot \mbox{Var}_{\bm \alpha(a)} \lb \la \theta_j \ra_{1\leq j \leq N} \rb
    \end{pmatrix}
\end{align*}
where we used the short-hand notation $\alpha_j(a):= \alpha_j(a,Q(a))$ and  $\bm \alpha(a):= \lb \alpha_1(a),\dots,\alpha_N(a)  \rb$.

Due to the form of $\bm C \lb a, \bm q(a,\bm x_\star) \rb$, the matrix-valued ODE in \eqref{eq:ODE-K-gmm} has solution of block type 
\[
\bm K \lb a, \bm x_* \rb =  \begin{pmatrix}
    \widetilde{\bm M}(a,\bm x_\star)  & \bm 0_{2\times 1} \\
    \bm 0_{1\times 2}   & k(a) 
\end{pmatrix}
\]
where $\widetilde{\bm M}(a,\bm x_\star)$ and $k(a)$ solve the ODEs
\[
\frac{d}{da} \widetilde{\bm M}(a,\bm x_\star) =  aR^2 \cdot \lb  \sum_{j=1}^N \alpha_j(a) \cdot \bm P_{\pi \theta_j^2} \rb \cdot \widetilde{\bm M}(a,\bm x_\star),
\qquad \widetilde{\bm M}(0,\bm x_\star) =  \bm I_2,
\]
and
\[
k'(a)= a\eta^2 \cdot \mbox{Var}_{\bm \alpha(a)} \lb \la \theta_j \ra_{1\leq j \leq N} \rb \cdot k(a), 
\qquad 
k(0)=1,
\]
respectively. 

With $\bm M \in C^1\lb [0,1], \rm GL_2(\mb R) \rb$ being the solution of the ODE in Lemma \ref{lem:ODE-gmm}, Lemma \ref{lem:perturbation-matrix-ode} (applied with $\bm C$ being the matrix in the ODE of $\widetilde{\bm M}$) yields
\begin{align} \label{eq:M-tilde-negative-eigen}
\left\| \widetilde{\bm M}(1,\bm x_\star) - \bm M \lb 1, \bm x_\star \rb \right\| &\leq \exp \lb R^2/2 \rb \cdot \la \exp \left[ \frac{R^2}{2} \cdot \sup_{a \in [0,1]} \left\| \sum_{j=1}^N \alpha_j(a) \cdot \bm P_{\pi \theta_j^2} - \bm P_{\pi a^2} \right\|   \right]-1 \ra \nonumber \\
&\leq \exp \lb 2\pi + \frac{\ve}{16} \rb \cdot \la \exp \left[ \lb 2\pi + \frac{\ve}{16}  \rb \cdot \lb \frac{\pi}{N}+ \frac{\pi\sqrt{2}}{\sqrt{B}} + \frac{\pi\ve}{4B} \rb \right]  -1 \ra,
\end{align}
where the second bound comes from Lemma \ref{lem:block-convergence-gmm} and that $\eta^2=\ve/4$. 

By choosing $N,B$ sufficiently large so that the last term in \eqref{eq:M-tilde-negative-eigen} is sufficiently small, we conclude that $\widetilde{\bm M}(1,\bm x_\star)$ has two distinct negative eigenvalues for such $(N,B)$. Keep in mind that $k(1)>0$ for finite $N,B$, so $D_{\bm x} \sfZ^{p_1}_1 \lb \bm x_\star \rb= \bm K \lb 1, \bm x_\star  \rb$ has two distinct negative eigenvalues and one positive eigenvalue.  The proof is completed. $\hfill$ $\square$

\section{Proofs of approximate straightness results}
\label{sec: proof of approx-straightness results}

\subsection{Proof of Proposition~\ref{prop: approximate straightness logconcave} }
\label{sec: proof of approximate straightness -- log-concave}

Let
\[
    \Gamma:=\frac14\log(L/m),
    \qquad
    \rho:=(mL)^{-1/4}
    \qquad \text{and} \qquad
    \eta:=\sqrt{\Big[ \Gamma^2-\pi^2 \Big]_+}.
\]
We first show that if non-injectivity happens for some $t \in [0,1]$, then 
\[
 \sqrt{\pi^2+
        \log^2\!\left(\frac{1-t}{t\rho}\right)} \leq \Gamma .
\]
Towards that end, suppose that $\sfH_t(\bm x)=\sfH_t(\bm y)$ for some distinct
$\bm x,\bm y$ and some $t\in(0,1)$. Then, as argued in \eqref{eq: H_t not 1-1},
\[
    \bm\delta_1(\bm x,\bm y)
    =-\frac{1-t}{t}\cdot \bm\delta_0(\bm x,\bm y).
\]
For any $s \in [0,1]$, write $\bm u_s:=\bm\delta_s(\bm x,\bm y)$ and
$\bm n_s:=\bm u_s/\norm{\bm u_s}_2$.
Then $\bm n_1=-\bm n_0$ and
\[
    \frac{\norm{\bm u_1}_2}{\norm{\bm u_0}_2}
    =\frac{1-t}{t}.
\] 
Define
\[
    b_s:=\frac{f_m(s)+f_L(s)}2,
    \qquad \text{and} \qquad 
    r_s:=\frac{f_m(s)-f_L(s)}2.
\]
where, for any $a > 0$, $f_a$ is the function 
\[
s \in[0,1] \mapsto  f_a(s):=\frac{s-a(1-s)}{s^2+a(1-s)^2}.
    \qquad
\]
Let $\bA_s$ be the matrix in \eqref{eq:A_t} and set  $a_s:=\bm n_s^\top\bA_s\bm n_s$.
Using the normalized-vector
calculation in the proof of
Theorem~\ref{theorem: straightness theorem}, we obtain
\begin{align*}
    a_s = \frac{\bm u_s^\top \, \bm A_s \, \bm u_s}{ \| \bm u_s \|_2^2} = \frac{1}{2\| \bm u_s \|_2^2} \cdot \frac{d}{ds} \lb \| \bm u_s \|_2^2 \rb = \frac{d}{ds} \lb. \log\norm{\bm u_s}_2 \rb
\end{align*}
where the second equality follows from the second equality in \eqref{eq:ode-norm-u}. The above display also yields
\[
\int_0^1 a_s \, ds =  \la \log\norm{\bm u_s}_2 \ra \Big|_{s=0}^{s=1} =  \log \lb \frac{\| \bm u_1\|_2}{\| \bm u_0 \|_2} \rb.
\]
By using the ODE \eqref{eq:n_t-ode} and the fact that $\bm n_s$ is orthogonal to $\dot{\bm n_s}$, we have 
\[
\| \lb \bm A_s - b_s \bm I_d \rb \, \bm n_s   \|_2^2 = \| \dot{\bm n}_s \|_2^2 + \lb a_s-b_s \rb^2 \cdot  \| \bm n_s \|_2^2 =  \| \dot{\bm n}_s \|_2^2 + \lb a_s-b_s \rb^2.
\]
Additionally, the bound $    f_L(s) \cdot \bI_d\preceq\bA_s\preceq f_m(s) \cdot \bI_d$ (due to Lemma \ref{lem:Jacobian-bound-logconcave}) yields 
\[
-r_s \cdot \bm I_d \preceq \bm A_s - b_s \cdot \bm I_d \preceq r_s \cdot \bm I_d.
\]
Consequently,
\begin{align} \label{eq:r_s-upper bound}
r_s^2 \geq \| \lb \bm A_s - b_s \bm I_d \rb \, \bm n_s   \|_2^2 \geq \| \dot{\bm n}_s \|_2^2 + \lb a_s-b_s \rb^2.
\end{align}
Moreover, since $\bm n_1=-\bm n_0$, the fact that the geodesic distance achieves the shortest length among all smooth curves connecting two points gives
\begin{align} \label{eq:geodesics}
\pi\le \int_0^1\norm{\dot{\bm n}_s}_2\,ds.
\end{align}
Putting the bounds together, we get from triangle inequality (for vector-valued integral) that
\begin{align*}
    \Gamma \stackrel{(i)}{=}  \int_0^1 r_s \, ds \stackrel{(ii)}{\geq} \int_0^1  \sqrt{\norm{\dot{\bm n}_s}_2^2+(a_s-b_s)^2}
    \,ds & \stackrel{(iii)}{\geq} 
    \sqrt{ \lb \int_0^1\norm{\dot{\bm n}_s}_2\,ds \rb^2 +  \left[  \int_0^1 (a_s - b_s)\, ds \right] ^2} \\
    & \stackrel{(iv)}{\geq} \sqrt{\pi^2 +  \left[  \int_0^1 (a_s - b_s)\, ds \right] ^2 } .
\end{align*}
Here $(i)$ follows from a simple calculation, $(ii)$ follows from \eqref{eq:r_s-upper bound}, $(iv)$ follows from \eqref{eq:geodesics} and $(iii)$ follows from Minkowski's integral inequality (Theorem 6.19 in \cite{folland1999real}): for the vector-valued function  $s \in [0,1] \mapsto \bm w_s:= \lb \| \dot{\bm n_s}\|_2, a_s-b_s \rb^\top$, Minkowski's integral inequality for the Euclidean norm (and integrating over the space $\{ 1,2\} \times [0,1]$ with respect to the product of the counting and Lebesgue measures) yields
\[
\sqrt{ \lb \int_0^1\norm{\dot{\bm n}_s}_2\,ds \rb^2 +  \left[  \int_0^1 (a_s - b_s)\, ds \right] ^2}
= \left\| \int_0^1 \bm w_s \, ds \right\| \leq \int_0^1 \| \bm w_s \|_2 \, ds
= \int_0^1  \sqrt{\norm{\dot{\bm n}_s}_2^2+(a_s-b_s)^2}
    \,ds.
\]

Let us evaluate the integral on the right-hand side of $(iv)$. Write
\[
\int_0^1 (a_s - b_s)\, ds= \int_0^1 a_s \, ds - \int_0^1 b_s \, ds = \log \lb \frac{\| \bm u_1\|_2}{\| \bm u_0 \|_2} \rb - \log \rho = \log \lb \frac{1-t}{t \rho} \rb.
\]
Therefore,
\[
\Gamma \geq  \sqrt{\pi^2+     \log^2\!\left(\frac{1-t}{t\rho}\right)}.
\]
Consequently, every non-injective time belongs to
\[
    I:=
    \left[
        \frac1{1+\rho e^\eta},
        \frac1{1+\rho e^{-\eta}}
    \right].
\]
Finally, put
\[
    \bm D:=\sfZ_1(\bm Z_0)-\bm Z_0,
    \qquad
    \bm H_t:=\sfH_t(\bm Z_0)
    \qquad \text{and} \qquad 
    S_t:=\bbE\norm{
        \bm D-\bbE[\bm D\mid\bm H_t]
    }_2^2.
\]
By Definition~\ref{def: approximate straightness},
$S(\bm{\mc Z})=\int_0^1S_t\,dt$.
Outside $I$, the continuous injective map $\sfH_t$ has a
measurable inverse on its image, so $S_t=0$.
For $t>0$, the identity $\bm D=(\bm H_t-\bm Z_0)/t$ gives
\[
    S_t
    =\frac1{t^2}\bbE\norm{
        \bm Z_0-\bbE[\bm Z_0\mid\bm H_t]
    }_2^2
    \le\frac d{t^2}.
\]
Therefore,
\[
S(\bm{\mc Z})
\leq d\int_I\frac{dt}{t^2}
= d\rho(e^\eta-e^{-\eta})
= 2d\rho\sinh\eta.
\]
which is the asserted bound. This also covers
$\Gamma\le\pi$, since then $I$ is a singleton. $\hfill$ $\square$

\subsection{Proof of Theorem~\ref{thm: trajectory controls coupling}}
\label{sec: proof of trajectory controls straightness}
To simplify notation, we suppress the superscript $K$
throughout the proof. Write
\[
\bm D=\bm Z_1-\bm Z_0,
\qquad
\bm H_t=(1-t)\bm Z_0+t\bm Z_1.
\]
By the optimality of the conditional expectation, we have
\[
S(\mc Z)
\le
\bbE\int_0^1
\norm{\bm D-\bm v_t(\bm H_t)}_2^2\,dt.
\]
Using Minkowski's inequality with respect to the product measure $\mathbb{P}_{\bm Z_0, \bm Z_1} \otimes {\rm Unif}[0,1]$  
 and the $\ell_K$-Lipschitz
continuity of $\bm v_t$, it follows that
\begin{align*}
\sqrt{S(\mc Z)}
&\le
\left(
\bbE\int_0^1
\norm{\bm D-\bm v_t(\bm Z_t)}_2^2\,dt
\right)^{1/2} +
\ell_K
\left(
\bbE\int_0^1
\norm{\bm Z_t-\bm H_t}_2^2\,dt
\right)^{1/2}\\
&=
\sqrt{\mc T}
+
\ell_K
\left(
\bbE\int_0^1
\norm{\bm Z_t-\bm H_t}_2^2\,dt
\right)^{1/2}.
\end{align*}

Next, define $\bm h_t=\bm Z_t-\bm H_t$.
Since $\bm h_0=\bm h_1=\bm 0$ and
\[
\frac{d}{dt}\bm h_t
=
\bm v_t(\bm Z_t)-\bm D.
\]
Then, Wirtinger's inequality (\cf. \cite{hardy1952inequalities}, p. 184 -- 187)
yields
\[
\int_0^1\norm{\bm h_t}_2^2\,dt
\le
\frac{1}{\pi^2}
\int_0^1
\norm{\bm v_t(\bm Z_t)-\bm D}_2^2\,dt.
\]
Taking expectations gives
\[
\bbE\int_0^1
\norm{\bm Z_t-\bm H_t}_2^2\,dt
\le
\frac{\mc T}{\pi^2}.
\]
Consequently,
\[
\sqrt{S(\mc Z)}
\le
\left(1+\frac{\ell_K}{\pi}\right)
\sqrt{\mc T},
\]
which proves the claim.\hfill $\square$

\subsection{Proof of Theorem~\ref{thm:trajectory_clt}}

Let $n:=N$, and recall from the discussion preceding \eqref{eq:trajectory_estimator} that, conditional on $\mc F_{K-1}$ (equivalently, conditional on the learned velocity field), the residuals $\{\psi_i^{(K)}\}_{i=1}^n$ are i.i.d. with conditional mean $\bbE[\psi_1^{(K)}\mid\mc F_{K-1}]=\mc T^{(K)}$.

\paragraph{Consistency.} Assume $\bbE\bigl[\lvert\psi_1^{(K)}\rvert\mid\mc F_{K-1}\bigr]<\infty$ almost surely. Then, by the conditional weak law of large numbers (applied on the conditional probability space given $\mc F_{K-1}$),
\[
\widehat{\mc T}_N^{(K)}
=\frac1n\sum_{i=1}^n\psi_i^{(K)}
\xrightarrow{p}
\bbE[\psi_1^{(K)}\mid\mc F_{K-1}]
=\mc T^{(K)}.
\]

\paragraph{Asymptotic normality.} If in addition $\sigma_K^2:=\Var(\psi_1^{(K)}\mid\mc F_{K-1})<\infty$ almost surely, then by the (conditional) Lindeberg--L\'evy central limit theorem,
\[
\frac{\sqrt n\,\bigl(\widehat{\mc T}_N^{(K)}-\mc T^{(K)}\bigr)}{\sigma_K}
\xrightarrow{d}
\bcN\lb 0,1\rb
\qquad \text{conditional on }\mc F_{K-1} \text{ a.s.}
\]
Since $n=N$, this yields the stated convergence.\hfill $\square$

\subsection{Proof of Proposition~\ref{prop:variance}}

Let $\bar\psi_{N,K}:=\widehat{\mc T}_N^{(K)}=\frac1N\sum_{i=1}^N\psi_i^{(K)}$. Conditional on $\mc F_{K-1}$, the variables $\{\psi_i^{(K)}\}_{i\ge 1}$ are i.i.d. with $\bbE[\psi_1^{(K)}\mid\mc F_{K-1}]=\mc T^{(K)}$ and $\Var(\psi_1^{(K)}\mid\mc F_{K-1})=\sigma_K^2<\infty$.

\paragraph{Consistency of $\widehat\sigma_{N,K}^2$.} Using the standard identity,
\[
\widehat\sigma_{N,K}^2
=\frac{1}{N-1}\sum_{i=1}^N\bigl(\psi_i^{(K)}-\bar\psi_{N,K}\bigr)^2
=\frac{N}{N-1}\left(\frac1N\sum_{i=1}^N\bigl(\psi_i^{(K)}\bigr)^2-\bar\psi_{N,K}^2\right).
\]
By the (conditional) weak law of large numbers,
\[
\bar\psi_{N,K}\xrightarrow{p}\mc T^{(K)},
\qquad
\frac1N\sum_{i=1}^N\bigl(\psi_i^{(K)}\bigr)^2\xrightarrow{p}\bbE\bigl[(\psi_1^{(K)})^2\mid\mc F_{K-1}\bigr],
\]
where the second convergence uses $\bbE[(\psi_1^{(K)})^2\mid\mc F_{K-1}]<\infty$ almost everywhere. By the continuous mapping theorem, $\bar\psi_{N,K}^2\xrightarrow{p}(\mc T^{(K)})^2$, hence
\[
\frac1N\sum_{i=1}^N\bigl(\psi_i^{(K)}\bigr)^2-\bar\psi_{N,K}^2
\xrightarrow{p}
\bbE\bigl[(\psi_1^{(K)})^2\mid\mc F_{K-1}\bigr]-(\mc T^{(K)})^2
=\sigma_K^2.
\]
Finally, $\frac{N}{N-1}\to 1$, so $\widehat\sigma_{N,K}^2\xrightarrow{p}\sigma_K^2$.

\paragraph{Studentization.} By Theorem~\ref{thm:trajectory_clt}, conditional on $\mc F_{K-1}$,
\[
\frac{\sqrt N\,\bigl(\bar\psi_{N,K}-\mc T^{(K)}\bigr)}{\sigma_K}\xrightarrow{d}\bcN\lb 0,1\rb.
\]
Moreover, $\widehat\sigma_{N,K}^2\xrightarrow{p}\sigma_K^2$ implies $\widehat\sigma_{N,K}\xrightarrow{p}\sigma_K$ (if $\sigma_K=0$, the conclusion is trivial). Therefore,
\[
\frac{\bar\psi_{N,K}-\mc T^{(K)}}{\widehat\sigma_{N,K}/\sqrt N}
=
\frac{\sqrt N\,\bigl(\bar\psi_{N,K}-\mc T^{(K)}\bigr)}{\sigma_K}\cdot\frac{\sigma_K}{\widehat\sigma_{N,K}}
\xrightarrow{d}\bcN\lb 0,1\rb
\]
by Slutsky's theorem.\hfill $\square$

\subsection{Proof of Proposition~\ref{prop:stopping}}

The results follow immediately from
Theorem~\ref{thm:trajectory_clt},
Proposition~\ref{prop:variance},
and the dominated convergence theorem (DCT).

\paragraph{Coverage probability.} Indeed, due to Proposition~\ref{prop:variance} we have the following. 
\[
\lim_{N \to \infty}\pr\lb \mc T^{(K)} \le U_{N,K} \mid \mc F_{K-1}\rb = 1- \alpha.
\]
Now, let us denote by $\bbE_{K-1}(\cdot)$ the expectation with respect to the randomness in $\mc F_{K-1}$. Since $\pr(\mc T^{(K)} \le U_{N, K}) = \bbE_{K-1} [\pr\lb \mc T^{(K)} \le U_{N,K} \mid \mc F_{K-1}\rb]$, we can directly apply DCT to conclude 
\[
\lim_{N \to \infty} \pr(\mc T^{(K)} \le U_{N, K}) = \bbE_{K-1} \left[\lim_{N \to \infty} \pr\lb \mc T^{(K)} \le U_{N,K} \mid \mc F_{K-1}\rb \right] = 1-\alpha.
\]

\paragraph{Type-I error control:} Similar to before, conditional on $\mc F_{K-1}$, we have 

\begin{equation}
\label{eq: type-I control}
    \begin{aligned}
        \pr_{H_0}(U_{N, K} \le \delta \mid  \mc F_{K-1}) & = \pr_{H_0}\lb \frac{\widehat{\mc T}^{(K)} - \mc T^{(K)}}{\widehat \sigma_{N, K}/\sqrt{N}} \le - z_{1- \alpha}+ \frac{\delta - \mc T^{(K)}}{\widehat \sigma_{N, K}/\sqrt{N}} \mid  \mc F_{K-1}\rb\\
        & \le \pr_{H_0}\lb \frac{\widehat{\mc T}^{(K)} - \mc T^{(K)}}{\widehat \sigma_{N, K}/\sqrt{N}} \le - z_{1- \alpha} \mid \mc F_{K-1}\rb \overset{N \to \infty}{\longrightarrow} \alpha.
    \end{aligned}
\end{equation}
Therefore, we can again apply the DCT on \eqref{eq: type-I control} to conclude 
\[
\lim_{N \to \infty}\pr_{H_0}(U_{N, K} \le \delta  ) = \bbE_{K-1}[\lim_{N \to \infty} \pr_{H_0}(U_{N, K} \le  \delta \mid  \mc F_{K-1})] \le \alpha.
\]\hfill $\square$

\section{Technical results used in the proof of Theorem~\ref{ns:thm:main}}

\subsection{Endpoint maps under zero-temperature limit} \label{sec:zero-temperature-ode}

Suppose $W \in C^3 \lb \bbR ^3 \rb$  such that \eqref{cond:W-curvature} in Definition \ref{ns:def:admissible} holds. Let $ \bm{Y} _{\beta,W}$ be the random vector taking values in $\bbR^3$ with law $ p_1^\beta  \propto \exp \lb -\beta W \rb$. The velocity field $\bm{v}_{t}^\beta( \bm{z} )$ corresponding to the independent coupling $N \lb  \bm{0} , \bI _3 \rb \otimes p_1^\beta $ has the form
\[
\bm{v}_{t}^\beta\lb   \bm{z}  \rb = \bbE  \left[  \bm{Y} _{\beta,W} -  \bm{X}  \Big| (1-t)  \bm{X}  + t \bm{Y} _{\beta,W}=  \bm{z}   \right], \qquad  \bm{X}  \sim N(  \bm{0} , \bI _3)
\]
Let $\sfZ_t^{\beta,W}\lb\bm x\rb$ be the flow generated by $\bm v_t^\beta$, starting at $\bm x$. The corresponding endpoint map is $\sfZ_1^{\beta,W}$, obtained by evaluating the flow at $t=1$. We refer to the limit $\beta\to\infty$ as the zero-temperature limit. We will show below in Proposition \ref{ns:prop:zero-temperature-endpoint}(iii) that, under appropriate rescalings, both the endpoint map and its Jacobian converge as $\beta\to\infty$.

Define the posterior measure
\[
\pi_{\beta,r,  \bm{h} }  \lb d  \bm{y}  \rb=  \frac{ \exp \la -\beta \cdot \left[  W ( \bm{y} ) + \frac{1}{2} \cdot \norm{  \bm{h}  - r \bm{y}  }^2 \right] \ra}{ \int_{\bbR ^3} \exp \la -\beta \cdot \left[  W ( \bm{y} ) + \frac{1}{2} \cdot \norm{  \bm{h}  - r \bm{y}  }^2 \right] \ra \, d  \bm{y}   }.
\]
\begin{prop} \label{ns:prop:ode-r-to-infty}
Let $r = r(t) := \frac{t}{\sqrt{\beta}(1-t)}$ and set
    \[
     \bm{h} _\beta \lb r,  \bm{x}  \rb:= \lb 1+r\sqrt{\beta} \rb \cdot \frac{ \sfZ_{(\sqrt{\beta}r)/(1+\sqrt{\beta}r)}^{\beta,W} \lb \sqrt{\beta}. \,  \bm{x}  \rb}{\sqrt{\beta}}.
     \]
    It holds that 
    \begin{itemize}
        \item For each fixed $\beta>0$, $ \bm{h} _\beta$ solves the ODE
        \begin{align}\label{ns:ODE-h-beta}
        \frac{d}{dr}  \bm{h} _\beta \lb r,  \bm{x}  \rb =  \bm{y} _\beta \lb r,  \bm{h} _\beta \lb r,  \bm{x}  \rb  \rb,
        \qquad  \bm{h} _\beta(0,  \bm{x} )= \bm{x} ,
        \qquad 
         \bm{y} _\beta = \bbE _{ \bm{Y}  \sim \pi_{\beta,r,  \bm{h} }} \left[  \bm{Y}  \right].
        \end{align}
        \item Locally on compact sets,
        \[
         \sfZ_1^{\beta,W}  \lb \sqrt{\beta} \,  \bm{x}  \rb =  \lim_{r \to \infty} \frac{ \bm{h} _\beta \lb r,  \bm{x}  \rb}{r}, \qquad 
        \text{and} \qquad 
        \sqrt{\beta} \cdot D  \sfZ_1^{\beta,W}  \lb \sqrt{\beta} \,  \bm{x}  \rb =  \lim_{r \to \infty} \frac{D \bm{h} _\beta \lb r,  \bm{x}  \rb}{r}.
        \]
    \end{itemize}
\end{prop}

\noindent \textbf{Proof of Proposition \ref{ns:prop:ode-r-to-infty}.}
Fix $\beta>0$ and write
\[
    t=t(r):=\frac{\sqrt\beta\,r}{1+\sqrt\beta\,r},
    \qquad \text{and} \qquad
    \frac{dt}{dr}=\sqrt\beta(1-t)^2.
\]
By definition, we have 
\[
    \bm{v}_{t}^\beta( \bm{z} )
    =
    \frac{ \bm{y} _\beta(r, \bm{h} )- \bm{z} }{1-t}.
\]
Thus
\begin{align*}
    \partial_r \bm{h} _\beta(r, \bm{x} )
    &=
    \sfZ_{t}^{\beta,W}(\sqrt\beta\, \bm{x} )
    +(1-t) \cdot 
      \bm{v}_{t}^\beta\lb \sfZ_{t}^{\beta,W}(\sqrt\beta\, \bm{x} )\rb 
      =
     \bm{y} _\beta\lb r, \bm{h} _\beta(r, \bm{x} )\rb .
\end{align*}
Also,
\[
     \bm{h} _\beta(0, \bm{x} )
    =\frac{\sfZ_{0}^{\beta,W}(\sqrt\beta\, \bm{x} )}{\sqrt\beta}
    = \bm{x} .
\]
This proves \eqref{ns:ODE-h-beta}.

We now prove the second claim. Lemmas \ref{ns:lem:zero-temp} and \ref{ns:lem:sandwich} give
\begin{align} \label{ns:eq:posterior-first-moment}
    \norm{  \bm{y} _\beta(r, \bm{h} )}
    &\leq \bbE _{\pi_{\beta,r, \bm{h} }}\norm{ \bm{Y} }
    \leq \frac{r\norm{ \bm{h} }+\norm{\nabla W( \bm{0} )}}{\kappa+r^2}
    +\sqrt{\frac{3}{\beta(\kappa+r^2)}}.
\end{align}
Together with the Lipschitz bound in Lemma \ref{ns:lem:sandwich}, this ensures that the ODE \eqref{ns:ODE-h-beta} has a unique global
solution, and is $C^1$ with respect to $ \bm{x} $. We split the rest of the proof into a few steps.

\underline{\it Step 1: Bounding $\norm{ \bm{h} _\beta(r, \bm{x} )}$.} We will show that for all fixed $ \bm{x}  \in \bbR ^3$,
\begin{equation}  \label{ns:norm-h-beta}
    \norm{ \bm{h} _\beta(r, \bm{x} )}
    \leq
    \sqrt{\frac{\kappa+r^2}{\kappa}}
    \lb 
        \norm{ \bm{x} }
        +\frac{\norm{\nabla W( \bm{0} )}}{\sqrt{\kappa}}
        +\frac{\pi}{2}\sqrt{\frac{3}{\beta}}
    \rb. 
\end{equation}
To show this, put $f(r):= \norm{  \bm{h} _\beta(r, \bm{x} )}$. Triangle inequality yields, for all $\delta>0$
\begin{align*}
    |f(r+\delta)-f(r)| \leq \norm{  \bm{h} _\beta(r+\delta, \bm{x} ) -  \bm{h} _\beta(r, \bm{x} ) } &\leq \int_{r}^{r+\delta}
    \norm{ \partial_s   \bm{h} _\beta(s, \bm{x} )    } \, ds \\
    &= \int_{r}^{r+\delta} \norm{  \bm{y} _\beta \lb s, \bm{h} _\beta(s, \bm{x} ) \rb } \, ds.
\end{align*}
Dividing both sides by $\delta$, letting $\delta\to0$, and using \eqref{ns:eq:posterior-first-moment}, we deduce that, for almost every $r>0$,
\begin{align*}
    |f'(r)| \leq \norm{  \bm{y} _\beta \lb r,  \bm{h} _\beta(r, \bm{x} ) \rb } &\leq
    \frac{r \cdot \norm{ \bm{h} _\beta(r, \bm{x} ) }+ \norm{\nabla W( \bm{0} )}}{\kappa+r^2}
    +\sqrt{\frac{3}{\beta(\kappa+r^2)}} \\
    &= \frac{r}{\kappa+r^2} \cdot f(r) +     \frac{\norm{\nabla W( \bm{0} )}}{\kappa+r^2} +\sqrt{\frac{3}{\beta(\kappa+r^2)}}.
\end{align*}
Lemma \ref{ns:lem:gronwall-inhomogeneous} then gives
\begin{align*}
    f(r) &\leq \norm{  \bm{x}  } \cdot \exp \la \int_0^r \frac{u}{\kappa+u^2} \, du \ra
    + \int_0^r \left[  \frac{ \norm{\nabla W( \bm{0} )} }{\kappa+s^2}
    +\sqrt{\frac{3}{\beta(\kappa+s^2)}} \, \right] \cdot \exp \la \int_s^r \frac{u}{\kappa+u^2} \, du \ra \, ds \\
    &=  \norm{  \bm{x}  } \cdot \sqrt{\frac{\kappa+r^2}{\kappa}} + 
    \sqrt{\kappa+r^2} \cdot \int_0^r  \left[  \frac{\norm{\nabla W( \bm{0} )}}{\lb \kappa + s^2 \rb^{3/2}} + \frac{\sqrt{3/\beta}}{\kappa+s^2}  \right] \, ds.
\end{align*}
Finally, the elementary bounds 
\[
\int_0^r \frac{ds}{\lb \kappa+s^2 \rb^{3/2}} \leq \frac 1\kappa, \qquad  \int_0^r \frac{ds}{\kappa+s^2} \leq \frac{\pi}{2\sqrt{\kappa}}
\]
give \eqref{ns:norm-h-beta}.

\underline{\it Step 2: Local convergence of $ \bm{h} _\beta$.}
By using  Stein's identity (see Lemma 2.2 in \cite{pmlr-v48-liub16})
$$\bbE _{ \bm{X}  \sim p} \left[ \nabla \log p \lb  \bm{X}  \rb \right] =  \bm{0} $$ 
 on $\pi_{\beta,r, \bm{h} }$, we obtain 
\[
    \bbE _{\pi_{\beta,r, \bm{h} _\beta}(r, \bm{x} )}[\nabla W( \bm{Y} )]
    +r^2 \cdot  \bm{y} _\beta \lb r, \bm{h} _\beta(r, \bm{x} ) \rb-r\cdot  \bm{h} _\beta (r, \bm{x} )
    = \bm{0} .
\]
Consequently, along the transformed flow,
\begin{equation}
\label{ns:eq:normalized-flow-derivative}
    \partial_r\left[\frac{  \bm{h} _\beta(r, \bm{x} )}r\right]
    = \frac{\partial_r  \bm{h} _\beta(r, \bm{x} )}{r} - \frac{ \bm{h} _\beta(r, \bm{x} )}{r^2}
  =  -\frac1{r^3} \cdot 
    \bbE _{\pi_{\beta,r, \bm{h} _\beta(r, \bm{x} )}}
       [\nabla W( \bm{Y} )].
\end{equation}
Recall that $ \bm{y} (r, \bm{h} )$ is the mode of the posterior distribution $\pi_{\beta,r, \bm{h} }$, defined as in \eqref{ns:y(r,h)} below. For any compact $K\subset\bbR ^3$, \eqref{ns:norm-h-beta} and Lemma \ref{ns:lem:sandwich} yield
\begin{align*}
\sup_{\substack{r\geq1\\ \bm x\in K}} \bbE_{\pi_{\beta,r,\bm h_\beta(r,\bm x)}}
\norm{\bm Y} &\leq \sup_{\substack{r\geq1\\ \bm x\in K}}
\norm{\bm y\lb r,\bm h_\beta(r,\bm x)\rb}
+\sqrt{\frac{3}{\beta(\kappa+1)}}\\
&\leq
C_K+\sqrt{\frac{3}{\beta(\kappa+1)}}<\infty.
\end{align*}
On the other hand, since
\[
    \norm{\nabla W( \bm{u} )}
    \leq\norm{\nabla W( \bm{0} )}+\Lambda\norm{ \bm{u} },
\]
the expectation on the right-hand side of
\eqref{ns:eq:normalized-flow-derivative} is also uniformly bounded.
Hence, for $s\geq R\geq1$,
\[
\begin{aligned}
\sup_{\bm{x}\in K}
\norm{
    \frac{\bm{h}_\beta(s,\bm{x})}{s}
    -\frac{\bm{h}_\beta(R,\bm{x})}{R}
}
&\leq
\int_R^s
\sup_{\bm{x}\in K}
\norm{
    \partial_u\left[
        \frac{\bm{h}_\beta(u,\bm{x})}{u}
    \right]
}\,du \\
&\leq
C_{\beta,K,\kappa,\Lambda}
\int_R^s u^{-3}\,du \leq
\frac{C_{\beta,K,\kappa,\Lambda}}{2R^2}.
\end{aligned}
\]
Thus $ \bm{h} _\beta(r,\cdot)/r$ converges locally uniformly.

By definition of $ \bm{h} _\beta$,
\[
    \sfZ_{t(r)}^{\beta,W}(\sqrt\beta\, \bm{x} )
    =t(r)\frac{ \bm{h} _\beta(r, \bm{x} )}r.
\]
Since $t(r)\to1$, the preceding convergence establishes the
continuous endpoint and gives
\[
     \sfZ_1^{\beta,W} (\sqrt\beta\, \bm{x} )
    = \lim_{t(r) \to 1}     \sfZ_{t(r)}^{\beta,W}(\sqrt\beta\, \bm{x} ) =
    \lim_{r\to\infty}\frac{ \bm{h} _\beta(r, \bm{x} )}r,
\]
locally uniformly in $ \bm{x} $.

\underline{\it Step 3: Local convergence of the normalized spatial derivatives.}
Put
$
    \bK _\beta(r, \bm{x} )
    :=D_{ \bm{x} } \bm{h} _\beta(r, \bm{x} ).
$
Differentiating \eqref{ns:ODE-h-beta} with respect to $ \bm{x} $ gives
\[
    \partial_r\bK _\beta(r, \bm{x} )
    =
    D_{ \bm{h} } \bm{y} _\beta
    \lb r, \bm{h} _\beta(r, \bm{x} )\rb  \cdot 
    \bK _\beta(r, \bm{x} ),
    \qquad
    \bK _\beta(0, \bm{x} )=\bI _3.
\]
By Lemma~\ref{ns:lem:sandwich},
\[
    \frac{r}{\Lambda+r^2}\bI _3
    \preceq
    D_{ \bm{h} } \bm{y} _\beta
    \lb r, \bm{h} _\beta(r, \bm{x} )\rb 
    \preceq
    \frac{r}{\kappa+r^2}\bI _3.
\]
Consequently, for every fixed $ \bm{v} \in\bbR ^3$,
\begin{align*}
    \frac{d}{dr}\norm{\bK _\beta(r, \bm{x} ) \bm{v} }^2 &= 2  \bm{v} ^\top \cdot \bK^\top _\beta \lb r, \bm{x}   \rb
    \cdot D_{ \bm{h} } \bm{y} _\beta
    \lb r, \bm{h} _\beta(r, \bm{x} )\rb  \cdot \bK _{\beta} \lb r,  \bm{x}  \rb \cdot \bm v\\
    &\leq
    \frac{2r}{\kappa+r^2} \cdot 
    \norm{\bK _\beta(r, \bm{x} ) \bm{v} }^2.
\end{align*}
Gronwall's inequality yields
\[
    \norm{\bK _\beta(r, \bm{x} ) \bm{v} }^2
    \leq
    \exp\lb \int_0^r\frac{2s}{\kappa+s^2}\,ds\rb
    \norm{ \bm{v} }^2
    =
    \frac{\kappa+r^2}{\kappa}\norm{ \bm{v} }^2.
\]
Taking the supremum over unit vectors $ \bm{v} $ gives
\[
    \norm{\bK _\beta(r, \bm{x} )}
    \leq\sqrt{\frac{\kappa+r^2}{\kappa}}.
\]
Moreover, the sandwich bound  above also implies, for $r>0$,
\[
     \bm{0} 
    \preceq
    \frac1r\bI _3-
    D_{ \bm{h} } \bm{y} _\beta
    \lb r, \bm{h} _\beta(r, \bm{x} )\rb 
    \preceq
    \frac{\Lambda}{r(\Lambda+r^2)}\bI _3.
\]
Thus
\[
    \norm{
    D_{ \bm{h} } \bm{y} _\beta
    \lb r, \bm{h} _\beta(r, \bm{x} )\rb 
    -\frac1r\bI _3
    }
    \leq\frac{\Lambda}{r(\Lambda+r^2)}.
\]
Differentiating the normalized matrix $\bK _\beta/r$ therefore gives
\begin{align*}
    \norm{
    \partial_r \la \frac{\bK _\beta(r, \bm{x} )}r \ra
    }
    &=
    \norm{
    \la
    D_{ \bm{h} } \bm{y} _\beta
    \lb r, \bm{h} _\beta(r, \bm{x} )\rb 
    -\frac1r\bI _3
    \ra \cdot 
    \frac{\bK _\beta(r, \bm{x} )}r
    }\\
    &\leq
    \frac{\Lambda}{r^2(\Lambda+r^2)}
    \sqrt{\frac{\kappa+r^2}{\kappa}}
    \leq\frac{C}{r^3},
\end{align*}
for all $r\geq 1$, where $C_{\kappa,\Lambda}<\infty$ depends only on $\kappa$ and $\Lambda$.

Thus, for $S>R\geq1$,
\begin{align} \label{ns:normalized-Jacobian}
    \sup_{ \bm{x} \in\bbR ^3}
    \norm{
    \frac{\bK _\beta(S, \bm{x} )}S
    -\frac{\bK _\beta(R, \bm{x} )}R
    }
    \leq
    \int_R^S\frac{C_{\kappa,\Lambda}}{r^3}\,dr
    \leq\frac{C_{\kappa,\Lambda}}{2R^2}.
\end{align}
Hence $\bK _\beta(r,\cdot)/r$ converges uniformly, and in particular locally uniformly. Let this limit be 
\[
\bL _\beta \lb  \bm{x}  \rb:= \lim_{r \to \infty} \frac{\bK _\beta(r, \bm{x} )}{r}.
\]
It suffices to show that $\bL _\beta( \bm{x} ) = \sqrt{\beta} \cdot D  \sfZ_1^{\beta,W}  \lb \sqrt{\beta} \,  \bm{x}  \rb$. Consider a coordinate vector $ \bm{e} _j$. For all $r,t >0$, we have
\[
\frac{ \bm{h} _\beta(r,  \bm{x}  + t  \bm{e} _j ) -  \bm{h} _\beta(r, \bm{x} )}{r} = \int_0^t  \frac{\bK _\beta \lb r,  \bm{x}  + s  \bm{e} _j \rb}{r} \cdot  \bm{e} _j \, ds
\]
Let $r \to \infty$, using the local convergence in Step 2, we deduce that
\[
 \sfZ_1^{\beta,W}  \lb \sqrt{\beta} \lb  \bm{x}  + t  \bm{e} _j \rb \rb -  \sfZ_1^{\beta,W}  \lb \sqrt{\beta}  \bm{x}  \rb = \int_0^t \bL _\beta \lb
 \bm{x}  + s  \bm{e} _j
\rb \bm e_j \, ds
\]
Divide both sides by $t$ and send $t\to 0$, we get
\[
\partial_{ x _j} \la  \sfZ_1^{\beta,W}  \lb \sqrt{\beta}  \bm{x}  \rb \ra = \bL _\beta( \bm{x} ) \cdot  \bm{e} _j. 
\]
Since $j$ is an arbitrary coordinate, the conclusion follows.
\hfill$\square$

Define
\begin{align} \label{ns:y(r,h)}
 \bm{y} (r, \bm{h} )
:=
\operatorname*{argmin}_{ \bm{u} \in\bbR ^3}
\la 
W( \bm{u} )+\frac12\norm{ \bm{h} -r \bm{u} }^2
\ra.
\end{align}
Equivalently, $ \bm{y} (r, \bm{h} )$ is the unique solution of
\[
\nabla W( \bm{y} )+r^2 \bm{y} -r \bm{h} = \bm{0} .
\]
We interpret the above as $ \bm{y}  \lb r,  \bm{h}  \rb$ being a function of $r$ and $ \bm{h} $: such a function is well-defined because due to strong convexity of $W$,  there exists a unique $ \bm{y} $ satisfying the above equality for fixed $r, \bm{h} $. 
Lemma \ref{ns:lem:y(r,h)-smooth} yields that, as a function of $\lb r,  \bm{h}  \rb$, $\bm{y} \lb r,  \bm{h}   \rb$ is of class $C^1$.

\begin{prop}
\label{ns:prop:zero-temperature-endpoint}
Under the same assumptions as in Proposition \ref{ns:prop:ode-r-to-infty}, it holds that
\begin{enumerate}[label=\textnormal{(\roman*)},leftmargin=2.8em]

\item There exists a unique map
$
 \bm{h} \in C^1\lb [0,\infty)\times\bbR ^3;\bbR ^3\rb 
$
satisfying
\begin{align} \label{ns:zero-temp-ODE}
\frac{\partial}{\partial r} \bm{h} (r, \bm{x} )
=
 \bm{y} \lb r, \bm{h} (r, \bm{x} )\rb ,
\qquad
 \bm{h} (0, \bm{x} )= \bm{x} .
\end{align}
Moreover,
$  \bm{h} _\beta\to  \bm{h}  \quad\text{in }
C^1_{\mathrm{loc}}
\lb [0,\infty)\times\bbR ^3;\bbR ^3\rb .
$

\item The limits
\[
\bm{\mc{P}}( \bm{x} )
:=
\lim_{r\to\infty}\frac{ \bm{h} (r, \bm{x} )}{r},
\qquad
D\bm{\mc{P}}( \bm{x} )
=
\lim_{r\to\infty}\frac{D_{ \bm{x} } \bm{h} (r, \bm{x} )}{r}
\]
exist and the convergences above hold locally uniformly in $ \bm{x} $. Moreover, we have 
$\bm{\mc{P}}\in C^1(\bbR ^3;\bbR ^3)$.

\item If
$
\bm{\mc{P}}_\beta( \bm{x} )
:=
 \sfZ_1^{\beta,W} (\sqrt{\beta}\, \bm{x} ),
$
then
\[
\bm{\mc{P}}_\beta
\longrightarrow
\bm{\mc{P}}
\qquad\text{in }C^1_{\mathrm{loc}}(\bbR ^3;\bbR ^3).
\]

\end{enumerate}
\end{prop}

To be precise, Proposition \ref{ns:prop:zero-temperature-endpoint}(i) claims that for every $R<\infty$ and every compact
$K\subset\bbR ^3$,
\[
\begin{split}
\sup_{\substack{0\leq r\leq R\\  \bm{x} \in K}}
\la
\norm{ \bm{h} _\beta(r, \bm{x} )- \bm{h} (r, \bm{x} )}
+\norm{\partial_r \bm{h} _\beta(r, \bm{x} )
       -\partial_r \bm{h} (r, \bm{x} )}
+\norm{D_{ \bm{x} } \bm{h} _\beta(r, \bm{x} )
       -D_{ \bm{x} } \bm{h} (r, \bm{x} )}
\ra
\to 0.
\end{split}
\]
Similarly, Proposition \ref{ns:prop:zero-temperature-endpoint}(iii) yields that  $D \bm{\mc{P}} _\beta \to D \bm{\mc{P}} $ uniformly over any compact set $K \subset \bbR ^3$
\[
\sup_{ \bm{x}  \in K} \norm{  D \bm{\mc{P}} _\beta ( \bm{x} ) - D \bm{\mc{P}}  ( \bm{x} ) }_{\rm op} \to 0.
\]
 Let us explain briefly why these results should be expected. The ODE in Proposition \ref{ns:prop:zero-temperature-endpoint}(i) has a $C^1$ solution $\bm h$ due to Lemma \ref{ns:lem:zero-temp} and the classical existence results for ODEs. The convergence $\bm h_\beta\to\bm h$ is suggested by taking $\beta\to\infty$ in \eqref{ns:ODE-h-beta} and observing that the posterior mean under $\pi_{\beta,r,\bm h}$ converges to the minimizer:
\[
         \bm{y} _\beta \lb r,  \bm{h}  \rb = \bbE _{ \bm{Y}  \sim \pi_{\beta,r,  \bm{h} }} \left[  \bm{Y}  \right] \stackrel{\beta\to\infty}{\longrightarrow}
         \bm{y} (r, \bm{h} ) = \operatorname*{argmin}_{ \bm{u} \in\bbR ^3} \la  W( \bm{u} )+\frac12\norm{ \bm{h} -r \bm{u} }^2\ra,
\]
uniformly over compact subsets.

Propositions \ref{ns:prop:zero-temperature-endpoint}(ii) and (iii) also follow from taking the derivatives and sending $\beta \to \infty$, which can be justified rigorously by bounding the derivatives of $ \bm{y} $. We note that the pair $( \bm{y} ,  \bm{h} )$ is best understood in the ODE sense. If one views $ \bm{y} $, $ \bm{h} $ as functions of $r$  and considers a flow $ \bm{h}  =  \bm{h} (r)$ from the ODE \eqref{ns:zero-temp-ODE}, then the joint trajectory $\lb  \bm{y} (r),  \bm{h} (r)  \rb$ satisfies
\[
\nabla  W  \lb  \bm{y} (r) \rb + r^2\cdot  \bm{y} (r) - r \cdot  \bm{h}  \lb r,  \bm{x}  \rb =  \bm{0} 
\]
with $ \bm{y} (r)=  \bm{y}  \lb r,  \bm{h} (r, \bm{x} ) \rb$.

\noindent\textbf{Proof of Proposition
\ref{ns:prop:zero-temperature-endpoint}.}
Put
\[
    \Phi( \bm{u} )
    :=W( \bm{u} )+\frac12\norm{ \bm{h} -r \bm{u} }^2,
    \qquad
    \bA (r, \bm{h} )
    :=\nabla^2\Phi_{ \bm{y} } ( \bm{y} (r, \bm{h} )) =  \nabla^2 W \lb  \bm{y} (r, \bm{h} ) \rb + r^2 \bI _3.
\]
All expectations below are taken with respect to $\pi_{\beta,r, \bm{h} }$.

\medskip
\underline{\it Proof of (i).}
The existence and uniqueness of $ \bm{h} \in C^1\lb [0,\infty)\times\bbR ^3;\bbR ^3\rb $ follow from  Lemma~\ref{ns:lem:sandwich}, which yields the Lipschitz property of $ \bm{y} $, and Lemma \ref{ns:lem:y(r,h)-smooth}, which yields the differentiability of $ \bm{y} $. Put
\[
    \bK _\beta \lb r,  \bm{x}  \rb:=D_{ \bm{x} } \bm{h} _\beta \lb r,  \bm{x}  \rb,
    \qquad
    \bK  \lb r,  \bm{x}  \rb:=D_{ \bm{x} } \bm{h}  \lb r,  \bm{x}  \rb.
\]
The pair $\lb  \bm{h} _\beta, \bK _\beta \rb$ solves the coupled ODE 
\[
\frac{d}{dr}\begin{pmatrix}
     \bm{h} _\beta \lb r,  \bm{x}  \rb \\
    \bK _\beta(r,  \bm{x} )
\end{pmatrix}
= \begin{pmatrix}
     \bm{y} _\beta \lb r,  \bm{h} _\beta \rb \\
    D  \bm{y} _{\beta} \lb r,  \bm{h} _\beta \rb \cdot \bK _\beta 
\end{pmatrix}, 
\qquad 
\begin{pmatrix}
     \bm{h} _\beta  \\ \bK _\beta 
\end{pmatrix}
\lb 0,  \bm{x}   \rb = \lb  \bm{x} , \bI _3 \rb^\top.
\]
On the other hand, $\lb  \bm{h} , \bK  \rb$ solves the coupled ODE
\[
\frac{d}{dr}\begin{pmatrix}
     \bm{h}  \lb r,  \bm{x}  \rb \\
    \bK (r,  \bm{x} )
\end{pmatrix}
= \begin{pmatrix}
     \bm{y}  \lb r,  \bm{h}  \rb \\
    D  \bm{y}  \lb r,  \bm{h}  \rb \cdot \bK  
\end{pmatrix}, 
\qquad 
\begin{pmatrix}
     \bm{h}   \\ \bK  
\end{pmatrix}
\lb 0,  \bm{x}   \rb = \lb  \bm{x} , \bI _3 \rb^\top.
\]
By perturbation results for ODEs (see, for example, Theorem 2.8 in \cite{teschl2012ordinary}), we only need to show that uniformly over (space-time) compact sets,
\[
 \bm{y} _\beta \lb r,  \bm{h}  \rb \stackrel{\beta \to \infty}{\longrightarrow}  \bm{y}   \lb r,  \bm{h}  \rb
\qquad 
\text{and} \qquad
D \bm{y} _\beta \lb r,  \bm{h}  \rb \stackrel{\beta \to \infty}{\longrightarrow} D \bm{y}  \lb r,  \bm{h}  \rb.
\]
The first convergence is immediate from Lemma~\ref{ns:lem:sandwich}, because
\[
    \norm{ \bm{y} _\beta(r, \bm{h} )- \bm{y} (r, \bm{h} )}
    \leq 
    \sqrt{\frac{3}{\beta(\kappa+r^2)}}.
\]
To prove the convergence of the Jacobian, note that
\[
D_{ \bm{h} }  \bm{y} _\beta(r, \bm{h} ) = \beta r \cdot \Cov \lb  \bm{Y}  \rb, 
\qquad 
D_{ \bm{h} }  \bm{y} (r, \bm{h} ) = r \cdot \left[  \nabla^2 W \lb  \bm{y} (r, \bm{h} ) \rb + r^2 \bI _3 \right]^{-1} = r \cdot \bA (r, \bm{h} )^{-1}.
\]
It suffices to show that
\[
\beta \cdot \Cov \lb  \bm{Y}  \rb \stackrel{\beta \to \infty}{\longrightarrow}  \bA (r, \bm{h} )^{-1}.
\]
The Brascamp--Lieb inequality and the Cram\'er--Rao lower bound yield
\[
 \la \bbE_{\bm Y \sim \pi_{r,\beta, \bm h }}  \left[   \nabla ^2 W( \bm{Y} ) + r^2 \bI _3  \right] \ra^{-1}
\preceq 
\beta \cdot \Cov \lb  \bm{Y}  \rb \preceq
\bbE_{\bm Y \sim \pi_{r,\beta,\bm h}} \la  \left[     \nabla ^2 W( \bm{Y} ) + r^2 \bI _3 \right]^{-1} \ra
\]
Put $\bH =    \nabla ^2 W( \bm{Y} ) + r^2 \bI _3$. By using Lemma \ref{ns:lem:inverse-lipschitz}, we obtain 
\begin{align*}
    \norm{ \lb \bbE_{\bm Y \sim \pi_{r,\beta,\bm h}}  \bH   \rb^{-1} - \bA (r, \bm{h} )^{-1}  } &\leq \frac{    \bbE_{\bm Y \sim \pi_{r,\beta,\bm h}}  \norm{  \nabla^2 W \lb  \bm{Y} \rb - \nabla^2 W \lb  \bm{y} (r, \bm{h} ) \rb  }}{\kappa^2} \\
    \norm{ \bbE_{\bm Y \sim \pi_{r,\beta,\bm h}}  \left[ \bH ^{-1} \right] - \bA (r, \bm{h} )^{-1} } &\leq  \frac{    \bbE_{\bm Y \sim \pi_{r,\beta,\bm h}}  \norm{  \nabla^2 W \lb  \bm{Y} \rb - \nabla^2 W \lb  \bm{y} (r, \bm{h} ) \rb  }}{\kappa^2}
\end{align*}
Consequently,
\begin{align*}
  -\frac{    \bbE_{\bm Y \sim \pi_{r,\beta,\bm h}}  \norm{  \nabla^2 W \lb  \bm{Y} \rb - \nabla^2 W \lb  \bm{y} (r, \bm{h} ) \rb  }}{\kappa^2} \cdot \bI _3 &\preceq
  \beta \cdot \Cov \lb  \bm{Y}  \rb - \bA (r, \bm{h} )^{-1} \\
  &\preceq    \frac{    \bbE_{\bm Y \sim \pi_{r,\beta,\bm h}}  \norm{  \nabla^2 W \lb  \bm{Y} \rb - \nabla^2 W \lb  \bm{y} (r, \bm{h} ) \rb  }}{\kappa^2} \cdot \bI _3
\end{align*}
It remains to bound the expectation term in the last display. Fix $R<\infty$ and a compact set $B\subset\bbR^3$.
Consider the compact set $\{\bm y(r,\bm h):r\in[0,R],\ \bm h\in B\}$.
By uniform continuity of $\nabla^2W$ on a compact
neighborhood of this set, for every $\ve>0$ there exists
a common $\delta=\delta(\ve,R,B)>0$ such that
\[
\norm{\bm u-\bm y(r,\bm h)}\leq\delta \implies
\norm{ \nabla^2W\lb\bm u\rb - \nabla^2W\lb\bm y(r,\bm h)\rb}
\leq\ve
\]
for all $r\in[0,R]$ and $\bm h\in B$.

Consequently,
\begin{align*}
\sup_{\substack{r\in[0,R]\\ \bm h\in B}}
\bbE\norm{
    \nabla^2W\lb\bm Y\rb
    -\nabla^2W\lb\bm y(r,\bm h)\rb
} 
&\quad\leq
\ve+2\Lambda
\sup_{\substack{r\in[0,R]\\ \bm h\in B}}
\mb P\lb
    \norm{\bm Y-\bm y(r,\bm h)}\geq\delta
\rb \\
&\quad\leq
\ve+\frac{6\Lambda}{\beta\kappa\delta^2}.
\end{align*}
where the second inequality follows from Markov's
inequality and the second claim in
Lemma~\ref{ns:lem:sandwich}.
Using the preceding covariance bounds, we obtain
\[
\sup_{\substack{r\in[0,R]\\ \bm h\in B}}
\norm{
D_{\bm h}\bm y_\beta(r,\bm h)
-
D_{\bm h}\bm y(r,\bm h)
}
\leq
\frac{R}{\kappa^2}
\left(
\ve+\frac{6\Lambda}{\beta\kappa\delta^2}
\right).
\]
Taking $\limsup_{\beta\to\infty}$ and then letting
$\ve\downarrow0$ proves that
$D_{\bm h}\bm y_\beta\to D_{\bm h}\bm y$
locally uniformly in $(r,\bm h)$.

\medskip
\underline{\it Proof of (ii) and (iii).}
For a compact set $K\subset\bbR ^3$, write
\[
    \norm{ \bm{F} }_{C^1(K)}
    :=\sup_{ \bm{x} \in K}\norm{ \bm{F} ( \bm{x} )}
      +\sup_{ \bm{x} \in K}\norm{D \bm{F} ( \bm{x} )}.
\]
The estimates in Steps~1--3 of the proof of
Proposition~\ref{ns:prop:ode-r-to-infty} give
\[
    \norm{
    \frac{ \bm{h} _\beta(S,\cdot)}S
    -\frac{ \bm{h} _\beta(R,\cdot)}R
    }_{C^1(K)}
    \leq\frac{C_K}{R^2},
    \qquad S\geq R\geq1,
\]
with $C_K$ independent of $\beta\geq1$.
Indeed, \eqref{ns:norm-h-beta} and \eqref{ns:normalized-Jacobian} are uniform over $\beta\geq1$.

For fixed $S$ and $R$, letting $\beta\to\infty$ and using (i), we obtain
\[
    \norm{
    \frac{ \bm{h} (S,\cdot)}S
    -\frac{ \bm{h} (R,\cdot)}R
    }_{C^1(K)}
    \leq\frac{C_K}{R^2}.
\]
Hence the normalized flow and its Jacobians are locally uniformly Cauchy.
The same fundamental theorem of calculus argument as in
Proposition~\ref{ns:prop:ode-r-to-infty} identifies the
derivative limit with the Jacobian of the function limit.
Thus $\bm{\mc{P}}\in C^1$ and
\[
    \frac{ \bm{h} (R,\cdot)}R
    \longrightarrow\bm{\mc{P}}
    \quad\text{in }C^1_{\mathrm{loc}},
\]
proving (ii).

Letting $S\to\infty$ in the two  estimates above gives
\[
    \norm{
    \bm{\mc{P}}_\beta-\frac{ \bm{h} _\beta(R,\cdot)}R
    }_{C^1(K)}
    \leq\frac{C_K}{R^2},
    \qquad
    \norm{
    \bm{\mc{P}}-\frac{ \bm{h} (R,\cdot)}R
    }_{C^1(K)}
    \leq\frac{C_K}{R^2}.
\]
Therefore,
\[
    \norm{\bm{\mc{P}}_\beta-\bm{\mc{P}}}_{C^1(K)}
    \leq
    \frac{2C_K}{R^2}
    +\frac1R
    \norm{ \bm{h} _\beta(R,\cdot)- \bm{h} (R,\cdot)}_{C^1(K)}.
\]
For fixed $R$, the last term tends to zero by (i).
Taking $\limsup_{\beta\to\infty}$ and then letting
$R\to\infty$ proves (iii).
\hfill$\square$

\subsection{ODE results}

\begin{lemma}[Gronwall-type inequality]
\label{ns:lem:gronwall-inhomogeneous}
Let $T>0$, let $a,b\in L^1([0,T];\bbR )$, and suppose
$f:[0,T]\to\bbR $ is absolutely continuous. Assume that $a(r)\geq0$ for almost every $r\in[0,T]$ and 
\[
    f'(r)\leq a(r)f(r)+b(r)
    \qquad\text{for almost every }r\in[0,T].
\]
Then, for every $r\in[0,T]$,
\[
    f(r)
    \leq
    f(0) \cdot \exp\la \int_0^r a(s)\,ds\ra
    +
    \int_0^r b(s) \cdot 
    \exp\la \int_s^r a(u)\,du\ra ds.
\]
\end{lemma}

\noindent \textbf{Proof of Lemma \ref{ns:lem:gronwall-inhomogeneous}.}
Set $\alpha(r):=f(0)+\int_0^r b(s)\,ds$.
Integrating the differential inequality gives
\[
    f(r)\leq\alpha(r)+\int_0^r a(s)f(s)\,ds.
\]
By Lemma~2.7 in \cite{teschl2012ordinary},
\begin{align*}
    f(r)
    &\leq
    \alpha(r)+\int_0^r
    \alpha(s) \cdot a(s) \cdot 
    \exp\la \int_s^r a(u)\,du\ra ds\\
    &=
    f(0) \cdot \exp\la \int_0^r a(u)\,du\ra
    +\int_0^r b(s) \cdot 
    \exp\la \int_s^r a(u)\,du\ra ds,
\end{align*}
To justify the last equality, fix $r$ and observe that
\[
    \frac{d}{ds}
    \exp\la \int_s^r a(u)\,du\ra
    =
    -a(s)\exp\la \int_s^r a(u)\,du\ra
\]
for almost every $s\in[0,r]$. Hence integration by parts gives
\begin{align*}
    &\int_0^r
    \alpha(s)a(s)
    \exp\la \int_s^r a(u)\,du\ra ds\\
    =&
    -\left[
       \alpha(s)\exp\la \int_s^r a(u)\,du\ra
    \right]_{s=0}^{s=r}
    +\int_0^r
       \alpha'(s)\exp\la \int_s^r a(u)\,du\ra ds\\
    =&
    -\alpha(r)
    +f(0)\exp\la \int_0^r a(u)\,du\ra
    +\int_0^r
       b(s)\exp\la \int_s^r a(u)\,du\ra ds,
\end{align*}
where we used $\alpha(0)=f(0)$ and $\alpha'=b$ almost everywhere.
Adding $\alpha(r)$ to both sides gives the asserted equality.

 $\hfill$ $\square$

\begin{lemma}
\label{ns:lem:Pi-exists}
Let $\bH :[0,\infty)\to\Sp_{++}^n$, $n\in\{2,3\}$, be continuous with
$\kappa\bI _n\preceq\bH \preceq\Lambda\bI _n$. The following statements hold
\begin{itemize}
    \item The matrix-valued ODE
    $$\bM '(r) =r(\bH +r^2\bI _n)^{-1} \cdot \bM (r), \qquad  \bM (0)=\bI _n.$$  
    has a unique solution $\bM  \in C^1\lb [0,\infty), \bbR ^{n \times n} \rb$.  
    \item $\bM (r)$ is invertible for all $r \geq 0$ and satisfies
$$\norm{\bM (r)}\leq\sqrt{(\kappa+r^2)/\kappa}.$$ 
    \item  The limit $\bPi =\lim_{r\to\infty}\bM (r)/r$ exists  and 
\begin{equation}
\label{ns:eq:Pi-tail}
    \norm{\bPi -\frac{\bM (R)}{R}}
    \leq\frac{\Lambda}{2R^2} \cdot \norm{\frac{\bM (R)}{R}}.
\end{equation}
\end{itemize}
\end{lemma}

\noindent \textbf{Proof of Lemma \ref{ns:lem:Pi-exists}.}
It is easy to see that $\bA (r):=r(\bH (r)+r^2\bI _n)^{-1}$ is continuous and symmetric with
\begin{equation}
\label{ns:eq:A-sandwich}
    \frac{r}{\Lambda+r^2}\bI _n\preceq\bA (r)
    \preceq\frac{r}{\kappa+r^2}\bI _n .
\end{equation}
A linear system with continuous coefficients has a unique global solution (Theorem 5.1, chapter I of \cite{CoddingtonLevinson1955}). Thus the first claim follows.

We now prove the second claim. Fix $ \bm{v} \in\bbR ^n$ and put
\[
    f_{ \bm{v} }(r):=\norm{\bM (r) \bm{v} }^2.
\]
Taking derivative in $r$, we obtain
\[
    f_{ \bm{v} }'(r)
    =2\langle\bM (r) \bm{v} ,
              \bA (r)\bM (r) \bm{v} \rangle.
\]
Consequently, \eqref{ns:eq:A-sandwich} gives
\[
    \frac{2r}{\Lambda+r^2}f_{ \bm{v} }(r)
    \leq f_{ \bm{v} }'(r)
    \leq \frac{2r}{\kappa+r^2}f_{ \bm{v} }(r).
\]
Equivalently,
\[
    \frac{\dd}{\dd r}
    \left[\frac{f_{ \bm{v} }(r)}{\Lambda+r^2}\right]\geq0,
    \qquad
    \frac{\dd}{\dd r}
    \left[\frac{f_{ \bm{v} }(r)}{\kappa+r^2}\right]\leq0.
\]
Integrating from $0$ to $r$, and using
$\bM (0)=\bI _n$, yields
\[
    \frac{\Lambda+r^2}{\Lambda}\norm{ \bm{v} }^2
    \leq \norm{\bM (r) \bm{v} }^2
    \leq \frac{\kappa+r^2}{\kappa}\norm{ \bm{v} }^2.
\]
Thus $\bM (r)$ is invertible for every $r\geq0$, and taking the supremum over $\norm{ \bm{v} }=1$ in the upper bound gives
\[
    \norm{\bM (r)}
    \leq\sqrt{\frac{\kappa+r^2}{\kappa}},
\]
as claimed.

To prove the third claim, take $r>0$ and define
\[
    \bL (r):=\frac{\bM (r)}{r}.
\]
Differentiating gives
\[
    \bL '(r)
    =\frac{\bM '(r)}{r}-\frac{\bM (r)}{r^2}
    =\bC (r)\bL (r),
\]
where
\begin{align*}
    \bC (r) :=\bA (r)-\frac1r\bI _n
    &=r\lb \bH (r)+r^2\bI _n\rb ^{-1}
       -\frac1r\bI _n\\
    &=-\frac1r\bH (r)
       \lb \bH (r)+r^2\bI _n\rb ^{-1}.
\end{align*}
Therefore, for each fixed $r>0$, $\bC (r)$ is symmetric and has eigenvalues
\[
    -\frac{h}{r(h+r^2)},
    \qquad h\in\operatorname{spec}(\bH (r))
                  \subset[\kappa,\Lambda].
\]
In particular, $\bC (r)$ is negative definite. Moreover, since
$h\mapsto h/(h+r^2)$ is increasing on $(0,\infty)$,
\begin{equation}
\label{ns:eq:C-tail}
    \norm{\bC (r)}
      \leq \frac{\Lambda}{r(\Lambda+r^2)},
    \qquad
    \int_R^\infty\norm{\bC (r)}\dd r
    \leq  \int_{R}^\infty \frac{\Lambda}{r\lb \Lambda +r^2 \rb} \, dr  \leq \frac12\log\lb 1+\frac{\Lambda}{R^2}\rb
    \leq\frac{\Lambda}{2R^2},
\end{equation}
for all $R>0$.

Next, for every fixed $ \bm{v} \in\bbR ^n$,
\[
    \frac{\dd}{\dd r}\norm{\bL (r) \bm{v} }^2
    =2\langle\bL (r) \bm{v} ,
              \bC (r)\bL (r) \bm{v} \rangle
    \leq0.
\]
Thus, for $r\geq R>0$,
\[
    \norm{\bL (r) \bm{v} }\leq\norm{\bL (R) \bm{v} }.
\]
Taking the supremum over all unit vectors gives
\[
    \norm{\bL (r)}\leq\norm{\bL (R)},
    \qquad r\geq R.
\]
Fix $R>0$. For $s\geq t\geq R$, the fundamental theorem of calculus
and the preceding estimates imply
\begin{align*}
    \norm{\bL (s)-\bL (t)} \leq\int_t^s\norm{\bL '(r)}\dd r
    &\leq\int_t^s\norm{\bC (r)}\cdot \norm{\bL (r)}\dd r\\
    &\leq\norm{\bL (R)} \cdot 
          \int_t^\infty\norm{\bC (r)}\dd r
    \leq\frac{\Lambda}{2t^2} \cdot \norm{\bL (R)}.
\end{align*}
The last expression tends to zero as $t\to\infty$, uniformly over
$s\geq t \geq R$, with $R$ fixed. Hence $\bL (r)$ is Cauchy as $r\to\infty$, so there
exists $\bPi \in\bbR ^{n\times n}$ such that
\[
    \bL (r)\longrightarrow\bPi .
\]
Taking $t=R$ and letting $s\to\infty$ in the same estimate yields
\[
    \norm{\bPi -\frac{\bM (R)}{R}}
    \leq\frac{\Lambda}{2R^2}
         \norm{\frac{\bM (R)}{R}},
\]
which proves \eqref{ns:eq:Pi-tail}. The proof is completed. $\hfill$ $\square$

\begin{lemma} \label{ns:lem:zero-temp}
    Recall $ \bm{y} (r, \bm{h} )$ from \eqref{ns:y(r,h)}. We have the estimates 
\begin{equation}
\label{ns:eq:y-lip}
    \norm{ \bm{y} (r, \bm{h} _1)- \bm{y} (r, \bm{h} _2)}
      \leq\frac{r}{\kappa+r^2}\norm{ \bm{h} _1- \bm{h} _2}
    \leq\frac{\norm{ \bm{h} _1- \bm{h} _2}}{2\sqrt\kappa},
    \qquad
    \norm{ \bm{y} (r, \bm{h} )}
      \leq\frac{r\norm{ \bm{h} }+\norm{\nabla W( \bm{0} )}}{\kappa+r^2}.
\end{equation}
\end{lemma}

\noindent \textbf{Proof of Lemma \ref{ns:lem:zero-temp}.}
By \textnormal{\eqref{cond:W-curvature}}, for every $ \bm{u} , \bm{v} \in\bbR ^3$,
\begin{align*}
    \langle\nabla W( \bm{u} )-\nabla W( \bm{v} ),
             \bm{u} - \bm{v} \rangle
    &\geq \kappa\norm{ \bm{u} - \bm{v} }^2.
\end{align*}
Fix $r\geq0$ and $ \bm{h} _1, \bm{h} _2\in\bbR ^3$, and write
$ \bm{y} _i:= \bm{y} (r, \bm{h} _i)$ for $i=1,2$.
The first-order optimality conditions give
\[
    \nabla W( \bm{y} _i)+r^2 \bm{y} _i=r \bm{h} _i,
    \qquad i=1,2.
\]
Subtracting these equations and taking the inner product with
$ \bm{y} _1- \bm{y} _2$, we obtain
\begin{align*}
    (\kappa+r^2)\norm{ \bm{y} _1- \bm{y} _2}^2
    &\leq
    \langle\nabla W( \bm{y} _1)-\nabla W( \bm{y} _2),
             \bm{y} _1- \bm{y} _2\rangle
    +r^2\norm{ \bm{y} _1- \bm{y} _2}^2\\
    &=r\langle \bm{h} _1- \bm{h} _2, \bm{y} _1- \bm{y} _2\rangle\\
    &\leq r\norm{ \bm{h} _1- \bm{h} _2}\,
             \norm{ \bm{y} _1- \bm{y} _2}.
\end{align*}
If $ \bm{y} _1= \bm{y} _2$, the desired bound is immediate.
Otherwise, dividing by
$(\kappa+r^2)\norm{ \bm{y} _1- \bm{y} _2}$ gives
\[
    \norm{ \bm{y} (r, \bm{h} _1)- \bm{y} (r, \bm{h} _2)}
    \leq \frac{r}{\kappa+r^2}\norm{ \bm{h} _1- \bm{h} _2}.
\]
This yields the first estimate. The second estimate follows from
\[
    \frac{r}{\kappa+r^2}\leq\frac1{2\sqrt{\kappa}}.
\]
For the remaining estimate, fix $ \bm{h} \in\bbR ^3$ and write
$ \bm{y} := \bm{y} (r, \bm{h} )$. Its optimality condition can be rewritten as
\[
    \nabla W( \bm{y} )-\nabla W( \bm{0} )+r^2 \bm{y} 
    =r \bm{h} -\nabla W( \bm{0} ).
\]
Taking the inner product with $ \bm{y} $ and applying the
strong-convexity inequality above with $ \bm{u} = \bm{y} $ and
$ \bm{v} = \bm{0} $, we obtain
\begin{align*}
    (\kappa+r^2)\norm{ \bm{y} }^2
    &\leq
    \langle\nabla W( \bm{y} )-\nabla W( \bm{0} ), \bm{y} \rangle
    +r^2\norm{ \bm{y} }^2\\
    &=\langle r \bm{h} -\nabla W( \bm{0} ), \bm{y} \rangle
    \leq
    \lb r\norm{ \bm{h} }+\norm{\nabla W( \bm{0} )}\rb  \cdot \norm{ \bm{y} }.
\end{align*}
Dividing both sides by $(\kappa+r^2)\norm{ \bm{y} }$ completes the proof. \hfill$\square$

\subsection{Other technical results}

\begin{lemma} \label{ns:lem:y(r,h)-smooth}
    Recall $ \bm{y} (r, \bm{h} )$ from \eqref{ns:y(r,h)}. As a function of $(r, \bm{h} )$, we have $ \bm{y}  \in C^1 \lb (0,\infty) \times \bbR ^3 , \bbR ^3\rb$.
\end{lemma}

\noindent \textbf{Proof of Lemma \ref{ns:lem:y(r,h)-smooth}.}
For \(r>0\) and \( \bm{h} \in\bbR^3\), define

\[
\Phi_{r, \bm{h} }( \bm{u} )
:=
W( \bm{u} )+\frac12\norm{\bm{h} -r \bm{u}}^2 .
\]
By \textnormal{\eqref{cond:W-curvature}}, \(\Phi_{r, \bm{h} }\) is strongly convex and coercive. Consequently, it has a unique minimizer, which is precisely \( \bm{y} (r, \bm{h} )\). The first-order optimality condition is thus
\begin{equation}
\label{ns:eq:y-mode-equation}
\nabla W\lb  \bm{y} (r, \bm{h} ) \rb
+r^2 \bm{y} (r, \bm{h} )-r \bm{h} 
= \bm{0} .
\end{equation}
Consider the map
\[
 \bm{F} :\bbR^3\times(0,\infty)\times\bbR^3
\longrightarrow\bbR^3,
\qquad
 \bm{F} ( \bm{u} ,r, \bm{h} )
=
\nabla W( \bm{u} )+r^2 \bm{u} -r \bm{h} .
\]
Since \(W\in C^3(\bbR^3)\), the map \( \bm{F} \) is \(C^2\). A direct computation yields 
\[
D_{ \bm{u} } \bm{F} ( \bm{u} ,r, \bm{h} )
=
\nabla^2W( \bm{u} )+r^2\bm{I} _3.
\]
By \textnormal{\eqref{cond:W-curvature}}, 
$
D_{ \bm{u} } \bm{F} ( \bm{u} ,r, \bm{h} )
\succeq
(\kappa+r^2)\bm{I} _3,
$
and hence \(D_{ \bm{u} } \bm{F} ( \bm{u} ,r, \bm{h} )\) is invertible for every
\(( \bm{u} ,r, \bm{h} )\). In particular, it is invertible at
\( \bm{u} = \bm{y} (r, \bm{h} )\).
The implicit function theorem therefore implies that, in a neighborhood of every
\((r_0, \bm{h} _0)\in(0,\infty)\times\bbR^3\), the equation
\[
 \bm{F} ( \bm{u} ,r, \bm{h} )= \bm{0} 
\]
has a \(C^1\) solution \( \bm{u} =\widetilde{ \bm{y} }(r, \bm{h} )\). 

For later use, differentiating \eqref{ns:eq:y-mode-equation} with respect to
\( \bm{h} \) gives

\[
\left[
\nabla^2W\lb  \bm{y} (r, \bm{h} ) \rb+r^2\bm{I} _3
\right]D_{ \bm{h} } \bm{y} (r, \bm{h} )
=
r\bm{I} _3,
\]
and therefore
\begin{equation}
\label{ns:eq:Dh-y}
D_{ \bm{h} } \bm{y} (r, \bm{h} ) =
r\left[
\nabla^2W\lb  \bm{y} (r, \bm{h} ) \rb+r^2\bm{I} _3
\right]^{-1}.
\end{equation}
The proof is completed.
\(\hfill\square\)

\begin{lemma}[Posterior covariance estimates]
\label{ns:lem:sandwich}
For all $\beta \geq1$, $r>0$, and $ \bm{h} \in\bbR ^3$,
\begin{equation}
\label{ns:eq:BN-sandwich}
    \frac{r}{\Lambda+r^2}\bm{I} _3
    \preceq D_{ \bm{h} }  \bm{y} _\beta(r, \bm{h} )
      =\beta r \cdot \Cov_{\pi_{\beta ,r, \bm{h} }}( \bm{Y} )
    \preceq\frac{r}{\kappa+r^2}\bm{I} _3,
\end{equation}
and
\[\bbE _{\pi_{\beta ,r, \bm{h} }}\norm{\bm{Y} - \bm{y} (r, \bm{h} )}^2
\leq \frac{3}{\beta(\kappa+r^2)}.\]
\end{lemma}

\noindent \textbf{Proof of Lemma \ref{ns:lem:sandwich}.}
Fix $\beta\geq1$, $r>0$, and $ \bm{h} \in\bbR ^3$, and put
\[
    \Phi( \bm{u} )
    :=W( \bm{u} )+\frac12\norm{\bm{h} -r \bm{u}}^2,
    \qquad
     \bm{y} _*:= \bm{y} (r, \bm{h} ).
\]
Throughout the proof, expectations are taken with respect to
$ \bm{Y} \sim\pi_{\beta,r, \bm{h} }$.
By \textnormal{\eqref{cond:W-curvature}},
\begin{equation}
\label{ns:eq:posterior-Hessian-bounds}
    (\kappa+r^2)\bm{I} _3
    \preceq\nabla^2\Phi( \bm{u} )
    =\nabla^2W( \bm{u} )+r^2\bm{I} _3
    \preceq(\Lambda+r^2)\bm{I} _3.
\end{equation}
Recall that $ \bm{y} _\beta(r, \bm{h} )$ is the conditional mean of the posterior distribution $\pi_{\beta,r, \bm{h} }$. A direct differentiation yields
\[
    D_{ \bm{h} } \bm{y} _\beta(r, \bm{h} )
    =\beta r\cdot \Cov( \bm{Y} ).
\]
Moreover, the Cramér--Rao and Brascamp--Lieb inequalities give, respectively,
\[
    \frac1\beta
    \lb \bbE [\nabla^2\Phi( \bm{Y} )] \rb^{-1}
    \preceq \Cov( \bm{Y} )
    \preceq
    \frac1\beta
    \bbE \left[\lb \nabla^2\Phi( \bm{Y} ) \rb^{-1}\right].
\]
Using \eqref{ns:eq:posterior-Hessian-bounds}, we therefore obtain
\[
    \frac1{\beta(\Lambda+r^2)}\bm{I} _3
    \preceq \Cov( \bm{Y} )
    \preceq
    \frac1{\beta(\kappa+r^2)}\bm{I} _3.
\]
Multiplying by $\beta r$ and using
$D_{ \bm{h} } \bm{y} _\beta=\beta r\,\Cov( \bm{Y} )$
proves \eqref{ns:eq:BN-sandwich}.

We now prove the second moment bound. For each $i\in\{1,2,3\}$, integration by parts gives
\[
    0
    =
    \int_{\bbR ^3}
    \partial_{u_i}
    \left[
       (u_i-( \bm{y} _*)_i)e^{-\beta\Phi( \bm{u} )}
    \right]\dd \bm{u} .
\]
Dividing by the normalizing constant and summing over $i$ yields
\[
    3
    =
    \beta\cdot \bbE \Big[
       \langle \bm{Y} - \bm{y} _*,\nabla\Phi( \bm{Y} )\rangle
    \Big].
\]
On the other hand, strong convexity and
$\nabla\Phi( \bm{y} _*)= \bm{0} $ imply
\begin{align*}
    \langle \bm{u} - \bm{y} _*,\nabla\Phi( \bm{u} )\rangle
    &=
    \langle \bm{u} - \bm{y} _*,
            \nabla\Phi( \bm{u} )-\nabla\Phi( \bm{y} _*)\rangle
            \geq
    (\kappa+r^2)\norm{\bm{u} - \bm{y} _*}^2.
\end{align*}
Therefore,
\[
    \beta(\kappa+r^2)\cdot 
    \bbE \norm{\bm{Y} - \bm{y} _*}^2
    \leq3.
\]
The proof is completed. \hfill$\square$

\begin{lemma} \label{ns:lem:Piota-identity}
    Let $ \bm{F} : [0,1] \times \bbR ^d \to \bbR ^d$ be a $C^2$ function and write
    \[
     \bm{F} _s( \bm{x} ):=  \bm{F} (s, \bm{x} ), \qquad  \bm{q} _s( \bm{x} ):= \partial_s  \bm{F} _s( \bm{x} ).
    \]
    For all $\varphi \in C^2_c \lb \bbR ^d \rb$, 
    we have 
    \[
    \frac{d}{ds} \Big[ \varphi \lb  \bm{F} _s( \bm{x} ) \rb \cdot \det D  \bm{F} _s( \bm{x} )  \Big] = \operatorname{div}_{ \bm{x} } \Big[
    \varphi \lb  \bm{F} _s( \bm{x} ) \rb \cdot \operatorname{cof} \lb D  \bm{F} _s( \bm{x} ) \rb^\top \cdot  \bm{q} _s( \bm{x} )
    \Big]
    \]
    where $\operatorname{cof}(\bm{A} )$ denotes the cofactor matrix of a matrix $\bm{A} $.
\end{lemma}
The identity above is  a variant of the well-known Piola identity (see \cite{kupferman2019piola} and the references therein).

\noindent \textbf{Proof of Lemma \ref{ns:lem:Piota-identity}.}
 We first show the identity for $F$'s that are smooth in both space and time. We compute
\begin{align*}
    \frac{d}{ds} \Big[ \varphi \lb  \bm{F} _s( \bm{x} ) \rb \cdot \det D  \bm{F} _s( \bm{x} )  \Big]
    =& \frac{d}{ds} \Big[  \varphi \lb  \bm{F} _s( \bm{x} ) \rb \Big] \cdot  \det D  \bm{F} _s( \bm{x} )
    +   \varphi \lb  \bm{F} _s( \bm{x} ) \rb \cdot \frac{d}{ds} \Big[   \det D  \bm{F} _s( \bm{x} )  \Big] \\
    =& \nabla \varphi \lb  \bm{F} _s( \bm{x} ) \rb^\top \cdot \partial_s  \bm{F} _s( \bm{x} ) \cdot \det D  \bm{F} _s( \bm{x} )
    +    \varphi \lb  \bm{F} _s( \bm{x} ) \rb \cdot \langle  \operatorname{cof} \lb D  \bm{F} _s( \bm{x} ) \rb, D  \bm{q} _s( \bm{x} ) \rangle
\end{align*}
where we have used the facts that
\[
\frac{d}{ds} \Big[  \varphi \lb  \bm{F} _s( \bm{x} ) \rb \Big] = \nabla \varphi \lb  \bm{F} _s( \bm{x} ) \rb^\top \cdot \partial_s  \bm{F} _s( \bm{x} )
\qquad \text{and} \qquad 
\frac{d}{ds} \Big[   \det D  \bm{F} _s( \bm{x} )  \Big] =  \langle  \operatorname{cof} \lb D  \bm{F} _s( \bm{x} ) \rb, D  \bm{q} _s( \bm{x} ) \rangle.
\]
To proceed, let $\bm{C} _s(\bm{x}):= \operatorname{cof} \lb D  \bm{F} _s( \bm{x} ) \rb$. We will simply write $\bm C_s$ in some calculations below. The Piola identity yields
\begin{align} \label{ns:eq:Piotta-identity}
\operatorname{Div}_{ \bm{x} } \lb  \bm{C} _s( \bm{x} ) \rb =  \bm{0} 
\end{align}
Here the operator $\operatorname{Div}(\cdot)$ acts on matrices and is defined as 
\[
\operatorname{Div}_{ \bm{x} }\bm{C} _s
:=
\begin{pmatrix}
\operatorname{div}_{ \bm{x} }\lb \text{first row of }\bm{C} _s \rb\\
\vdots\\
\operatorname{div}_{ \bm{x} }\lb \text{$d$-th row of }\bm{C} _s \rb
\end{pmatrix}
=
\begin{pmatrix}
\displaystyle\sum_{j=1}^d \partial_{x_j}(\bm{C} _s)_{1j}\\
\vdots\\
\displaystyle\sum_{j=1}^d \partial_{x_j}(\bm{C} _s)_{dj}
\end{pmatrix}
\in\bbR^d.
\]
Thus
\begin{align*}
    \operatorname{div}_{ \bm{x} } \lb \bm{C} _s^\top  \bm{q} _s \rb &= \sum_{j=1}^d \partial_j \left[ \sum_{i=1}^d \lb \bm{C} _s \rb_{ij} \, \lb  \bm{q} _s \rb_i \right] \\
    &= \sum_{1\leq i,j \leq d} \partial_{x_j} \lb \bm{C} _s \rb_{ij} \, \lb  \bm{q} _s \rb_i + \sum_{1\leq i,j \leq d} \lb \bm{C} _s \rb_{ij} \cdot \partial_{x_j} \lb  \bm{q} _s \rb_i  \\
    &= \underbrace{\sum_{i=1}^d \lb \operatorname{Div} \bm{C} _s \rb_i \cdot \lb  \bm{q} _s \rb_i}_{= 0 \, \text{by} \;\eqref{ns:eq:Piotta-identity} }
    + \langle \bm{C} _s, D \bm{q} _s \rangle_{\rm HS} = \langle \bm{C} _s, D \bm{q} _s \rangle_{\rm HS} .
\end{align*}
Therefore,
\[
    \frac{d}{ds} \Big[ \varphi \lb  \bm{F} _s( \bm{x} ) \rb \cdot \det D  \bm{F} _s( \bm{x} )  \Big] 
    = \nabla \varphi \lb  \bm{F} _s( \bm{x} ) \rb^\top \cdot  \bm{q} _s \cdot \det D  \bm{F} _s( \bm{x} )
    +     \varphi \lb  \bm{F} _s( \bm{x} ) \rb \cdot     \operatorname{div}_{ \bm{x} } \lb \bm{C} _s^\top  \bm{q} _s \rb
\]
On the other hand, we have 
\begin{align*}
 \operatorname{div}_{ \bm{x} } \Big[
     \varphi \lb  \bm{F} _s( \bm{x} ) \rb \cdot \operatorname{cof} \lb D  \bm{F} _s( \bm{x} ) \rb^\top \cdot  \bm{q} _s( \bm{x} ) \Big] 
    &=  \nabla_{ \bm{x} } \Big[ \varphi \lb  \bm{F} _s( \bm{x} ) \rb \Big]^\top \cdot \bm{C} _s^\top  \bm{q} _s 
    +      \varphi \lb  \bm{F} _s( \bm{x} ) \rb \cdot     \operatorname{div}_{ \bm{x} } \lb \bm{C} _s^\top  \bm{q} _s \rb \\
    &= \nabla \varphi \lb  \bm{F} _s \rb^\top \cdot D  \bm{F} _s \cdot \bm{C} _s^\top  \bm{q} _s +       \varphi \lb  \bm{F} _s( \bm{x} ) \rb \cdot     \operatorname{div}_{ \bm{x} } \lb \bm{C} _s^\top  \bm{q} _s \rb
\end{align*}
Moreover, using the cofactor identity
\[
\bm{A}  \cdot \operatorname{cof} \lb \bm{A}  \rb^\top = \det \lb \bm{A}  \rb \cdot \bm{I} _d,
\]
with $\bm{A}  \equiv D  \bm{F} _s$, we arrive at
\[
\nabla \varphi \lb  \bm{F} _s \rb^\top \cdot D  \bm{F} _s \cdot \bm{C} _s^\top  \bm{q} _s = \nabla \varphi \lb  \bm{F} _s \rb^\top \cdot  \bm{q} _s \cdot \det D  \bm{F} _s( \bm{x} ).
\]
Thus the conclusion follows when $F$ is smooth in both space and time.  

For a general $C^1$ map, take smooth approximations $\bm F^\varepsilon\to\bm F$ in $C^1_{\mathrm{loc}}$. This can be done by taking the convolution 
\[
\bm F^\ve = \bm F * \rho_\ve.
\]
where $\rho_\ve : \mb R^{d+1} \to \mb R$ is a mollifier indexed by $\ve > 0$, and the convolution above is understood in the coordinate-wise sense (applying the convolution to each coordinate of $\bm F$, including the time coordinate). The readers are referred to Section 8.2 of \cite{folland1999real} for a comprehensive, compact exposition of mollifiers and their smoothing properties. 

Both $\varphi(\bm F^\varepsilon) \times \det D_{\bm x}\bm F^\varepsilon$ and $\varphi(\bm F^\varepsilon) \times 
\operatorname{cof}(D_{\bm x}\bm F^\varepsilon)^\top \times 
\partial_s\bm F^\varepsilon$
converge locally uniformly to their corresponding limits as $\ve \to 0$.$\hfill$ $\square$

\begin{lemma} \label{ns:lem:spectrum-bound}
    Let $\bm{A} , \bm{B}  \in \Sp ^d_{++}$ and put 
    \[
    \bm{C} _t := t \bm{A}  + (1-t) \bm{B} , \qquad t\in[0,1].
    \]
    It holds that
    \[
    t \lambda_{\rm min} \lb \bm{A}  \rb + (1-t)\lambda_{\rm min}(\bm{B} ) \leq \lambda_{\rm min} \lb \bm{C} _t \rb
    \leq \lambda_{\rm max} \lb \bm{C} _t \rb \leq   t \lambda_{\rm max} \lb \bm{A}  \rb + (1-t)\lambda_{\rm max}(\bm{B} ).
    \]
\end{lemma}

\noindent \textbf{Proof of Lemma \ref{ns:lem:spectrum-bound}.}
The proof follows directly from the facts that the maximum eigenvalue is a convex functional on $\Sp ^d_{++}$, while the minimum eigenvalue is a concave functional on the same space. $\hfill$ $\square$



\begin{lemma} \label{ns:lem:inverse-lipschitz}
Let $\kappa>0$ and let $\bm{U} ,\bm{V} \in\bbR ^{n\times n}$
be symmetric matrices satisfying
\[
    \bm{U} ,\bm{V} \succeq\kappa\bm{I} _n.
\]
Then, in the operator norm,
\[
    \norm{\bm{U} ^{-1}-\bm{V} ^{-1}}
    \leq\kappa^{-2}\norm{\bm{U} -\bm{V}}.
\]
\end{lemma}

\noindent \textbf{Proof of Lemma \ref{ns:lem:inverse-lipschitz}.}
The lower eigenvalue bounds imply that both matrices are
invertible, with
$\norm{\bm{U} ^{-1}},\norm{\bm{V} ^{-1}}\leq\kappa^{-1}$.
Using
\[
    \bm{U} ^{-1}-\bm{V} ^{-1}
    =\bm{U} ^{-1}(\bm{V} -\bm{U} )\bm{V} ^{-1},
\]
we obtain
\[
    \norm{\bm{U} ^{-1}-\bm{V} ^{-1}}
    \leq
    \norm{\bm{U} ^{-1}}\cdot \norm{\bm{V} -\bm{U}}\cdot \norm{\bm{V} ^{-1}}
    \leq\kappa^{-2}\norm{\bm{U} -\bm{V}}.
\]
The proof is completed. $\hfill$ $\square$

\section{Technical results used in the proof of Theorem~\ref{ns:thm:gmm}}

\begin{lemma}\label{lem:maximum-interpoint}
Recall the mean vectors $\bm\mu_j^{\pm}$ in
\eqref{eq:theta-mu-p}. Then
\[
4R^2
\leq
\max_{\substack{
1\leq i,j\leq N\\
s,t\in\{+,-\}\\
(i,s)\neq(j,t)
}}
\|\bm\mu_i^s-\bm\mu_j^t\|_2^2
\leq
4R^2+\eta^2.
\]
\end{lemma}

\noindent \textbf{Proof of Lemma \ref{lem:maximum-interpoint}.}
First note that the two means in the same pair satisfy
\[
\|\bm\mu_j^+-\bm\mu_j^-\|_2
=
2R\|\bm u_{\pi\theta_j^2}\|_2
=
2R.
\]
Thus the maximum interpoint distance is at least $2R$. On the other hand, for arbitrary indices $i,j$ and signs $s,t$, we have 
\
\[
\left\|  \bm \mu_i^{s} - \bm \mu_j^t \right\|^2  = R^2 \cdot \left\| \mbox{sign}(s) \cdot \bm u_{\pi \theta_i^2} -  \mbox{sign}(t) \cdot \bm u_{\pi \theta_j^2} \right\|^2
+\eta^2 \lb \theta_i - \theta_j \rb^2 \leq 4R^2 + \eta^2. 
\]
The proof is completed. $\hfill$ $\square$

The next result is a perturbation result for matrix-valued ODEs. Although it is standard, its non-asymptotic nature is helpful for the proof of Theorem~\ref{ns:thm:gmm}.

\begin{lemma} \label{lem:perturbation-matrix-ode}
    Let $\bm C, \widetilde{\bm C}:[0,1] \to \mb R^{d\times d}$ be continuous matrix-valued mappings and suppose $\bm M, \widetilde{\bm M}$ solve the ODEs
    \[
    \frac{d}{dt} \bm M(t) =  \bm C(t) \cdot  \bm M(t), \qquad \frac{d}{dt} \widetilde{\bm M}(t)= \widetilde{\bm C}(t) \cdot \widetilde{\bm M}(t), \qquad \bm M(0)= \widetilde{\bm M}(0)= \bm I_d.
    \]
    Then, it holds that
    \[
    \sup_{t \in [0,1]}\left\| \bm M(t) - \widetilde{\bm M}(t)  \right\| \leq \exp \lb \int_0^1 \left\|  \bm C(t) \right\| \, dt \rb \cdot \la \exp \left[ \int_0^1 \left\| \widetilde{\bm C}(s) - \bm C(s) \right\| \, ds  \right]  -1 \ra .
    \]
\end{lemma}

\noindent \textbf{Proof of Lemma \ref{lem:perturbation-matrix-ode}.}  Set
\[
\bm\Delta(t):=\widetilde{\bm C}(t)-\bm C(t),
\qquad
A(t):=\int_0^t\|\bm C(s)\|\,ds,
\qquad
D(t):=\int_0^t\|\bm\Delta(s)\|\,ds.
\]
By Gr\"onwall's inequality,
\[
\left\|\widetilde{\bm M}(t)\right\|
\le
\exp\left(\int_0^t\|\widetilde{\bm C}(s)\|\,ds\right)
\le e^{A(t)+D(t)}.
\]
Let $\bm E(t):=\widetilde{\bm M}(t)-\bm M(t)$.
Subtracting the two ODEs gives
\[
\bm E'(t)
=
\bm C(t)\bm E(t)+\bm\Delta(t)\widetilde{\bm M}(t),
\qquad
\bm E(0)=\bm0_{d \times d}.
\]
Integrating the above and taking norms on both sides, we obtain
\[
\|\bm E(t)\|
\le
\int_0^t
\left[
\|\bm C(s)\| \cdot \|\bm E(s)\|
+
\|\bm\Delta(s)\| \cdot e^{A(s)+D(s)}
\right]\,ds
=:F(t).
\]
Since $\|\bm E(t)\|\le F(t)$, the function $F$ satisfies
\[
F'(t)
\le A'(t)F(t)+D'(t)e^{A(t)+D(t)},
\qquad F(0)=0.
\]
Multiplying by the integrating factor $e^{-A(t)}$ yields
\[
\frac{d}{dt}\left[e^{-A(t)}F(t)\right]
\le D'(t)e^{D(t)}
=\frac{d}{dt}e^{D(t)}.
\]
Integrating from $0$ to $t$ therefore gives
\[
\left\|\bm E(t)\right\|
\le F(t)
\le e^{A(t)}\left(e^{D(t)}-1\right).
\]
Finally, $A$ and $D$ are nondecreasing, so
\[
\sup_{t\in[0,1]}
\left\|\widetilde{\bm M}(t)-\bm M(t)\right\|
\le
\exp\left(\int_0^1\|\bm C(s)\|\,ds\right) \cdot 
\left[
\exp\left(\int_0^1
\|\widetilde{\bm C}(s)-\bm C(s)\|\,ds\right)-1
\right].
\]
This completes the proof. \hfill $\square$

\begin{lemma} \label{lem:properness}
    Recall $\sfZ_1^{p_1} \lb \bm x \rb$ in \eqref{eq:sfZ_1}. It holds that
    \[
    \sup_{\bm x \in \mb R^3} \left\| \sfZ_1^{p_1} (\bm x) - \bm x \right\| < \infty. 
    \]
\end{lemma}

\noindent \textbf{Proof of Lemma \ref{lem:properness}.}   Recall the rescaled flow $\bm q \lb a(t),\bm x \rb$ in \eqref{eq:rescaled-q}. Integrating \eqref{eq:ODE-q-gmm} gives
\[
\sfZ_1^{p_1} \lb \bm x \rb - \bm x = \bm q(1,\bm x) - \bm x = \int_0^1 \bm m \lb a, \bm q(a, \bm x)  \rb \, da. 
\]
On the other hand, 
\begin{align*}
\left\|\bm m(a,\bm q)\right\|
&\leq \sum_{j=1}^N\sum_{s\in\{+,-\}}
w_j^s(a,\bm q)\left\|\bm\mu_j^s\right\|\\
&\leq \max_{\substack{1\leq j\leq N\\s\in\{+,-\}}}
\left\|\bm\mu_j^s\right\|\\
&= \max_{1\leq j\leq N}\sqrt{R^2+\eta^2\theta_j^2}\\
&\leq \sqrt{R^2+\eta^2}
= \sqrt{4\pi+\frac{3\ve}{8}}.
\end{align*}
Thus
\[
\sup_{\bm x \in \mb R^3} \left\| \sfZ_1^{p_1} (\bm x) - \bm x \right\|  \leq 4\pi + \ve.
\]
The proof is completed. $\hfill$ $\square$

\begin{lemma} \label{lem:block-convergence-gmm}
    With the notation as in the proof of Theorem~\ref{ns:thm:gmm}, we have 
    \[
    \sup_{a \in [0,1]} \left\| \sum_{j=1}^N \alpha_j(a) \cdot \bm P_{\pi \theta_j^2} - \bm P_{\pi a^2} \right\|
    \leq \min \la 1; \frac{\pi}{N} + \frac{\pi\sqrt{2}}{\sqrt{B}} + \frac{\pi\eta^2}{B} \ra.
    \]
\end{lemma}

\noindent \textbf{Proof of Lemma \ref{lem:block-convergence-gmm}.} 
Let $\Theta_a$ be a random variable taking values in $\la \theta_1,\dots,\theta_N \ra$ with the law
\[
\mb P \lb \Theta_a = \theta_i \rb = \alpha_i(a) =  \frac{ \exp \la -\beta(a) \cdot \left[ \theta_i - c(a) \right]^2 \ra}
{\sum_{k=1}^N \exp \la -\beta(a) \cdot \left[ \theta_k - c(a) \right]^2 \ra}
\]
where
\[
\beta(a):= B +\frac{a^2\eta^2}{2}, \qquad c(a):= \frac{a\eta \cdot Q(a)}{2B + a^2 \eta^2}.
\]
Here $Q$ is the solution of the ODE in \eqref{eq:Q-ode-gmm}. Note that the representations of the $\alpha_i(a)$ follow directly from completing the squares in their original definitions; see the proof of Theorem~\ref{ns:thm:gmm}. 

We first collect some elementary properties. It is clear from definitions that for all $a \in [0,1]$,
\[
\beta(a) \geq B, \quad c(a) \in [0,1], \qquad \left| c(a)-a \right| \leq \frac{\eta^2}{2B}. 
\]
The last bound follows from the fact that $L \leq Q(a) \leq L+a\eta$, where we recall that $L=2B/\eta$. Next, observe that
\begin{align*}
     \sup_{a \in [0,1]} \left\| \sum_{j=1}^N \alpha_j(a) \cdot \bm P_{\pi \theta_j^2} - \bm P_{\pi a^2} \right\|
     = \sup_{a \in [0,1]} \left\| \mb E \left[ \bm P_{\pi \Theta_a^2} - \bm P_{\pi a^2} \right] \right\|
     &\leq \sup_{a \in [0,1]}  \mb E \left\| \bm P_{\pi \Theta_a^2} - \bm P_{\pi a^2}  \right\| 
\end{align*}
A simple calculation yields 
\[
 \left\| \bm P_{\pi \Theta_a^2} - \bm P_{\pi a^2}  \right\| = \left|\sin \lb \pi \Theta_a^2 - \pi a^2 \rb \right| \leq 2\pi \cdot |\Theta_a-a|
\]
where the last bound follows from the fact that $|\sin(x)| \leq |x|$. 

Thus
\begin{align*}
         \sup_{a \in [0,1]} \left\| \sum_{j=1}^N \alpha_j(a) \cdot \bm P_{\pi \theta_j^2} - \bm P_{\pi a^2} \right\|
         \leq 2\pi \cdot \sup_{a \in [0,1]} \mb E \left| \Theta_a -a \right|
         &\leq 2\pi \cdot \sup_{a \in [0,1]} \Big\{  \mb E \left| \Theta_a -c(a) \right| + |c(a)-a| \Big\} \\
         &\leq 2\pi \cdot \sup_{a \in [0,1]} \mb E  \left| \Theta_a -c(a) \right| + \frac{\pi \eta^2}{B}.
\end{align*}
To finish the proof, it suffices to bound the first term in the last display. We will show that
\begin{align} \label{eq:mean-Theta_a}
     \mb E  \left| \Theta_a -c(a) \right| \leq \frac{1}{2N} + \frac{1}{\sqrt{2 B}}.
\end{align}
To prove \eqref{eq:mean-Theta_a}, we partition $[0,1]$ into $N$ intervals $I_j:=[(j-1)/N,j/N]$. Note that the partition yields $\theta_j \in I_j$. Put 
\[
H(x,a):= \int_{-1/(2N)}^{1/(2N)} e^{-\beta(a) u^2} \cdot \cosh \lb 2\beta(a) \cdot xu  \rb \, du. 
\]
Write
\begin{align}
& \int_{I_j} \exp\la -\beta(a) \cdot \left[ y-c(a) \right]^2 \ra \, dy \nonumber \\
=& \int_{(j-1)/N}^{j/N} \exp \la -\beta(a) \cdot \lb y - \theta_j \rb^2 \ra \cdot \exp \la -2\beta(a) \cdot \lb y- \theta_j \rb \lb \theta_j-c(a) \rb \ra, dy \nonumber \\
\times& \exp \la -\beta(a) \cdot \left[ \theta_j-c(a) \right]^2 \ra \nonumber \\ 
=&  \exp \la -\beta(a) \cdot \left[ \theta_j-c(a) \right]^2 \ra \times 
\int_{-\frac{1}{2N}}^{\frac{1}{2N}} \exp \la -\beta(a) \cdot u^2 \ra \cdot 
\cosh \left[ 2\beta(a) \cdot u \lb \theta_j-c(a) \rb \right] \, du
\nonumber \\
=& \exp \la -\beta(a) \cdot \left[ \theta_j - c(a) \right]^2 \ra \cdot H \lb |\theta_j-c(a)|,a \rb. \label{eq:change-of-variable}
\end{align}
Note that the second equality follows from the facts that $\exp(-x)= \cosh(x)-\sinh(x)$, $\sinh(x)$ is an odd function, and a change of variable. 

Since $H(x,a)$ is non-decreasing on $[0,\infty)$ for each $a \in [0,1]$, Chebyshev's association inequality (Theorem 2.14 from \citep{boucheron2013concentration}) gives
\[
0 \leq \mbox{Cov} \lb \left| \Theta_a - c(a) \right|,\,H \lb |\Theta_a-c(a)|,a \rb  \rb, \qquad \forall a \in [0,1].
\]
Equivalently, for all $a \in [0,1]$,
\begin{align*}
    \mb E \lb |\Theta_a-c(a)| \rb &\leq \frac{\mb E \Big[ |\Theta_a-c(a)| \cdot H \lb |\Theta_a-c(a)|,a \rb \Big]}{\mb E H \lb |\Theta_a-c(a)|, a \rb} \\
    &\stackrel{\rm (i)}{=} \frac{\sum_{j=1}^N |\theta_j-c(a)| \cdot \intop_{I_j} \exp \la -\beta(a) \cdot \left[ y-c(a) \right]^2 \ra \, dy  }{\intop_0^1 \exp \la -\beta(a) \cdot \left[ y-c(a) \right]^2 \ra \, dy } \\
    & \stackrel{\rm (ii)}{\leq} \frac{1}{2N} + 
    \frac{\intop_0^1 |y-c(a)| \cdot \exp \la -\beta(a) \cdot \left[ y-c(a) \right]^2 \ra \, dy}{\intop_0^1 \exp \la -\beta(a) \cdot \left[ y-c(a) \right]^2 \ra \, dy}.
\end{align*}
Here (i) follows from \eqref{eq:change-of-variable} and (ii)  follows from the triangle inequality
\[
|\theta_j -c(a)| \leq |y-c(a)| + |\theta_j-y| \leq |y-c(a)| + \frac{1}{2N}
\]
for all $y \in I_j$. 

To finish the proof, we bound the second term in (ii). The Cauchy--Schwarz inequality yields
\[
\frac{\intop_0^1 |y-c(a)| \cdot \exp \la -\beta(a) \cdot \left[ y-c(a) \right]^2 \ra \, dy}{\intop_0^1 \exp \la -\beta(a) \cdot \left[ y-c(a) \right]^2 \ra \, dy}
\leq 
\sqrt{\frac{ \intop_0^1 \lb y-c(a) \rb^2 \cdot \exp \la -\beta(a) \cdot \left[ y-c(a) \right]^2 \ra \, dy}{\intop_0^1 \exp \la -\beta(a) \cdot \left[ y-c(a) \right]^2 \ra \, dy}.}
\]
To bound the last term, we do integration by parts to get
\begin{align*}
\intop_0^1 \lb y-c(a) \rb^2 \cdot \exp \la -\beta(a) \cdot \left[ y-c(a) \right]^2 \ra \, dy 
&= -\frac{1}{2\beta(a)} \la (y-c(a)) \cdot \exp \left[ -\beta(a) \cdot (y-c(a))^2 \right] \ra \Big|_{y=0}^{y=1} \\
&+ \frac{1}{2\beta(a)} \cdot \int_0^1 \exp \la -\beta(a) \cdot \left[ y-c(a) \right]^2 \ra \, dy \\
&\leq \frac{1}{2\beta(a)} \cdot \int_0^1 \exp \la -\beta(a) \cdot \left[ y-c(a) \right]^2 \ra \, dy.
\end{align*}
Consequently,
\[
\frac{\intop_0^1 |y-c(a)| \cdot \exp \la -\beta(a) \cdot \left[ y-c(a) \right]^2 \ra \, dy}{\intop_0^1 \exp \la -\beta(a) \cdot \left[ y-c(a) \right]^2 \ra \, dy}
\leq \frac{1}{\sqrt{2 \beta(a)}} \leq \frac{1}{\sqrt{2B}}.
\]
This yields \eqref{eq:mean-Theta_a} and in turn completes the proof. $\hfill$ $\square$

%% file: paper-ref-url.bib
@article{liu2022rectified,
  title={Rectified flow: A marginal preserving approach to optimal transport},
  author={Liu, Qiang},
  journal={arXiv preprint arXiv:2209.14577},
  year={2022},
  url={https://arxiv.org/abs/2209.14577}
}

@inproceedings{kingma2014auto,
  title     = {Auto-Encoding Variational {Bayes}},
  author    = {Kingma, Diederik P. and Welling, Max},
  booktitle = {International Conference on Learning Representations},
  year      = {2014},
  url       = {https://arxiv.org/abs/1312.6114}
}

@inproceedings{goodfellow2014generative,
  title     = {Generative Adversarial Nets},
  author    = {Goodfellow, Ian J. and Pouget-Abadie, Jean and
               Mirza, Mehdi and Xu, Bing and Warde-Farley, David and
               Ozair, Sherjil and Courville, Aaron and Bengio, Yoshua},
  booktitle = {Advances in Neural Information Processing Systems},
  volume    = {27},
  pages     = {2672--2680},
  year      = {2014},
  url       = {https://arxiv.org/abs/1406.2661}
}

@inproceedings{sohl2015deep,
  title={Deep unsupervised learning using nonequilibrium thermodynamics},
  author={Sohl-Dickstein, Jascha and Weiss, Eric and Maheswaranathan, Niru and Ganguli, Surya},
  booktitle={International conference on machine learning},
  pages={2256--2265},
  year={2015},
  organization={PMLR},
  url={https://arxiv.org/abs/1503.03585}
}

@article{ho2020denoising,
  title={Denoising diffusion probabilistic models},
  author={Ho, Jonathan and Jain, Ajay and Abbeel, Pieter},
  journal={Advances in neural information processing systems},
  volume={33},
  pages={6840--6851},
  year={2020},
  url={https://arxiv.org/abs/2006.11239}
}

@inproceedings{song2021denoising,
title={Denoising Diffusion Implicit Models},
author={Jiaming Song and Chenlin Meng and Stefano Ermon},
booktitle={International Conference on Learning Representations},
year={2021},
url={https://openreview.net/forum?id=St1giarCHLP}
}

@article{lipman2022flow,
  title={Flow matching for generative modeling},
  author={Lipman, Yaron and Chen, Ricky TQ and Ben-Hamu, Heli and Nickel, Maximilian and Le, Matt},
  journal={arXiv preprint arXiv:2210.02747},
  year={2022},
  url={https://arxiv.org/abs/2210.02747}
}

@inproceedings{liu2023instaflow,
  title={Instaflow: One step is enough for high-quality diffusion-based text-to-image generation},
  author={Liu, Xingchao and Zhang, Xiwen and Ma, Jianzhu and Peng, Jian and others},
  booktitle={The Twelfth International Conference on Learning Representations},
  year={2024},
  url={https://arxiv.org/abs/2309.06380}
}

@inproceedings{yan2024perflow,
  title={PeRFlow: Piecewise Rectified Flow as Universal Plug-and-Play Accelerator},
  author={Yan, Hanshu and Liu, Xingchao and Pan, Jiachun and Liew, Jun Hao and Liu, Qiang and Feng, Jiashi},
  booktitle={Advances in Neural Information Processing Systems},
  volume={37},
  year={2024},
  url={https://papers.nips.cc/paper_files/paper/2024/hash/8f9d1362a386e80a723e98451a8c7564-Abstract-Conference.html}
}

@article{audio_RF_wang2024frieren,
  title={Frieren: Efficient video-to-audio generation network with rectified flow matching},
  author={Wang, Yongqi and Guo, Wenxiang and Huang, Rongjie and Huang, Jiawei and Wang, Zehan and You, Fuming and Li, Ruiqi and Zhao, Zhou},
  journal={Advances in Neural Information Processing Systems},
  volume={37},
  pages={128118--128138},
  year={2024},
  url={https://papers.nips.cc/paper_files/paper/2024/hash/743db1fd01fe1c4795f4b1aa11fc9c2d-Abstract-Conference.html}
}

@inproceedings{audio_RF_liu2025flashaudio,
  title={FlashAudio: Rectified Flow for Fast and High-Fidelity Text-to-Audio Generation},
  author={Liu, Huadai and Wang, Jialei and Huang, Rongjie and Liu, Yang and Lu, Heng and Zhao, Zhou and Xue, Wei},
  booktitle={Proceedings of the 63rd Annual Meeting of the Association for Computational Linguistics (Volume 1: Long Papers)},
  pages={13694--13710},
  year={2025},
  url={https://aclanthology.org/2025.acl-long.666/}
}

@inproceedings{video_RF_chen2025goku,
  title={Goku: Flow based video generative foundation models},
  author={Chen, Shoufa and Ge, Chongjian and Zhang, Yuqi and Zhang, Yida and Zhu, Fengda and Yang, Hao and Hao, Hongxiang and Wu, Hui and Lai, Zhichao and Hu, Yifei and others},
  booktitle={Proceedings of the Computer Vision and Pattern Recognition Conference},
  pages={23516--23527},
  year={2025},
  url={https://openaccess.thecvf.com/content/CVPR2025/html/Chen_Goku_Flow_Based_Video_Generative_Foundation_Models_CVPR_2025_paper.html}
}

@article{video_RF_liu2025improving,
  title={Improving video generation with human feedback},
  author={Liu, Jie and Liu, Gongye and Liang, Jiajun and Yuan, Ziyang and Liu, Xiaokun and Zheng, Mingwu and Wu, Xiele and Wang, Qiulin and Xia, Menghan and Wang, Xintao and others},
  journal={arXiv preprint arXiv:2501.13918},
  year={2025},
  url={https://arxiv.org/abs/2501.13918}
}

@inproceedings{liu2023flow,
  title={Flow straight and fast: Learning to generate and transfer data with rectified flow},
  author={Liu, Xingchao and Gong, Chengyue and Liu, Qiang},
  booktitle={The Eleventh International Conference on Learning Representations},
  year={2023},
  url={https://arxiv.org/abs/2209.03003}
}

@InProceedings{bansal2025wasserstein,
  title = 	 { On the Convergence and Straightness of Rectified Flow },
  author = 	 {Bansal, Vansh and Roy, Saptarshi and Rinaldo, Alessandro and Sarkar, Purnamrita},
  booktitle = 	 {Proceedings of The 29th International Conference on Artificial Intelligence and Statistics},
  pages = 	 {4636--4644},
  year = 	 {2026},
  publisher = 	 {PMLR},
  url = {https://openreview.net/forum?id=qKKxzTucgy}
}

@article{mena2025statistical,
  author  = {Mena, Gonzalo and Kuchibhotla, Arun Kumar
             and Wasserman, Larry},
  title   = {Statistical Properties of Rectified Flow},
  journal = {arXiv preprint arXiv:2511.03193},
  year    = {2025},
  url     = {https://arxiv.org/abs/2511.03193}
}

@inproceedings{zhu2026modelcollapse,
  title     = {Analyzing and Mitigating Model Collapse in Reflow Methods},
  author    = {Zhu, Huminhao and Wang, Fangyikang and Ding, Tianyu
               and Qu, Qing and Zhu, Zhihui},
  booktitle = {Conference on Parsimony and Learning},
  series    = {Proceedings of Machine Learning Research},
  volume    = {328},
  pages     = {314--340},
  year      = {2026},
  publisher = {PMLR},
  url       = {https://proceedings.mlr.press/v328/zhu26a.html}
}

@inproceedings{kim2025balanced,
  title     = {Balanced Conic Rectified Flow},
  author    = {Kim, Shin Seong and Kwon, Mingi and Jeong, Jaeseok
               and Uh, Youngjung},
  booktitle = {Advances in Neural Information Processing Systems},
  volume    = {38},
  year      = {2025},
  url       = {https://openreview.net/forum?id=kqHzgTV9AU}
}

@inproceedings{roy2026optimal,
title={Optimal Self-Distillation for Rectified Flow via Linear Probing},
author={Saptarshi Roy and Debepsita Mukherjee and Pratik Patil},
booktitle={The First Conference on Statistics and Trustworthy AI for Cross (X)-Domain Acceleration},
year={2026},
url={https://arxiv.org/abs/2607.14947}
}

@article{plastock1974homeomorphisms,
  title={Homeomorphisms between Banach spaces},
  author={Plastock, Roy},
  journal={Transactions of the American mathematical society},
  volume={200},
  pages={169--183},
  year={1974},
  url={https://doi.org/10.2307/1996934}
}

@article{ruzhansky2015global,
  title={On global inversion of homogeneous maps},
  author={Ruzhansky, Michael and Sugimoto, Mitsuru},
  journal={Bulletin of Mathematical Sciences},
  volume={5},
  number={1},
  pages={13--18},
  year={2015},
  publisher={Springer},
  url={https://doi.org/10.1007/s13373-014-0059-1}
}

@article{hertrich2026relation,
  title={On the relation between rectified flows and optimal transport},
  author={Hertrich, Johannes and Chambolle, Antonin and Delon, Julie},
  journal={Advances in Neural Information Processing Systems},
  volume={38},
  pages={172578--172610},
  year={2026},
  url={https://papers.nips.cc/paper_files/paper/2026/hash/2f6c89b3d2f1f1696b5b2dd55f0cb0e4-Abstract-Conference.html}
}

@article{nonstraight_tsimpos2026one,
  title={One-Shot Generative Flows: Existence and Obstructions},
  author={Tsimpos, Panos and Sharp, Daniel and Marzouk, Youssef},
  journal={arXiv preprint arXiv:2604.15439},
  year={2026},
  url={https://arxiv.org/abs/2604.15439}
}

@article{nonstraight_pierret2026tessellations,
  title={Tessellations of Semi-Discrete Flow Matching},
  author={Pierret, Emile and Hertrich, Johannes and Hurault, Samuel and Delon, Julie},
  journal={arXiv preprint arXiv:2605.07513},
  year={2026},
  url={https://arxiv.org/abs/2605.07513}
}

@article{nonstraight_chambolle2026limit,
  title={Limit Points of Reflow with Minibatch Optimal Transport},
  author={Chambolle, Antonin and Hertrich, Johannes},
  journal={arXiv preprint arXiv:2608.07042},
  year={2026},
  url={https://arxiv.org/abs/2608.07042}
}

@book{ambrosio2005gradient,
  title={Gradient flows: in metric spaces and in the space of probability measures},
  author={Ambrosio, Luigi and Gigli, Nicola and Savar{\'e}, Giuseppe},
  year={2005},
  publisher={Springer},
  url={https://doi.org/10.1007/b137080}
}

@article{popoviciu1935equations,
  title={Sur les {\'e}quations alg{\'e}briques ayant toutes leurs racines r{\'e}elles},
  author={Popoviciu, Tiberiu},
  journal={Mathematica},
  volume={9},
  number={129-145},
  pages={20},
  year={1935},
  url={https://www.ictp.acad.ro/popoviciu/articles/sur_les_equations_algebriques_ayant_toutes_leurs_racines_reelles.pdf}
}

@article{bhatia2012some,
  title={Some inequalities for positive linear maps},
  author={Bhatia, Rajendra and Sharma, Rajesh},
  journal={Linear algebra and its applications},
  volume={436},
  number={6},
  pages={1562--1571},
  year={2012},
  publisher={Elsevier},
  url={https://doi.org/10.1016/j.laa.2011.12.017}
}

@article{magnus1954exponential,
  title={On the exponential solution of differential equations for a linear operator},
  author={Magnus, Wilhelm},
  journal={Communications on pure and applied mathematics},
  volume={7},
  number={4},
  pages={649--673},
  year={1954},
  publisher={Wiley Online Library},
  url={https://doi.org/10.1002/cpa.3160070404}
}

@article{efron2011tweedie,
  title={Tweedie’s formula and selection bias},
  author={Efron, Bradley},
  journal={Journal of the American Statistical Association},
  volume={106},
  number={496},
  pages={1602--1614},
  year={2011},
  publisher={Taylor \& Francis},
  url={https://doi.org/10.1080/01621459.2011.622401}
}

@article{meng2021estimating,
  title={Estimating high order gradients of the data distribution by denoising},
  author={Meng, Chenlin and Song, Yang and Li, Wenzhe and Ermon, Stefano},
  journal={Advances in Neural Information Processing Systems},
  volume={34},
  pages={25359--25369},
  year={2021},
  url={https://openreview.net/forum?id=YTkQQrqSyE1}
}

@book{teschl2012ordinary,
  title={Ordinary differential equations and dynamical systems},
  author={Teschl, Gerald},
  volume={140},
  year={2012},
  publisher={American Mathematical Soc.},
  url={https://bookstore.ams.org/gsm-140/}
}

@book{folland1999real,
  author    = {Folland, Gerald B.},
  title     = {Real Analysis: Modern Techniques and Their Applications},
  edition   = {2},
  publisher = {John Wiley \& Sons},
  year      = {1999},
  url       = {https://www.wiley.com/en-us/Real+Analysis%3A+Modern+Techniques+and+Their+Applications%2C+2nd+Edition-p-9780471317166}
}

@misc{simader2005homotopy,
  author = {Simader, Christian G.},
  title  = {A homotopy argument and its applications to the
            transformation rule for bi-{Lipschitz} mappings,
            the {Brouwer} fixed point theorem and the
            {Brouwer} degree},
  year   = {2005},
  note   = {preprint, University of Bayreuth},
  url    = {https://epub.uni-bayreuth.de/id/eprint/813/}
}

@inproceedings{pmlr-v48-liub16,
  title     = {A Kernelized {Stein} Discrepancy for Goodness-of-fit Tests},
  author    = {Liu, Qiang and Lee, Jason and Jordan, Michael},
  booktitle = {Proceedings of the 33rd International Conference on Machine Learning},
  volume    = {48},
  pages     = {276--284},
  year      = {2016},
  url       = {https://proceedings.mlr.press/v48/liub16.html}
}

@book{boucheron2013concentration,
  author    = {Boucheron, St{\'e}phane and Lugosi, G{\'a}bor
               and Massart, Pascal},
  title     = {Concentration Inequalities:
               A Nonasymptotic Theory of Independence},
  publisher = {Oxford University Press},
  year      = {2013},
  url       = {https://global.oup.com/academic/product/concentration-inequalities-9780199535255}
}

@book{CoddingtonLevinson1955,
  author    = {Coddington, Earl A. and Levinson, Norman},
  title     = {Theory of Ordinary Differential Equations},
  publisher = {McGraw-Hill},
  address   = {New York},
  year      = {1955},
  url       = {https://archive.org/details/theoryofordinaryd00codd}
}

@article{heiny2025maximum,
  title={Maximum interpoint distance of high-dimensional random vectors},
  author={Heiny, Johannes and Kleemann, Carolin},
  journal={Bernoulli},
  volume={31},
  number={1},
  pages={537--560},
  year={2025},
  publisher={Bernoulli Society for Mathematical Statistics and Probability},
  url={https://doi.org/10.3150/24-BEJ1738}
}

@book{hardy1952inequalities,
  title={Inequalities},
  author={Hardy, Godfrey Harold and Littlewood, John Edensor and P{\'o}lya, George},
  year={1952},
  publisher={Cambridge University Press},
  url={https://openlibrary.org/books/OL7715912M/Inequalities}
}

@article{kupferman2019piola,
  author  = {Kupferman, Raz and Shachar, Asaf},
  title   = {A geometric perspective on the {Piola} identity in {Riemannian} settings},
  journal = {Journal of Geometric Mechanics},
  volume  = {11},
  number  = {1},
  pages   = {59--76},
  year    = {2019},
  doi     = {10.3934/jgm.2019004}
}

@article{roy2026low,
  title={Low-Dimensional Adaptation of Rectified Flow: A Diffusion and Stochastic Localization Perspective},
  author={Roy, Saptarshi and Rinaldo, Alessandro and Sarkar, Purnamrita},
  journal={arXiv preprint arXiv:2601.15500},
  year={2026},
  url = {https://arxiv.org/abs/2601.15500}
}

@article{kumar2026conditional,
  title={Conditional Flow Matching for Transport Between Markov Processes},
  author={Kumar, Syamantak and Nagaraj, Dheeraj and Roy, Saptarshi and Sarkar, Purnamrita},
  journal={arXiv preprint arXiv:2610.07229},
  year={2026},
  url = {https://arxiv.org/abs/2610.07229}
}

@article{low_kumar2026flow,
  title={Flow matching is adaptive to manifold structures},
  author={Kumar, Shivam and Wang, Yixin and Lin, Lizhen},
  journal={arXiv preprint arXiv:2602.22486},
  year={2026},
  url = {https://arxiv.org/pdf/2602.22486}
}

@inproceedings{low_boffi2025shallow,
  title={Shallow diffusion networks provably learn hidden low-dimensional structure},
  author={Boffi, Nicholas and Jacot, Arthur and Tu, Stephen and Ziemann, Ingvar},
  booktitle={International Conference on Learning Representations},
  volume={2025},
  pages={52889--52923},
  year={2025},
  url = {https://openreview.net/forum?id=KlxK4ncqWZ}
}

@inproceedings{guan2026mirror,
  title={Mirror flow matching with heavy-tailed priors for generative modeling on convex domains},
  author={Guan, Yunrui and Balasubramanian, Krishna and Ma, Shiqian},
  booktitle={International Conference on Learning Representations},
  volume={2026},
  pages={130098--130124},
  year={2026},
  url = {https://openreview.net/forum?id=dZKl7uc0XQ}
}

@article{wang2026computational,
  title={Computational and Statistical Guarantees of the$\backslash$textit $\{$c$\}$-Rectified flow},
  author={Wang, Leda and Xu, Zhehao and Liu, Qiang and Zhou, Harrison H},
  journal={arXiv preprint arXiv:2608.02487},
  year={2026},
  url={https://arxiv.org/abs/2608.02487}
}

@InProceedings{alignment_potaptchik2026meta,
  title = 	 {Meta Flow Maps enable scalable reward alignment},
  author =       {Potaptchik, Peter and Saravanan, Adhi and Mammadov, Abbas and Prat, Alvaro and Albergo, Michael Samuel and Teh, Yee Whye},
  booktitle = 	 {Proceedings of the 43rd International Conference on Machine Learning},
  pages = 	 {99255--99295},
  year = 	 {2026},
  editor = 	 {Zhang, Tong and Dudik, Miroslav and Jaggi, Martin and Agarwal, Alekh and Li, Sharon and Schuurmans, Dale and Zhu, Jerry and Berkenkamp, Felix and Dong, Hanze and Bietti, Alberto},
  volume = 	 {306},
  series = 	 {Proceedings of Machine Learning Research},
  month = 	 {06--11 Jul},
  publisher =    {PMLR},
  url = 	 {https://proceedings.mlr.press/v306/potaptchik26b.html}
}

@inproceedings{alignment_fan2025online,
  title={Online reward-weighted fine-tuning of flow matching with wasserstein regularization},
  author={Fan, Jiajun and Shen, Shuaike and Cheng, Chaoran and Chen, Yuxin and Liang, Chumeng and Liu, Ge},
  booktitle={The Thirteenth International Conference on Learning Representations},
  year={2025},
  url = {https://openreview.net/forum?id=2IoFFexvuw}
}

@article{dalalyan2017theoretical,
  author  = {Dalalyan, Arnak S.},
  title   = {Theoretical Guarantees for Approximate Sampling
             from Smooth and Log-Concave Densities},
  journal = {Journal of the Royal Statistical Society:
             Series B (Statistical Methodology)},
  year    = {2017},
  volume  = {79},
  number  = {3},
  pages   = {651--676},
  url     = {https://doi.org/10.1111/rssb.12183}
}

@article{durmus2019analysis,
  author  = {Durmus, Alain and Majewski, Szymon and
             Miasojedow, B{\l}a{\.z}ej},
  title   = {Analysis of {Langevin Monte Carlo}
             via Convex Optimization},
  journal = {Journal of Machine Learning Research},
  year    = {2019},
  volume  = {20},
  number  = {73},
  pages   = {1--46},
  url     = {https://www.jmlr.org/papers/v20/18-173.html}
}

@book{mclachlan2000finite,
  author    = {McLachlan, Geoffrey and Peel, David},
  title     = {Finite Mixture Models},
  publisher = {Wiley},
  year      = {2000},
  url       = {https://doi.org/10.1002/0471721182}
}

@book{bishop2006pattern,
  author    = {Bishop, Christopher M.},
  title     = {Pattern Recognition and Machine Learning},
  publisher = {Springer},
  year      = {2006},
  url       = {https://link.springer.com/book/9780387310732}
}

@article{nguyen2020approximation,
  author  = {Nguyen, T. Tin and Nguyen, Hien D. and
             Chamroukhi, Faicel and McLachlan, Geoffrey J.},
  title   = {Approximation by Finite Mixtures of Continuous
             Density Functions That Vanish at Infinity},
  journal = {Cogent Mathematics \& Statistics},
  year    = {2020},
  volume  = {7},
  number  = {1},
  pages   = {1750861},
  url     = {https://doi.org/10.1080/25742558.2020.1750861}
}

@inproceedings{lee2018beyond,
  author    = {Lee, Holden and Risteski, Andrej and Ge, Rong},
  title     = {Beyond Log-concavity: Provable Guarantees for Sampling
               Multi-modal Distributions using Simulated Tempering
               {Langevin Monte Carlo}},
  booktitle = {Advances in Neural Information Processing Systems},
  volume    = {31},
  year      = {2018},
  url       = {https://proceedings.neurips.cc/paper/2018/hash/c6ede20e6f597abf4b3f6bb30cee16c7-Abstract.html}
}

@inproceedings{chen2025learning,
  author    = {Chen, Sitan and Kontonis, Vasilis and Shah, Kulin},
  title     = {Learning General {Gaussian} Mixtures
               with Efficient Score Matching},
  booktitle = {Proceedings of the Thirty-Eighth Conference
               on Learning Theory},
  series    = {Proceedings of Machine Learning Research},
  volume    = {291},
  pages     = {1029--1090},
  year      = {2025},
  publisher = {PMLR},
  url       = {https://proceedings.mlr.press/v291/chen25e.html}
}

@article{chewi2023entropic,
  author  = {Chewi, Sinho and Pooladian, Aram-Alexandre},
  title   = {An Entropic Generalization of {Caffarelli}'s
             Contraction Theorem via Covariance Inequalities},
  journal = {Comptes Rendus. Math{\'e}matique},
  year    = {2023},
  volume  = {361},
  pages   = {1471--1482},
  doi     = {10.5802/crmath.486},
  url     = {https://doi.org/10.5802/crmath.486}
}

@inproceedings{chen2026highaccuracy,
  author    = {Chen, Fan and Chewi, Sinho and
               Daskalakis, Constantinos and Rakhlin, Alexander},
  title     = {High-Accuracy Sampling for Diffusion Models
               and Log-Concave Distributions},
  booktitle = {Proceedings of the 43rd International Conference
               on Machine Learning},
  series    = {Proceedings of Machine Learning Research},
  volume    = {306},
  pages     = {13835--13869},
  year      = {2026},
  publisher = {PMLR},
  url       = {https://proceedings.mlr.press/v306/chen26p.html}
}

@inproceedings{gatmiry2026highaccuracy,
  author    = {Gatmiry, Khashayar and Chen, Sitan and Salim, Adil},
  title     = {High-Accuracy and Dimension-Free Sampling
               with Diffusions},
  booktitle = {Proceedings of the 43rd International Conference
               on Machine Learning},
  series    = {Proceedings of Machine Learning Research},
  volume    = {306},
  pages     = {34242--34274},
  year      = {2026},
  publisher = {PMLR},
  url       = {https://proceedings.mlr.press/v306/gatmiry26a.html}
}

@misc{chewi2026logconcave,
  author = {Chewi, Sinho},
  title  = {Log-Concave Sampling},
  year   = {2026},
  note   = {Book draft, version of August 20, 2026},
  url    = {https://chewisinho.github.io/main.pdf}
}

@inproceedings{wang2025rectified,
  title     = {Rectified Diffusion: Straightness Is Not
               Your Need in Rectified Flow},
  author    = {Wang, Fu-Yun and Yang, Ling and
               Huang, Zhaoyang and Wang, Mengdi and Li, Hongsheng},
  booktitle = {International Conference on Learning Representations},
  year      = {2025},
  url       = {https://proceedings.iclr.cc/paper_files/paper/2025/hash/4df9a5e6bad9e64ebcea453e031142bb-Abstract-Conference.html}
}

@inproceedings{lee2024improving,
  title     = {Improving the Training of Rectified Flows},
  author    = {Lee, Sangyun and Lin, Zinan and Fanti, Giulia},
  booktitle = {Advances in Neural Information Processing Systems},
  volume    = {37},
  year      = {2024},
  url       = {https://papers.neurips.cc/paper_files/paper/2024/hash/7343a5c976f8399880b695267f1f9e9f-Abstract-Conference.html}
}
